\documentclass[12pt,twoside,a4paper]{report}

\usepackage[T1]{fontenc}

\usepackage[DM,newLogo,final]{uaThesis}
\def\ThesisYear{2012}

\usepackage{hyperref}

\usepackage{a4wide}
\usepackage{amsmath,amssymb,amsthm}
\usepackage{makeidx}

\usepackage{cite}
\usepackage{fancyhdr}

\theoremstyle{theorem}
\newtheorem{theorem}{Theorem}
\newtheorem{cor}[theorem]{Corollary}
\newtheorem{lemma}[theorem]{Lemma}
\newtheorem{prop}[theorem]{Proposition}
\theoremstyle{definition}
\newtheorem{definition}[theorem]{Definition}
\newtheorem{example}[theorem]{Example}
\theoremstyle{remark}
\newtheorem{remark}[theorem]{Remark}

\def\MR#1{\href{http://www.ams.org/mathscinet-getitem?mr=#1}{MR#1}}

\def\I{\mathtt{i}}

\def\EXP#1.{e^{2\pi\I #1}}

\newcommand{\N}{\mathbb{N}}

\newcommand{\R}{\mathbb{R}}

\newcommand{\T}{\mathbb{T}}

\makeindex

\begin{document}

\fancyhf{}

\fancyhead[LO]{\slshape \rightmark}
\fancyhead[RE]{\slshape\leftmark} \fancyfoot[C]{\thepage}

\TitlePage
  \HEADER{\BAR\FIG{\includegraphics[height=60mm]{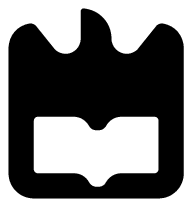}}}
         {\ThesisYear}
    \TITLE{Nuno Rafael de \newline Oliveira Bastos}
          {C\'{a}lculo Fraccional em Escalas Temporais\newline\newline
Fractional Calculus on Time Scales}
          \EndTitlePage
\titlepage\ \endtitlepage

\TitlePage
  \HEADER{\BAR\FIG{\begin{minipage}{50mm}
          ``Mathematical discoveries, small or great, are never born of
spontaneous generation. They always presuppose a soil seeded with
preliminary knowledge and well prepared by labour, both conscious
and subconscious.''
             \begin{flushright}
            --- Henri Poincar\'{e}
             \end{flushright}
            \end{minipage}}}
         {\ThesisYear}
    \TITLE{Nuno Rafael de \newline Oliveira Bastos}
          {C\'{a}lculo Fraccional em Escalas Temporais\newline\newline
Fractional Calculus on Time Scales}
\EndTitlePage
\titlepage\ \endtitlepage

\TitlePage
  \HEADER{}{\ThesisYear}
  \TITLE{Nuno Rafael de \newline Oliveira Bastos}
        {C\'{a}lculo Fraccional em Escalas Temporais\\
Fractional Calculus on Time Scales}
  \vspace*{15mm}
  \TEXT{}
       {Tese apresentada \`a Universidade de Aveiro para cumprimento dos requisitos
        necess\'arios \`a obten\c c\~ao do grau de Doutor em Matem\'{a}tica, Programa Doutoral 
        em Matem\'{a}tica e Aplica\c{c}\~{o}es -- PDMA 2077-2011 -- da Universidade de Aveiro 
        e Universidade do Minho, realizada sob a orienta\c c\~ao
        cient\'\i fica de Delfim Fernando Marado Torres, Professor Associado com Agrega\c{c}\~{a}o 
        do Departamento de Matem\'{a}tica da Universidade de Aveiro.\\\\\\
Thesis submitted to the University of Aveiro in fulfilment of the requirements for the degree 
of Doctor of Philosophy in Mathematics, Doctoral Programme in Mathematics 
and Applications -- PDMA 2007-2011 -- of the University of Aveiro and University of Minho, 
under the supervision of Professor Delfim Fernando Marado Torres, Associate Professor 
with tenure and Habilitation of the University of Aveiro.}
\EndTitlePage
\titlepage\ \endtitlepage

\TitlePage
  \vspace*{55mm}
  \TEXT{\textbf{o j\'uri~/~the jury\newline}}
       {}
  \TEXT{presidente~/~president}
       {\textbf{Prof. Doutor Jos\'e Carlos da Silva Neves}\newline {\small
        Professor Catedr\'atico da Universidade de Aveiro (por delega\c c\~ao do Reitor da
        Universidade de Aveiro)}}
  \vspace*{5mm}
  \TEXT{vogais~/~examiners committee}
       {\textbf{Prof. Doutor Manuel Duarte Ortigueira
}\newline {\small
        Professor Associado com Agrega\c c\~ao da Faculdade de Ci\^encias e Tecnologia da Universidade Nova de Lisboa}}
  \vspace*{5mm}
  \TEXT{}
       {\textbf{Prof. Doutor Delfim Fernando Marado Torres}\newline {\small
        Professor Associado com Agrega\c c\~ao da Universidade de Aveiro (orientador)}}
  \vspace*{5mm}
  \TEXT{}
       {\textbf{Prof. Doutor Filipe Artur Pacheco Neves Carteado Mena
}\newline {\small
        Professor Associado da Escola de Ci\^encias da Universidade do Minho}}
\vspace*{5mm}
  \TEXT{}
       {\textbf{Prof. Doutora Agnieszka Barbara Malinowska
}\newline {\small
        Professora Auxiliar da Bia\l ystok University of Technology, Pol\'onia}}
\vspace*{5mm}
  \TEXT{}
       {\textbf{Prof. Doutor Ricardo Miguel Moreira de Almeida
}\newline {\small
        Professor Auxiliar da Universidade de Aveiro}}
\EndTitlePage
\titlepage\ \endtitlepage

\TitlePage
  \vspace*{55mm}
  \TEXT{\textbf{agradecimentos~/\newline acknowledgements}}
       {I would like to acknowledge my supervisor Delfim F. M. Torres for accepting me as his student. 
       I could not imagine having a better advisor because without his knowledge, 
       invaluable guidance, advice, trust, encouragement and endless patience, 
       I would never have finished this thesis.\\
\\
I wish to express a special gratitude for the financial support 
of the \emph{Polytechnic Institute of Viseu} and
the \emph{Portuguese Foundation for Science and Technology} (FCT), 
through the ``Programa de apoio \`{a} forma\c{c}\~{a}o avan\c{c}ada de docentes 
do Ensino Superior Polit\'{e}cnico'', PhD fellowship SFRH/PROTEC/49730/2009, 
without which this work wouldn't be possible.\\
\\
I'm thankful to the people with whom I interacted at the Department 
of Mathematics of University of Aveiro and at the Department of Mathematics 
of Faculty of Computer Science, Bia\l ystok University of Technology, 
for providing a kind atmosphere and to contribute, in some sense, to this work.
\\\\
I also thank to my friends for their untiring support and to my department colleagues 
for their encouraging comments during these years, 
especially when this thesis seemed stopped.
\\\\
I wish to thank to Vanda for her patience, never-ending support and for never 
allowing me feeling alone in this hard and long way, providing, on that sense, 
the perfect conditions to work on this thesis.
\\\\
I also wish to express my deepest gratitude to my family for their invaluable 
support and for understanding my missing in several moments of this work.\\\\
Finally, I wish to dedicate this thesis to my parents, to my brother and to Vanda.
}
\EndTitlePage
\titlepage\ \endtitlepage

\TitlePage
  \vspace*{55mm}
\TEXT{\textbf{Palavras-chave}}
       {Problemas variacionais fraccionais em tempo discreto, operadores fraccionais discretos 
       do tipo de Riemann--Liouville, soma por partes fraccional, equa\c c\~oes de Euler--Lagrange, 
       condi\c c\~ao necess\'aria de optimalidade do tipo de Legendre, escalas temporais, 
       sistema de computa\c c\~ao alg\'ebrico \emph{Maxima}. }
\TEXT{}\\
\TEXT{\textbf{Resumo}}
       {Introduzimos um c\'alculo das varia\c{c}\~oes fraccionais nas escalas temporais 
       $\mathbb{Z}$ e $\left(h\mathbb{Z}\right)_a$. Estabelecemos a primeira e a segunda 
       condi\c{c}\~ao necess\'aria  de optimalidade. S\~ao dados alguns exemplos num\'ericos 
       que ilustram o uso quer da nova condi\c{c}\~ao de Euler--Lagrange quer da nova condi\c{c}\~ao 
       do tipo de Legendre. Introduzimos tamb\'em novas defini\c{c}\~oes de derivada fraccional 
       e de integral fraccional numa escala temporal com recurso \`a transformada 
       inversa generalizada de Laplace. }
  \TEXT{}
       {}
\EndTitlePage
\titlepage\ \endtitlepage

\TitlePage
  \vspace*{55mm}
\TEXT{\textbf{Keywords}}
       {Fractional discrete-time variational problems, discrete analogues of Riemann--Liouville 
       fractional-order operators, fractional formula for summation by parts, Euler--Lagrange equations, 
       Legendre type necessary optimality condition, time scales, \emph{Maxima} computer algebra system.}
\TEXT{}\\
 \TEXT{\textbf{Abstract}}
       {We introduce a discrete-time fractional calculus of variations on
the time scales $\mathbb{Z}$ and $(h\mathbb{Z})_a$. First and second
order necessary optimality conditions are established. Some
numerical examples illustrating the use of the new Euler--Lagrange
and Legendre type conditions are given. We also give new definitions of fractional
derivatives and integrals on time scales via the inverse generalized
Laplace transform.
}  \EndTitlePage
\titlepage\ \endtitlepage

\pagenumbering{roman}

\tableofcontents

\cleardoublepage

\pagenumbering{arabic}


\chapter*{Introduction}\markboth{INTRODUCTION}{INTRODUCTION}

The main goal of this thesis is to develop a more general fractional calculus on time scales.\\
In the first year of my PhD Doctoral Programme I followed several one-semester courses in distinct 
fields of mathematics. One of the one-semester courses was called Research Lab where five different 
subjects were covered. One of those subjects was calculus of variations on time scales. 
Time scales theory looked to me a new and wonderful subject. As the doctoral programme 
PDMA Aveiro--Minho is backed up by two research units, not only Professors of the first semester 
but any researcher that belongs to any of these R\&D units is invited to present, 
at the end of the first semester, research topics for PhD thesis by giving seminars 
to interested students. The author's supervisor was one of the researchers that 
was available to be advisor if any student show interest in the  theme 
\emph{Fractional Calculus on Time Scales} -- the title of this thesis.\\
The main idea was to connect in one theory two subjects that were, 
and still are, subject to strong research and development.
Our contribute on this new area, until now, appears in part II of this thesis. 
On one hand we develop a discrete fractional calculus of variations 
for the time scale $\mathbb{Z}$ (Chapter~\ref{chap3}) and for the time scale 
$(h\mathbb{Z})_a$ (Chapter~\ref{chap4}). One the other hand we give new definitions 
for differintegral fractional calculus on a time scale using 
a Laplace transform approach (Chapter~\ref{chap5}).

Nowadays fractional differentiation (differentiation of an arbitrary order)
plays an important role in various fields: physics (classic and quantum mechanics,
thermodynamics, etc.), chemistry, biology, economics, engineering,
signal and image processing, and control theory
\cite{Agrawal:2004b,CD:Hilfer:2000,book:Kilbas,CD:Klimek:2002,Miller1,oldman,Ortigueira,Podlubny,Samko,TenreiroMachado}. 
It is a subject as old as Calculus itself but much in progress \cite{livro:ortigueira}. 
The origin of fractional calculus goes back three centuries, when in 1695 L'Hopital 
asked Leibniz what should be the meaning of a derivative of order $1/2$. Leibniz's response: 
``An apparent paradox, from which one day useful consequences will be drawn''. 
After that episode, which most authors consider the born of fractional calculus, 
several famous mathematicians contributed to the development of Fractional 
Calculus~\cite{CD:Hilfer:2000,Miller1,Samko}: Abel, Fourier, Liouville,~Gr\"{u}nwald, 
Letnikov, Caputo, Riemann, Riesz, just to mention a few names.
In the last decades, considerable research has been done in fractional calculus.
This is particularly true in the area of the calculus of variations, which is being 
subject to intense investigations during the last few years 
\cite{B:08,Ozlem,R:N:H:M:B:07,R:T:M:B:07,Tat:Mal:Torres:2012,Almeida:Pooseh:Torres}. 
F.~Riewe \cite{Riewe,CD:Riewe:1997} obtained a version of the Euler-Lagrange equations 
for problems of the calculus of variations with fractional derivatives, 
that combines the conservative and non-conservative cases.
In fractional calculus of variations our main interest is the study 
of necessary optimality conditions for fractional problems of the calculus of variations.
The study of fractional problems of the calculus of variations and respective 
Euler--Lagrange equations is a fairly recent issue -- see 
\cite{agr0,agr1,agr3,agr2,RicDel,MyID:182,Ozlem,El-Nabulsi1,El-Nabulsi2,gastao:delfim,gasta1,M:Baleanu}
and references therein -- and include only the continuous case.
It is well known that discrete analogues of differential equations can be very useful in
applications \cite{B:J:06,J:B:07,book:DCV} and that fractional Euler-Lagrange differential
equations are extremely difficult to solve, being necessary to discretize them \cite{agr2,Ozlem}.
Therefore, we consider pertinent to develop a fractional discrete-time theory of the
calculus of variations in a different time scale than $\mathbb{R}$.
We dedicate two chapters to that: one for the time scale $\mathbb{Z}$ and another for the time scale $(h\mathbb{Z})_a$.
Applications of fractional calculus of variations include fractional variational principles in mechanics and physics, 
quantization, control theory, and description of conservative,
nonconservative, and constrained systems \cite{B:08,B:M:R:08,B:M:08,R:T:M:B:07}.
Roughly speaking, the classical calculus of variations and optimal control are extended
by substituting the usual derivatives of integer order
by different kinds of fractional (non-integer) derivatives.
It is important to note that the passage from the integer/classical differential calculus
to the fractional one is not unique because we have at our disposal
different notions of fractional derivatives.
This is, as argued in \cite{B:08,R:N:H:M:B:07},
an interesting and advantage feature of the area.
Most part of investigations in the fractional variational calculus are
based on the replacement of the classical derivatives by fractional derivatives
in the sense of Riemann--Liouville, Caputo and Riesz
\cite{agr0,MyID:182,B:08,MyID:149,Almeida:Torres:2011,Tatiana:Torres:2011}.
Independently of the chosen fractional derivatives, one obtains, when the fractional
order of differentiation tends to an integer order, the usual problems and results
of the calculus of variations.

A time scale is any nonempty closed subset of the real line.
The theory of time scales is a fairly new area of research. It was introduced in Stefan Hilger's 
1988 Ph.D. thesis~\cite{Hilger2} and subsequent landmark papers \cite{Hilger'97,Hilger},
as a way to unify the seemingly disparate fields of discrete dynamical systems (\textrm{i.e.}, 
difference equations) and continuous dynamical systems (\textrm{i.e.}, differential equations).
His dissertation referred to such unification as ``Calculus on Measure Chains'' 
\cite{Aul:Hil,book:Lak}. Today it is better known as the time scale calculus.
Since the nineties of XX century, the study of dynamic equations on time scales received 
a lot of attention (see, \textrm{e.g.}, \cite{Agarwal&Bohner&O'Regan&Peterson'02,livro:2001,livro:2003}).
In 1997, the German mathematician Martin Bohner came across time scale calculus by
chance, when he took up a position at the National University of
Singapore. On the way from Singapore airport, a colleague, Ravi
Agarwal, mentioned that time scale calculus might be the key to the problems 
that Bohner was investigating at that time. After that episode, 
time scale calculus became one of its main areas of research.
To the reader who wants a gentle overview of time scales we advise 
to begin with \cite{Spedding}, that was written in a didactic way 
and at the same time points some possible applications (e.g. in biology). 
Here we are interested in the calculus of variations on time scales
\cite{Natalia:Delfim:2011,Delfim:alone:2010}. Our goal is to connect 
the theories of calculus of variations on time scales and fractional calculus.

This thesis is divided in two major parts. The first part has two chapters 
in which we provide some preliminaries  on fractional calculus and time scales calculus, 
respectively. The second part is splitted in three chapters where we present our original work. 
In Chapter~\ref{chap3} we introduce the fractional calculus of variations 
on the time scale $\mathbb{Z}$. The main results of this chapter are a first-order 
necessary optimality condition (Euler-Lagrange equation) and a second-order necessary 
optimality condition (Legendre inequality). In Chapter~\ref{chap4} we introduce 
a fractional factorial function that allow us to define left and right fractional 
derivatives and to develop further the two necessary optimality conditions 
of Chapter~\ref{chap3}. In Chapter~\ref{chap4} the time scale is $h\mathbb{Z}_a$.
We believe that our results open some possible doors to research, also in the continuous case,
when $h\rightarrow 0$. Chapter~\ref{chap5} is devoted to a new approach 
to define fractional derivatives and fractional integrals in an arbitrary time scale,
by using the Laplace transform of time scales as support.
Finally, we write our conclusions in Chapter~\ref{conclusao} 
as well some future research directions.


\clearpage{\thispagestyle{empty}\cleardoublepage}

\part{Synthesis}

\clearpage{\thispagestyle{empty}\cleardoublepage}


\chapter{Fractional Calculus}
\label{chap1}

\begin{flushright}
``The fractional calculus is the calculus of the XXI-st century''\\
K. Nishimoto (1989)
\end{flushright}

The theory of discrete fractional calculus
is in its infancy~\cite{Atici0,Atici,Miller}.
In contrast, the theory of continuous fractional
calculus is much more developed \cite{Samko}.
The fractional discrete theory
has its foundations in the pioneering work of
Kuttner in 1957, where it appears the first definition of
fractional order differences.
In Sections \ref{discretefract} and \ref{continuosfract} some
definitions of fractional discrete and continuous operators are
given, respectively.

\section{Discrete fractional calculus}\label{discretefract}

In 1957 \cite{Kuttner} Kuttner defined, for any sequence of complex
numbers, $\{a_n\}$, the $s$-th order difference as
\begin{equation}\label{Kuttner:Def}
\Delta^s a_n=\sum_{m=0}^{\infty}\binom{-s-1+m}{m}a_{n+m},
\end{equation}
where
\begin{equation}\label{binom1}
\binom{t}{m}=\frac{t(t-1)\ldots(t-m+1)}{m!}~.
\end{equation}

In \cite{Kuttner} Kuttner also remarks that
$\displaystyle{\binom{t}{m}}$ means $0$ when $m$ is negative and
also when $t-m$ is a negative integer but $t$ is not a negative
integer. Clearly, \eqref{Kuttner:Def} only makes sense when the
series converges.
\\\\
In 1974, Diaz and Osler \cite{Diaz} gave the following definition
for a fractional difference of order $\nu$:
$$
\Delta^\nu f(x)=\sum_{k=0}^{\infty}(-1)^k\binom{\nu}{k}f(x+\nu-k),
$$
\begin{equation}\label{binom2}
\binom{\nu}{k}=\frac{\Gamma(\nu+1)}{\Gamma(\nu-k+1)k!}
\end{equation} where $\nu$ is any real or complex number.
\\\\
The above definition uses the usual well-known gamma function which
is defined by

\begin{equation}\label{Gamma}
\Gamma(\nu):=\int_0^{\infty}t^{\nu-1}e^{-t}dt,~\nu\in
\mathbb{C}\backslash(-\mathbb{N}_0)~.
\end{equation}

The gamma function was first introduced by the Swiss mathematician
Leonard Euler in his goal to generalize the factorial to non integer
values.

Throughout this thesis we use some of its most important
properties:
\begin{eqnarray}
\Gamma(1)&=& 1, \\
\Gamma(\nu+1)&=&\nu\Gamma(\nu)~~\text{for}~\nu\in
\mathbb{C}\backslash(-\mathbb{N}_0)\label{Z::naosei9},\\
\Gamma(x+1)&=&x!,~~\text{for}~x\in \mathbb{N}_0~.
\end{eqnarray}

\begin{remark}
Using the properties of gamma function it is easily proved
that formulas (\ref{binom1}) and (\ref{binom2}) coincide for $\nu$ an integer.
\end{remark}

We begin by introducing some notation used throughout. Let $a$ be an
arbitrary real number and $b = a + k$ for a certain $k
\in\mathbb{N}$ with $k \ge 2$. Let $\mathbb{T}= \{a, a + 1, \ldots,
b\}$. According with \cite{livro:2001}, we define the factorial function
$$
t^{(n)} = t(t-1)(t-2)\ldots(t-n+1), \quad n\in\mathbb{N}
$$
and $t^{(0)}=0$.\\
Extending the above definition from
an integer $n$ to an arbitrary real number $\alpha$, we have
\begin{equation}\label{factorial:fracti}
t^{(\alpha)}=\frac{\Gamma(t+1)}{\Gamma(t+1-\alpha)},
\end{equation}
where $\Gamma$ is the Euler gamma function.
In \cite{Miller} Miller and Ross define a fractional sum of order
$\nu>0$ \emph{via} the solution of a linear difference equation.
Namely, they present it as follows:
\begin{definition}\label{Miller:fract}
\begin{equation}\label{Z::naosei8}
\Delta^{-\nu}f(t)=\frac{1}{\Gamma(\nu)}\sum_{s=a}^{t-\nu}(t-\sigma(s))^{(\nu-1)}f(s).
\end{equation}
Here $f$ is defined for $s=a \mod (1)$ and
$\Delta^{-\nu}f$ is defined for $t=(a+\nu) \mod (1)$.
\end{definition}
This was done in analogy with the Riemann--Liouville fractional
integral of order $\nu>0$ (\textrm{cf.} formula~\eqref{left::RLFI}) 
which can be obtained \emph{via} the
solution of a linear differential equation \cite{Miller,Miller1}.
Some basic properties of the sum in \eqref{Z::naosei8} were obtained
in \cite{Miller}. Although there are other definitions of fractional
difference operators, throughout this thesis, we follow mostly,
the spirit of Miller and Ross, Atici and Eloe~\cite{Atici,Miller}.

\section{Continuous fractional calculus}
\label{continuosfract}

In the literature there are several definitions of fractional
derivatives and fractional integrals (simply called differintegrals)
like Riemann--Liouville, Caputo, Riesz and Hadamard.
Throughout this thesis only Riemann--Liouville and Caputo
definitions are used.\\
By historical precedent over Caputo definition, we will start with
the Riemann-Liouville definition.
Let $[a,b]$ be a finite interval and $\alpha\in \mathbb{R}^+$. The
\emph{left}\index{Riemann--Liouville!left fractional integral} 
and the \emph{right}\index{Riemann--Liouville!right fractional integral} 
Riemann--Liouville fractional integrals (RLFI) of order $\alpha$ of a function $f$
are defined, respectively, by:
\begin{equation}\label{left::RLFI}
{}_a I_x^\alpha
f(x):=\frac{1}{\Gamma(\alpha)}\int_a^x(x-t)^{\alpha-1}f(t)dt,~~~x>a,
\end{equation}
and
\begin{equation}\label{right::RLFI}
{}_x I_b^\alpha
f(x):=\frac{1}{\Gamma(\alpha)}\int_x^b(t-x)^{\alpha-1}f(t)dt,~~~x<b,
\end{equation}
where $\Gamma(\alpha)$ is the Gamma function~\eqref{Gamma}.

\begin{remark}
If $f$ is a continuous function ${}_a I_x^0 f = {}_x I_b^0 =f$.
\end{remark}

\begin{remark}
The semigroup property of Riemann--Liouville fractional operators
are given by

\begin{equation}\label{right::RLSP}
({}_a I_x^\alpha {}_a I_x^\beta f)(x)={}_a
I_x^{\alpha+\beta}f(x),~~x>a,~\alpha>0,~\beta>0.
\end{equation}

If $f$ is a continuous function ${}_a I_x^0 f = {}_x I_b^0 =f$.
\end{remark}

Let $f$ be a function, $\alpha\in \mathbb{R}_0^{+}$ and $n=[\alpha]+1$,
where $[\alpha]$ means the integer part of $\alpha$. 
The \emph{left}\index{Riemann--Liouville!left fractional derivative} 
and the \emph{right}\index{Riemann--Liouville!right fractional derivative} 
Riemann--Liouville fractional derivatives (RLFD) of order
$\alpha$ of $f$ are defined, respectively, by
\begin{equation}\label{left::RLFD}
    {}_a D_x^\alpha
    f(x):=\frac{1}{\Gamma(n-\alpha)}\left(\frac{d}{dx}\right)^n
    \int_a^x (x-t)^{-\alpha+n-1}f(t)dt=\left(\frac{d}{dx}\right)^n
    {}_a I_x^{n-\alpha} f(x),~~x>a
\end{equation}
and
\begin{equation}\label{right::RLFD}
    {}_x D_b^\alpha
    f(x):=\frac{1}{\Gamma(n-\alpha)}\left(-\frac{d}{dx}\right)^n
    \int_x^b (t-x)^{-\alpha+n-1}f(t)dt=\left(-\frac{d}{dx}\right)^n
    {}_x I_b^{n-\alpha} f(x),~~x<b)~.
\end{equation}
Let $AC([a,b])$ represent the space of absolutely continuous
functions on $[a,b]$. For $n\in \mathbb{N}$ we denote by $AC^n[a,b]$
the space of functions $f$ which have continuous derivatives up to
order $n-1$ on $[a,b]$ such that $f^{(n-1)}\in AC[a,b]$. In
particular, $AC^1[a,b]=AC[a,b]$.

The \emph{left}\index{Caputo!left fractional derivative} and the 
\emph{right}\index{Caputo!right fractional derivative} Caputo 
fractional derivatives (CFD) of order $\alpha\in \mathbb{R}_0^{+}$ 
of $f\in AC^n([a,b])$ are defined, respectively, by:

\begin{equation}\label{left::CFD}
    {}_a^C D_x^\alpha
    f(x):={}_a I_x^{n-\alpha} \frac{d^n}{dx^n}f(x),~~x>a
\end{equation}
and
\begin{equation}\label{right::CFD}
    {}_x^C D_b^\alpha
    f(x):=(-1)^n {}_x I_b^{n-\alpha} \frac{d^n}{dx^n}f(x),~~x<b,
\end{equation}
where $n=\left[\alpha\right]+1$.

\begin{remark}
The Riemann-Liouville approach requires the initial conditions
for differential equations in terms of non-integer derivatives,
which are very difficult to be interpreted from the physical point
of view~\cite{Dorota:Torres:2011}, whereas the Caputo approach uses 
integer-order initial conditions that are more easy to get in real 
world problems~\cite{Dorota:Torres:2010}. Caputo approach it's also 
preferable when we need fractional derivatives of constants to be zero.
The Caputo fractional derivative is zero for a constant while the
Riemann--Liouville derivative is not.
\end{remark}

\begin{remark}
If $f(a)=f'(a)=\cdots=f^{(n-1)}(a)=0$, then both Riemann--Liouville
and Caputo derivatives coincide. In particular, for $\alpha\in
(0,1)$ and $f(a)=0$ one has ${}_a^C D^{\alpha}_{x}f(t)={}_a
D^{\alpha}_{x}f(t)$.
\end{remark}

\clearpage{\thispagestyle{empty}\cleardoublepage}


\chapter{Time Scales}
\label{chap2}

In this chapter we recall some basic results
on time scales  (see Section~\ref{TS:defini}) that we use in the sequel. 
In Section~\ref{sec::cv} we review some results 
of the calculus of variations on time scales.

\section{Basic definitions}
\label{TS:defini}

\begin{definition}
A \emph{time scale}\index{Time Scale} is an arbitrary nonempty
closed subset of $\mathbb{R}$ and is denoted by  $\mathbb{T}$.
\end{definition}

\begin{example}
Here we just give some examples of sets that are time scales and others that
are not.

\begin{itemize}
  \item The set $\mathbb{R}$ is a time scale;
  \item The set $\mathbb{Z}$ is a time scale;
  \item The Cantor set is a time scale;
  \item The set $\mathbb{C}$ is not a time scale;
  \item The set $\mathbb{Q}$ is not a time scale.
\end{itemize}
\end{example}

In applications, the quantum calculus that has, as base, the set of powers 
of a given number $q$ is, in addition to classical continuous and purely 
discrete calculus, one of the most important time scales. Applications 
of this calculus appear, for example, in physics~\cite{book:Kac} 
and economics~\cite{Malin:Torres:Hahn}. The quantum derivative 
was introduced by Leonhard Euler and a fractional formulation 
of that derivative can be found in \cite{Ortigueira1,Ortigueira2}.

In this thesis attention will be given to \emph{time
scales} $\mathbb{Z}$, $\mathbb{R}$ and $(h\mathbb{Z})_a=\{a, a+h,
a+2h, \ldots\},~a\in \mathbb{R},~h>0$.
\\
A time scale of the form of a union of disjoint closed real intervals 
constitutes a good background for the study of population 
(of plants, insects, etc.) models. Such models appear, for example, 
when a plant population exhibits exponential growth during the months 
of Spring and Summer, and at the beginning of Autumn all plants die while
the seeds remain in the ground. Similar examples
concerning insect populations, where all the adults die
before the babies are born can be found in~\cite{livro:2001,Spedding}.
\\
The following operators of time scales theory are used, in literature and
throughout this thesis, several times:

\begin{definition}
The mapping $\sigma:\mathbb{T}\rightarrow\mathbb{T}$, defined by
$\sigma(t)=\inf{\{s\in\mathbb{T}:s>t\}}$ with
$\inf\emptyset=\sup\mathbb{T}$ (\textrm{i.e.}, $\sigma(M)=M$ if
$\mathbb{T}$ has a maximum $M$) is called the \emph{forward jump
operator}\index{Forward jump operator}. Accordingly, we define the
\emph{backward jump operator}\index{Backward jump operator}
$\rho:\mathbb{T}\rightarrow\mathbb{T}$ by
$\rho(t)=\sup{\{s\in\mathbb{T}:s<t\}}$ with
$\sup\emptyset=\inf\mathbb{T}$ (\textrm{i.e.}, $\rho(m)=m$ if
$\mathbb{T}$ has a minimum $m$). The symbol $\emptyset$ denotes the
empty set.
\end{definition}

The following classification of points is used within the theory:
a point $t\in\mathbb{T}$ is called
\emph{right-dense}\index{Points!right-dense},
\emph{right-scattered}\index{Points!right-scattered},
\emph{left-dense}\index{Points!left-dense} or
\emph{left-scattered}\index{Points!left-scattered} if $\sigma(t)=t$,
$\sigma(t)>t$, $\rho(t)=t$, $\rho(t)<t$, respectively.
A point $t$ is called \emph{isolated} if
$\rho(t)<t<\sigma(t)$ and \emph{dense} if $\rho(t)=t=\sigma(t)$.

\begin{definition}
A function $f:\mathbb{T}\rightarrow \mathbb{R}$ is called
\emph{regulated}\index{Regulated function} provided its right-sided limits exist (finite) at
all right-dense points in $\mathbb{T}$ and its left-sided limits
exist (finite) at all left-dense points in $\mathbb{T}$.
\end{definition}

\begin{example}
Let
$\displaystyle{\mathbb{T}=\left\{0\right\}\cup\left\{\frac{1}{n}:~n\in
\mathbb{N}\right\}\cup\{2\}\cup\left\{2-\frac{1}{n}:~n\in
\mathbb{N}\right\}}$ and define $f:\mathbb{T}\rightarrow \mathbb{R}$
by
\[f(t):=\left\{
\begin{array}{ll}
0 & \mbox{if $t\neq 2$};\\
t & \mbox{if $t=2$}.\end{array} \right.
~~.\]
Function $f$ is regulated. 
\end{example}

Now, let us define the sets $\mathbb{T}^{\kappa^n}$, inductively:
$$
\mathbb{T}^{\kappa^1}=\mathbb{T}^\kappa=\{t\in\mathbb{T}:\mbox{$t$
non-maximal or left-dense}\}
$$
and $\mathbb{T}^{\kappa^n}$=$(\mathbb{T}^{\kappa^{n-1}})^\kappa$, $n\geq 2$.

\begin{definition}
The \emph{forward graininess function}
$\mu:\mathbb{T}\rightarrow[0,\infty)$ is defined by
$\mu(t)=\sigma(t)-t$.
\end{definition}

\begin{remark}
Throughout the thesis, we refer to the forward graininess
function simply as the graininess function.
\end{remark}

\begin{definition}
The \emph{backward graininess function}
$\nu:\mathbb{T}\rightarrow[0,\infty)$ is defined by
$\nu(t)=t-\rho(t)$.
\end{definition}

\begin{example}
If $\mathbb{T}=\mathbb{R}$, then $\sigma(t)=\rho(t)=t$ and
$\mu(t)=\nu(t)=0$.
 If $\mathbb{T}=[-4,1]\cup \mathbb{N}$, then
\[ \sigma(t) = \left\{ \begin{array}{ll}
t & \mbox{if $t\in[-4,1)$};\\
t+1 & \mbox{otherwise},\end{array} \right. \] while
\[ \rho(t) = \left\{ \begin{array}{ll}
t & \mbox{if $t\in[-4,1]$};\\
t-1 & \mbox{otherwise}.\end{array} \right. \] Moreover,
\[ \mu(t) = \left\{ \begin{array}{ll}
0 & \mbox{if $t\in[-4,1)$};\\
1 & \mbox{otherwise},\end{array} \right. \] and
\[ \nu(t) = \left\{ \begin{array}{ll}
0 & \mbox{if $t\in[-4,1]$};\\
1 & \mbox{otherwise}.\end{array} \right. \]
\end{example}
\begin{example}\label{Z::sigma:mu}
If $\mathbb{T}=\mathbb{Z}$, then $\sigma(t)=t+1$, $\rho(t)=t-1$ and
$\mu(t)=\nu(t)=1$. If $\mathbb{T}=(h\mathbb{Z})_a$ then
$\sigma(t)=t+h,~\rho(t)=t-h$ and $\mu(t)=\nu(t)=h$.
\end{example}
For two points $a,b\in\mathbb{T}$, the time scale interval is
defined by
$$[a,b]_{\mathbb{T}}=\{t\in\mathbb{T}:a\leq t\leq b\}.$$

Throughout the thesis we will frequently write
$f^\sigma(t)=f(\sigma(t))$ and $f^\rho(t)=f(\rho(t))$.

Next results are related with differentiation on time scales.

\begin{definition}\label{def2}
We say that a function $f:\mathbb{T}\rightarrow\mathbb{R}$ is
\emph{$\Delta$-differentiable}\index{Function!$\Delta$-differentiable}
at $t\in\mathbb{T}^\kappa$ if there is a number $f^{\Delta}(t)$ such
that for all $\varepsilon>0$ there exists a neighborhood $U$ of $t$
 such that
$$|f^\sigma(t)-f(s)-f^{\Delta}(t)(\sigma(t)-s)|
\leq\varepsilon|\sigma(t)-s| \quad\mbox{ for all $s\in U$}.$$ We
call $f^{\Delta}(t)$ the \emph{$\Delta$-derivative} of $f$ at $t$.
\end{definition}

The $\Delta$-derivative of order $n\in\mathbb{N}$ of a function $f$
is defined by
$f^{\Delta^n}(t)=\left(f^{\Delta^{n-1}}(t)\right)^\Delta$,
$t\in\mathbb{T}^{\kappa^n}$, provided the right-hand side of the
equality exists, where $f^{\Delta^0}=f$.

Some basic properties will now be given for the $\Delta$-derivative.

\begin{theorem}\cite[Theorem~1.16]{livro:2001}
\label{teorema0} Assume $f:\mathbb{T}\rightarrow\mathbb{R}$ is a
function and let $t\in\mathbb{T}^\kappa$. Then we have the
following:
\begin{enumerate}
    \item If $f$ is $\Delta$-differentiable at $t$, then $f$ is
    continuous at $t$.
    \item If $f$ is continuous at $t$ and $t$ is right-scattered,
    then $f$ is $\Delta$-differentiable at $t$ with
    \begin{equation}\label{derdiscr}f^\Delta(t)=\frac{f^\sigma(t)-f(t)}{\mu(t)}.\end{equation}
    \item If $t$ is right-dense, then $f$ is $\Delta$-differentiable at
    $t$ if and only if the limit
    $$\lim_{s\rightarrow t}\frac{f(s)-f(t)}{s-t}$$
    exists as a finite number. In this case,
    $$f^\Delta(t)=\lim_{s\rightarrow t}\frac{f(s)-f(t)}{s-t}.$$
    \item If $f$ is $\Delta$-differentiable at $t$, then
    \begin{equation}\label{transfor}
    f^\sigma(t)=f(t)+\mu(t)f^\Delta(t).
    \end{equation}
\end{enumerate}
\end{theorem}
It is an immediate consequence of Theorem \ref{teorema0} that if
$\mathbb{T}=\mathbb{R}$, then the $\Delta$-derivative becomes the
classical one, i.e., $f^\Delta=f'$ while if $\mathbb{T}=\mathbb{Z}$,
the $\Delta$-derivative reduces to the \emph{forward difference}
$f^\Delta(t)=\Delta f(t)=f(t+1)-f(t)$.

\begin{theorem} \cite[Theorem~1.20]{livro:2001}
Assume $f,g:\mathbb{T}\rightarrow\mathbb{R}$ are
$\Delta$-differentiable at $t\in\mathbb{T}^\kappa$. Then:
\begin{enumerate}
    \item The sum $f+g:\mathbb{T}\rightarrow\mathbb{R}$ is $\Delta$-differentiable at
    $t$ and $(f+g)^\Delta(t)=f^\Delta(t)+g^\Delta(t)$.
    \item For any number $\xi\in\mathbb{R}$, $\xi f:\mathbb{T}\rightarrow\mathbb{R}$ is $\Delta$-differentiable at
    $t$ and $(\xi f)^\Delta(t)=\xi f^\Delta(t)$.
    \item The product $fg:\mathbb{T}\rightarrow\mathbb{R}$ is $\Delta$-differentiable at
    $t$ and
    \begin{equation}\label{produto}
    (fg)^\Delta(t)=f^\Delta(t)g(t)+f^\sigma(t)g^\Delta(t)=f(t)g^\Delta(t)+f^\Delta(t)g^\sigma(t).
    \end{equation}
\end{enumerate}
\end{theorem}

\begin{definition}
A function $f:\mathbb{T}\rightarrow\mathbb{R}$ is called
\emph{rd-continuous}\index{Function!rd-continuous} if it is
continuous at right-dense points and if the left-sided limit exists
at left-dense points.
\end{definition}

\begin{remark}
We denote the set of all rd-continuous functions by
C$_{\textrm{rd}}$ or C$_{\textrm{rd}}(\mathbb{T})$, and the set of
all $\Delta$-differentiable functions with rd-continuous derivative
by C$_{\textrm{rd}}^1$ or C$_{\textrm{rd}}^1(\mathbb{T})$.
\end{remark}

\begin{example}
Let $\mathbb{T}=\mathbb{N}_0 \cup \left\{1-\dfrac{1}{n}:~n\in
\mathbb{N}\right\}$ and define $f:\mathbb{T}\rightarrow \mathbb{R}$
by
\[f(t):=\left\{
\begin{array}{ll}
0 & \mbox{if $t\in\mathbb{N}$};\\
t & \mbox{otherwise}.\end{array} \right.
\]
It is easy to verify that $f$ is continuous at the isolated points.
So, the points that need our careful attention are the right-scattered
point $0$ and the left-dense point $1$.
The right-sided limit of $f$ at $0$ exist and is finite (is
equal to $0$). The left-sided limit of $f$ at $1$ exist and is
finite (is equal to $1$). We conclude that $f$ is
rd-continuous in $\mathbb{T}$.
\end{example}

We consider now some results about integration on time scales.

\begin{definition}
A function $F:\mathbb{T}\rightarrow\mathbb{R}$ is called an
\emph{antiderivative}\index{Antiderivative} of
$f:\mathbb{T}^\kappa\rightarrow\mathbb{R}$ provided
$F^\Delta(t)=f(t)$ for all $t\in\mathbb{T}^\kappa$.
\end{definition}

\begin{theorem} \cite[Theorem~1.74]{livro:2001}
Every rd-continuous function has an antiderivative.
\end{theorem}
Let $f:\mathbb{T}^\kappa:\rightarrow\mathbb{R}$ be a rd-continuous
function and let $F:\mathbb{T}\rightarrow\mathbb{R}$ be an
antiderivative of $f$. Then, the
\emph{$\Delta$-integral}\index{$\Delta$-integral} is defined by
$\int_a^b f(t)\Delta t=F(b)-F(a)$ for all $a,b\in\mathbb{T}$.

One can easily prove \cite[Theorem~1.79]{livro:2001} that, when
$\mathbb{T}=\mathbb{R}$ then $\int_a^bf(t)\Delta t=\int_a^bf(t)dt$,
being the right-hand side of the equality the usual Riemann
integral, and when $[a,b]\cap\mathbb{T}$ contains only isolated points, then
\begin{equation}\label{integration:isolated}
\int_a^b f(t)\Delta t=\left\{
                        \begin{array}{ll}
                          \displaystyle{\sum_{t=a}^{\rho(b)}\mu(t)f(t)} & \hbox{if $a<b$;} \\
                          0 & \hbox{if $a=b$;}\\
                          \displaystyle{-\sum_{t=a}^{\rho(b)}\mu(t)f(t)}, & \hbox{if $a>b$.}
                        \end{array}
                      \right.
\end{equation}

\begin{remark}
When $\mathbb{T}=\mathbb{Z}$ or $\mathbb{T}=h\mathbb{Z}$,
equation \eqref{integration:isolated} holds.
\end{remark}

The $\Delta-integral$ also satisfies
\begin{equation}
\label{sigma} \int_t^{\sigma(t)}f(\tau)\Delta\tau=\mu(t)f(t),\quad
t\in\mathbb{T}^\kappa.
\end{equation}

\begin{theorem}\cite[Theorem~1.77]{livro:2001}\label{teorema1}
Let $a,b,c\in\mathbb{T}$, $\xi\in\mathbb{R}$ and
$f,g\in$C$_{\textrm{rd}}(\mathbb{T}^\kappa)$. Then,
\begin{enumerate}
    \item $\int_a^b[f(t)+g(t)]\Delta t=\int_a^bf(t)\Delta t+\int_a^bg(t)\Delta
    t$.
    \item $\int_a^b(\xi f)(t)\Delta t=\xi\int_a^b f(t)\Delta t$.
    \item $\int_a^b f(t)\Delta t=-\int_b^a f(t)\Delta t$.
    \item $\int_a^bf(t)\Delta t=\int_a^cf(t)\Delta t+\int_c^bf(t)\Delta
    t$.
    \item $\int_a^af(t)\Delta t=0$.
    \item If $|f(t)|\leq g(t)$ on $[a,b]_{\mathbb{T}}^\kappa$,
    then
    $$\left|\int_a^bf(t)\Delta t\right|\leq \int_a^bg(t)\Delta t.$$
\end{enumerate}
\end{theorem}
We now present the integration by parts formulas for the
$\Delta$-integral:
\begin{lemma}(\textrm{cf.} \cite[Theorem~1.77]{livro:2001})
\label{integracao:partes} If $a,b\in\mathbb{T}$ and
$f,g\in C_{\textrm{rd}}^{1}$, then
\begin{equation}
 \int_{a}^{b}f^\sigma(t)g^{\Delta}(t)\Delta t
 =\left[(fg)(t)\right]_{t=a}^{t=b}-\int_{a}^{b}f^{\Delta}(t)g(t)\Delta
 t\label{partes1};
\end{equation}
\begin{equation}
\int_{a}^{b}f(t)g^{\Delta}(t)\Delta t
=\left[(fg)(t)\right]_{t=a}^{t=b}-\int_{a}^{b}f^{\Delta}(t)g^\sigma(t)\Delta
t.\label{partes2}
\end{equation}
\end{lemma}
\begin{remark}
For analogous results on $\nabla$-integrals the reader can consult,
e.g., \cite{livro:2003}.
\end{remark}

Some more definitions and results must be presented since we will
need them. We start defining functions that generalize
polynomials functions in $\mathbb{R}$ to the times scales calculus.
There are, at least, two ways of doing that:
\begin{definition}\label{polinomios}
We define the
\emph{polynomials on time
scales}\index{Polynomials on time scales}, $g_k,h_k:\mathbb{T}^2\rightarrow\mathbb{R}$ for $k\in
\mathbb{N}_0$ as follows:

$$g_0(t,s)=h_0(t,s)\equiv 1\quad \forall s,t\in\mathbb{T}$$

$$g_{k+1}(t,s)=\int_s^tg_k(\sigma(\tau),s)\Delta\tau\quad \forall s,t\in\mathbb{T},$$

$$h_{k+1}(t,s)=\int_s^th_k(\tau,s)\Delta\tau\quad \forall s,t\in\mathbb{T}.$$
\end{definition}

\begin{remark}
If we let $h^\Delta_k(t,s)$ denote, for each fixed $s\in \mathbb{T}$,
the derivative of $h_k(t,s)$ with respect to $t$, then we have the
following equalities:
\begin{enumerate}
  \item $h^\Delta_k(t,s)=h_{k-1}(t,s)\quad\text{for}\quad k\in \mathbb{N},~t\in \mathbb{T}^\kappa$;
  \item $g^\Delta_k(t,s)=g_{k-1}(\sigma(t),s)\quad\text{for}\quad k\in \mathbb{N},~t\in \mathbb{T}^\kappa$;
  \item $g_1(t,s)=h_1(t,s)=t-s\quad\text{for all}\quad s,t\in \mathbb{T}$.
\end{enumerate}
\end{remark}

\begin{remark}
It's not easy to find the explicit form of  $g_k$ and $h_k$ for a
generic time scale.
\end{remark}

To give the reader an idea what is the formula of such polynomials
for particular time scales, we give the following two examples:
 \begin{example}
For $\mathbb{T}=\mathbb{R}$ and $k\in \mathbb{N}_0$ we have:

\begin{equation}\label{poly:R}
    g_k(t,s)=h_k(t,s)=\frac{(t-s)^k}{k!}\quad\text{for all}\quad
s,t\in \mathbb{T}.
\end{equation}

\end{example}

\begin{example}
Consider $\mathbb{T}=\overline{q^\mathbb{Z}}=\left\{q^k~:~k\in
\mathbb{Z}\right\}\cup \left\{0\right\}$. For $k\in \mathbb{N}_0$
and $q>1$ we have:

\begin{equation}\label{poly:q}
    \displaystyle{h_k(t,s)=\prod_{i=0}^{k-1}\frac{t-q^i s}{\sum_{j=0}^{i}q^j}\quad\text{for all}\quad
s,t\in \mathbb{T}}.
\end{equation}

\end{example}

\section{Calculus of Variations}\label{sec::cv}

It is human nature to optimize. Typically, we try to maximize
profit, minimize cost, travel to a destination following the
smallest route or in the quickest time. The human propensity to
optimize and the necessity to develop systematic tools for that has
created a area of mathematics called optimization. Our main interest
in this chapter is to introduce some concepts necessary to solve
problems which involve finding extrema of an integral (functionals).
The area of mathematics that deals with these problems is a branch of
optimization called calculus of variations. Because this kind of
calculus involve functionals it was called, earlier, functional
calculus. Originally the name ``calculus of variations'' came from
representing a perturbed curve using a Taylor polynomial plus some
other term which was called the variation. Let us consider, for now,
the following basic (but fundamental) problem: seek a function $y\in C^1[a,b]$ such
that
\begin{equation}\label{Prob0}
\mathcal{L}[y(\cdot)]=\int_{a}^{b}L(t,y(t),y'(t))d t \longrightarrow
\min ,\quad y(a)=y_a,\quad y(b)=y_b,
\end{equation}
with $a,b,y_a,y_b\in\mathbb{R}$ and $L(t,u,v)$ satisfying some
smoothness properties.

\begin{remark}
Although we write \eqref{Prob0} as a minimization problem, we can
formulate it as a maximization problem using the fact
to minimize $\mathcal{L}[y(\cdot)]$ is the same as to maximize $-\mathcal{L}[y(\cdot)]$.
\end{remark}

We start by giving a simple application of the calculus of variations
on the real setting and then we refer to some results of the calculus of
variations on time scales, which includes as special cases the
classical calculus of variations ($\mathbb{T}=\mathbb{R}$) and the
discrete calculus of variations ($\mathbb{T}=\mathbb{Z}$).

\begin{example}\label{shortestdistance}
Given two distinct points $A=(a,y_1)$ and $B=(b,y_2)$ in the plane
$\mathbb{R}^2$ our task is to find the curve of shortest length
connecting them. In childhood, all of us learn that the shortest path
between two points is a straight line. That can be proved using,
for example, the theory of calculus of variations. The formulation
of this problem, in the sense of the calculus of variations, is to find the
function $y(\cdot)$ that solves the following problem:
\begin{equation}\label{Prob00}
\mathcal{L}[y(\cdot)]=\int_{a}^{b}\sqrt{1+\left(y'(t)\right)^2}d t
\longrightarrow \min ,\quad y(a)=y_1,\quad y(b)=y_2.
\end{equation}
\end{example}

In \cite{CD:Bohner:2004} Martin Bohner initiated the theory of
calculus of variations on time scales. The rest of the section will
be dedicated to present some results that will be necessary during
the thesis. For more on the subject we refer to~\cite{Tese:Rui} and references therein.

\begin{definition}
A function $f$ defined on $[a,b]_\mathbb{T}\times\mathbb{R}$ is
called \emph{continuous in the second variable, uniformly in the
first variable}, if for each $\varepsilon>0$ there exists $\delta>0$
such that $|x_1-x_2|<\delta$ implies
$|f(t,x_1)-f(t,x_2)|<\varepsilon$ for all $t\in[a,b]_\mathbb{T}$.
\end{definition}
\begin{lemma}[\textrm{cf.} Lemma 2.2 in \cite{CD:Bohner:2004}]
Suppose that $F(x)=\int_a^bf(t,x)\Delta t$ is well defined. If $f_x$
is continuous in $x$, uniformly in $t$, then
$F'(x)=\int_a^bf_x(t,x)\Delta t$.
\end{lemma}

Let us now extend \eqref{Prob0} to a generic time-scale.\\Let
$a,b\in\mathbb{T}$ and
$L(t,u,v):[a,b]^\kappa_\mathbb{T}\times\mathbb{R}^2\rightarrow\mathbb{R}$. 
Find a function $y\in \textrm{C}_{\textrm{rd}}^1$ such that
\begin{equation}\label{Prob2}
\mathcal{L}[y(\cdot)]=\int_{a}^{b}L(t,y^\sigma(t),y^\Delta(t))\Delta
t \longrightarrow \min ,\quad y(a)=y_a,\quad y(b)=y_b,
\end{equation}
with $y_a,y_b\in\mathbb{R}$.

\begin{remark}
If we fix $\mathbb{T}=\mathbb{R}$ problem~\eqref{Prob2} reduces to \eqref{Prob0}.
\end{remark}

\begin{remark}
If we fix $\mathbb{T}=\mathbb{Z}$ we obtain the discrete version of
\eqref{Prob2}. The goal is to find a function $y\in
\textrm{C}_{\textrm{rd}}^1$ such that
\begin{equation}\label{Prob3}
\mathcal{L}[y(\cdot)]=\sum_{t=a}^{b-1}L(t,y(t+1),y^\Delta(t))\Delta t
\longrightarrow \min ,\quad y(a)=y_a,\quad y(b)=y_b,
\end{equation}
where $a,~b\in \mathbb{Z}$ with $a<b$; $y_a,y_b\in \mathbb{R}$ and
$L:\mathbb{Z}\times\mathbb{R}^2\rightarrow\mathbb{R}$.
\end{remark}

\begin{definition}
For $f\in\textrm{C}_{\textrm{rd}}^1$ we define the norm
$$\|f\|=\sup_{t\in[a,b]^\kappa_\mathbb{T}}|f^\sigma(t)|
+\sup_{t\in[a,b]^\kappa_\mathbb{T}}|f^\Delta(t)|.$$
A function $\hat{y}\in \textrm{C}_{\textrm{rd}}^1$ with
$\hat{y}(a)=y_a$ and $\hat{y}(b)=y_b$ is called a (weak) local
minimizer for problem \eqref{Prob2} provided there exists $\delta>0$
such that $\mathcal{L}(\hat{y})\leq\mathcal{L}(\hat{y})$ for all
$y\in \textrm{C}_{\textrm{rd}}^1$ with $y(a)=y_a$ and $y(b)=y_b$ and
$\|y-\hat{y}\|<\delta$.
\end{definition}
\begin{definition}
A function $\eta\in\textrm{C}_{\textrm{rd}}^1$ is called an
\emph{admissible variation} provided $\eta\neq 0$ and
$\eta(a)=\eta(b)=0$.
\end{definition}
\begin{lemma}[\textrm{cf.} Lemma 3.4 in \cite{CD:Bohner:2004}]\label{lema1}
Let $y,\eta\in\textrm{C}_{\textrm{rd}}^1$ be arbitrary fixed
functions. Put
$f(t,\varepsilon)=L(t,y^\sigma(t)+\varepsilon\eta^\sigma(t),
y^\Delta(t)+\varepsilon\eta^\Delta(t))$
and $\Phi(\varepsilon)=\mathcal{L}(y+\varepsilon\eta)$,
$\varepsilon\in\mathbb{R}$. If $f_\varepsilon$ and
$f_{\varepsilon\varepsilon}$ are continuous in $\varepsilon$,
uniformly in $t$, then
\begin{align}
\Phi'(\varepsilon)&=\int_a^b[L_u(t,y^\sigma(t),y^\Delta(t))\eta^\sigma(t)
+L_v(t,y^\sigma(t),y^\Delta(t))\eta^\Delta(t)]\Delta t,\nonumber\\
\Phi''(\varepsilon)&=\int_a^b\{L_{uu}[y](t)(\eta^\sigma(t))^2
+2\eta^\sigma(t)L_{uv}[y](t)\eta^\Delta(t)
+L_{vv}[y](t)(\eta^\Delta(t))^2\}\Delta t,\nonumber
\end{align}
where $[y](t)=(t,y^\sigma(t),y^\Delta(t))$.
\end{lemma}
The next lemma is the time scales version of the classical
Dubois--Reymond lemma\index{Dubois--Reymond lemma on time scales}.
\begin{lemma}[\textrm{cf.} Lemma 4.1 in \cite{CD:Bohner:2004}]
\label{lem:DR} Let $g\in C_{\textrm{rd}}([a,b]_\mathbb{T}^\kappa)$.
Then,
$$\int_{a}^{b}g(t)\eta^\Delta(t)\Delta t=0,\quad
\mbox{for all $\eta\in C_{\textrm{rd}}^1([a,b]_\mathbb{T})$ with
$\eta(a)=\eta(b)=0$},$$ holds if and only if
$$g(t)=c,\quad\mbox{on $[a,b]^\kappa_\mathbb{T}$ for some
$c\in\mathbb{R}$}.$$
\end{lemma}

The following theorem give first and second order necessary 
optimality conditions for problem \eqref{Prob2}.

\begin{theorem}\cite{CD:Bohner:2004}\label{necessary:conditions:TS}
Suppose that $L$ satisfies the assumption of Lemma \ref{lema1}. If
$\hat{y}\in\textrm{C}_{\textrm{rd}}^1$ is a (weak) local minimizer for
problem given by \eqref{Prob2}, then necessarily
\begin{enumerate}
    \item
    $ L_v^\Delta[\hat{y}](t)=L_u[\hat{y}](t),\quad
    t\in[a,b]^{\kappa}_\mathbb{T}\quad\mbox{\emph{(time scales 
    Euler--Lagrange equation\index{Euler--Lagrange equation!time scales case})}}$.
    \item $L_{vv}[\hat{y}](t)+\mu(t)\{2L_{uv}[\hat{y}](t)+\mu(t)L_{uu}[\hat{y}](t)
    +(\mu^\sigma(t))^\ast L_{vv}[\hat{y}](\sigma(t))\}\geq 0,
    \quad t\in[a,b]^{\kappa^2}_\mathbb{T}\quad\\\mbox{\emph{(time scales Legendre's
    condition\index{Legendre's necessary condition!time scales case})}}$,
\end{enumerate}
where $[y](t)=(t,y^\sigma(t),y^\Delta(t))$ and
$\alpha^\ast=\frac{1}{\alpha}$ if
$\alpha\in\mathbb{R}\backslash\{0\}$ and $0^\ast=0$.
\end{theorem}

Using the theory of calculus of variations on time scales and considering
$y_1=0$ and $y_2=1$ in Example \ref{shortestdistance}, we can generalize
that example to the following one.
\begin{example}\cite[Example~4.3]{CD:Bohner:2004} Find the solution of the problem
\begin{equation}\label{shortestdistance:TS}
\int_a^b\sqrt{1+\left(y^\Delta\right)^2}\Delta t
\rightarrow~min,\quad\quad y(a)=0,~~y(b)=1.
\end{equation}
Writing \eqref{shortestdistance:TS} in form of \eqref{Prob2} we have
$$L(t,u,v)=\sqrt{1+v^2},\quad\quad L_u(t,u,v)=0\quad\text{and}\quad L_v(t,u,v)=\frac{v}{1+v^2}.$$
Suppose $\hat{y}$ is a local minimizer of \eqref{shortestdistance:TS}.
Then, by item 1 of Theorem~\ref{necessary:conditions:TS}, equation
\begin{equation}\label{EL:TS:example}
L_v^\Delta(t,\hat{y}^\sigma(t),\hat{y}^\Delta(t))=0,\quad
t\in[a,b]^\kappa,
\end{equation}
must hold. The last equation implies that there exist a constant
$c\in \mathbb{R}$ such that
$$L_v(t,\hat{y}^\sigma(t),\hat{y}^\Delta(t))\equiv c,\quad t\in[a,b],$$
i.e.,
\begin{equation}
\label{eqfinal}
\hat{y}^\Delta(t)=c\sqrt{1+\left(\hat{y}^\Delta\right)^2},\quad t\in[a,b]
\end{equation}
holds. Solving equation \eqref{eqfinal} we obtain
$$\hat{y}(t)=\frac{t-a}{b-a}\quad\text{for all}\quad t\in[a,b],$$
which is the equation of the straight line connecting the given
points.
\end{example}

We refer the reader to the PhD thesis~\cite{Tese:Rui} for more recent results 
on calculus of variations on time scales. Here our main interest on the subject 
is just as the starting point to go further and start developing the \emph{fractional} 
case on the time scale $\mathbb{Z}$ (Chapter~\ref{chap3}), 
$\mathbb{T}=\left(h\mathbb{Z}\right)_a$ (Chapter~\ref{chap4}) 
and possible on an arbitrary time scale (cf. Chapter \ref{chap5}). 
Before the original work of this thesis, already published in the international 
journals~\cite{comNuno:Rui:Z,7}, results on the fractional calculus of variations 
were restricted to the continuous case $\mathbb{T}=\mathbb{R}$ 
(see, e.g., \cite{RicDel,gastao:delfim,MyID:169}).

\clearpage{\thispagestyle{empty}\cleardoublepage}

\part{Original Work}

\clearpage{\thispagestyle{empty}\cleardoublepage}


\chapter{Fractional Variational Problems in $\mathbb{T}=\mathbb{Z}$}
\label{chap3}

In this chapter we introduce a discrete-time fractional 
calculus of variations on the time scale $\mathbb{Z}$. 
First and second order necessary optimality conditions 
are established. We finish the chapter 
with some examples illustrating the use
of the new Euler--Lagrange and Legendre type conditions.

\section{Introduction}
\label{Z::int}

In the last decades, considerable research has been done in
fractional calculus. This is particularly true in the area of the
calculus of variations, which is being subject to intense
investigations during the last few years
\cite{B:08,Ozlem,comNuno:Rui:Z,R:N:H:M:B:07,R:T:M:B:07}.
\\\\
As mentioned in Section~\ref{discretefract}, Miller and Ross define the
fractional sum of order $\nu>0$ by
\begin{equation}\label{Z::naosei88}
\Delta^{-\nu}f(t)=\frac{1}{\Gamma(\nu)}\sum_{s=a}^{t-\nu}(t-\sigma(s))^{(\nu-1)}f(s).
\end{equation}
\begin{remark}
This was done in analogy with the left Riemann--Liouville fractional
integral of order $\nu>0$~(\textrm{cf.}~(\ref{left::RLFI})),
\begin{equation*}
_a I_x^{-\nu}f(x)=\frac{1}{\Gamma(\nu)}\int_{a}^{x}(x-s)^{\nu-1}f(s)ds,
\end{equation*}
which can be obtained \emph{via} the solution of a linear
differential equation \cite{Miller,Miller1}. Some basic properties
of the sum in \eqref{Z::naosei88} were obtained in \cite{Miller}.
\end{remark}

More recently, F.~Atici and P.~Eloe \cite{Atici0,Atici} defined the
fractional difference of order $\alpha>0$, \textrm{i.e.},
$\Delta^\alpha f(t)=\Delta^m(\Delta^{-(m-\alpha)}f(t))$ with $m$ the
least integer satisfying $m \ge \alpha$, and developed some of its
properties that allow to obtain solutions of certain fractional
difference equations.

To the best of the author's knowledge, the first paper with a theory 
for the discrete calculus of variations (using forward difference operator) 
was written by Tomlinson Fort \cite{Fort} in 1937. Some important results 
of discrete calculus of variations are summarized in~\cite[Chap.~8]{book:DCV}.
It is well known that discrete analogues of differential equations can
be very useful in applications \cite{Zeidan1,book:DCV}. Therefore,
we consider pertinent  to start here a fractional discrete-time
theory of the calculus of variations.

Throughout the chapter we proceed to develop the
theory of \emph{fractional difference calculus}, namely, we
introduce the concept of left and right fractional sum/difference
(\textrm{cf.} Definitions~\ref{Z::def0} and \ref{Z::def1} below) and
prove some new results related to them.

We begin, in Section~\ref{Z::sec0}, to give the definitions and
results needed throughout. In Section~\ref{Z::sec1} we present and
prove the new results; in Section~\ref{Z::sec2} we give some
examples. Finally, in Section~\ref{Z::sec:conc} we mention the
main conclusions of the chapter, and some possible extensions and open questions.
Computer code done in the Computer Algebra System \textsf{Maxima} is given in Appendix A.


\section{Preliminaries}
\label{Z::sec0}

We begin by introducing some notation used throughout. Let $a$ be an
arbitrary real number and $b = k + a$ for a certain $k
\in\mathbb{N}$ with $k \ge 2$. In this chapter we restrict ourselves
to the time scale $\mathbb{T}= \{a, a + 1, \ldots, b\}$.

Denote by $\mathcal{F}$ the set of all real valued functions defined
on $\mathbb{T}$.
According with Example~\ref{Z::sigma:mu}, we have $\sigma(t)=t+1$,
$\rho(t)=t-1$.

The usual conventions $\sum_{t=c}^{c-1}f(t)=0$, $c\in\mathbb{T}$,
and $\prod_{i=0}^{-1}f(i)=1$ remain valid here. As usual, the
forward difference is defined by $\Delta f(t)=f^\sigma(t)-f(t)$. If
we have a function $f$ of two variables, $f(t,s)$, its partial
(difference) derivatives are denoted by $\Delta_t$ and $\Delta_s$,
respectively.
Recalling~\eqref{factorial:fracti} we can, for arbitrary $x,y\in\mathbb{R}$, 
define (when it makes sense)
\begin{equation}\label{powerZ}
x^{(y)}=\frac{\Gamma(x+1)}{\Gamma(x+1-y)}
\end{equation} where
$\Gamma$ is the gamma function.

While reaching the proof of
Theorem~\ref{Z::teor1} we actually ``find'' the definition of left
and right fractional sum:

\begin{definition}\label{Z::def0}
Let $f\in\mathcal{F}$. The \emph{left fractional sum} and the
\emph{right fractional sum} of order $\nu>0$ are defined,
respectively, as
\begin{equation}\label{Z::sum1}
_a\Delta_t^{-\nu}f(t)=\frac{1}{\Gamma(\nu)}\sum_{s=a}^{t-\nu}(t-\sigma(s))^{(\nu-1)}f(s),
\end{equation}
and
\begin{equation}\label{Z::sum2}
_t\Delta_b^{-\nu}f(t)=\frac{1}{\Gamma(\nu)}\sum_{s=t+\nu}^{b}(s-\sigma(t))^{(\nu-1)}f(s).
\end{equation}
\end{definition}

\begin{remark}
The above sums \eqref{Z::sum1} and \eqref{Z::sum2} are defined for
$t \in \{a+\nu, a + \nu + 1, \ldots, b + \nu\}$ and $t \in \{a-\nu,
a - \nu + 1, \ldots, b- \nu\}$, respectively, while $f(t)$ is
defined for $t \in \{a, a + 1, \ldots, b\}$. Throughout we will
write \eqref{Z::sum1} and \eqref{Z::sum2}, respectively, in the
following way:
\begin{equation*}
_a\Delta_t^{-\nu}f(t)=\frac{1}{\Gamma(\nu)}\sum_{s=a}^{t}(t+\nu-\sigma(s))^{(\nu-1)}f(s),\quad
t\in\mathbb{T},
\end{equation*}
\begin{equation*}
_t\Delta_b^{-\nu}f(t)=\frac{1}{\Gamma(\nu)}\sum_{s=t}^{b}(s+\nu-\sigma(t))^{(\nu-1)}f(s),\quad
t\in\mathbb{T}.
\end{equation*}
\end{remark}
\begin{remark}
The left fractional sum defined in \eqref{Z::sum1} coincides with
the fractional sum defined in \cite{Miller} (see also
\eqref{Z::naosei8}). The analogy of \eqref{Z::sum1} and
\eqref{Z::sum2} with the Riemann--Liouville left and right
fractional integrals of order $\nu>0$ is clear:
$$_a I_x^{-\nu}f(x)=\frac{1}{\Gamma(\nu)}\int_{a}^{x}(x-s)^{\nu-1}f(s)ds,$$
$$_x I_b^{-\nu}f(x)=\frac{1}{\Gamma(\nu)}\int_{x}^{b}(s-x)^{\nu-1}f(s)ds.$$
\end{remark}
It was proved in \cite{Miller} that $\lim_{\nu\rightarrow
0}{_a}\Delta_t^{-\nu}f(t)=f(t)$. We do the same for the right
fractional sum using a different method. Let $\nu>0$ be arbitrary.
Then,
\begin{align*}
{_t}\Delta_b^{-\nu}f(t)&=\frac{1}{\Gamma(\nu)}\sum_{s=t}^{b}(s+\nu-\sigma(t))^{(\nu-1)}f(s)\\
&=f(t)+\frac{1}{\Gamma(\nu)}\sum_{s=\sigma(t)}^{b}(s+\nu-\sigma(t))^{(\nu-1)}f(s)\\
&=f(t)+\sum_{s=\sigma(t)}^{b}\frac{\Gamma(s+\nu-t)}{\Gamma(\nu)\Gamma(s-t+1)}f(s)\\
&=f(t)+\sum_{s=\sigma(t)}^{b}\frac{\prod_{i=0}^{s-t-1}(\nu+i)}{\Gamma(s-t+1)}f(s).
\end{align*}
Therefore, $\lim_{\nu\rightarrow 0}{_t}\Delta_b^{-\nu}f(t)=f(t)$. It
is now natural to define
\begin{equation}\label{Z::naosei10}
_a\Delta_t^{0}f(t)={_t}\Delta_b^{0}f(t)=f(t),
\end{equation}
which we do here, and to write
\begin{equation}\label{Z::seila1}
{_a}\Delta_t^{-\nu}f(t)=f(t)+\frac{\nu}{\Gamma(\nu+1)}\sum_{s
=a}^{t-1}(t+\nu-\sigma(s))^{(\nu-1)}f(s),\quad t\in\mathbb{T},\quad
\nu\geq 0,
\end{equation}
\begin{equation*}
{_t}\Delta_b^{-\nu}f(t)=f(t)+\frac{\nu}{\Gamma(\nu+1)}\sum_{s
=\sigma(t)}^{b}(s+\nu-\sigma(t))^{(\nu-1)}f(s),\quad
t\in\mathbb{T},\quad \nu\geq 0.
\end{equation*}

The next theorem was proved in \cite{Atici}.
\begin{theorem}\cite{Atici}
\label{Z::thm2} Let $f\in\mathcal{F}$ and $\nu>0$. Then, the
equality
\begin{equation*}
{_a}\Delta_{t}^{-\nu}\Delta
f(t)=\Delta(_a\Delta_t^{-\nu}f(t))-\frac{(t+\nu-a)^{(\nu-1)}}{\Gamma(\nu)}f(a),\quad
t\in\mathbb{T}^\kappa,
\end{equation*}
holds.
\end{theorem}
\begin{remark}\label{Z::rem0}
It is easy to include the case $\nu=0$ in Theorem~\ref{Z::thm2}.
Indeed, in view of \eqref{Z::naosei9} and \eqref{Z::naosei10}, we
get
\begin{equation}\label{Z::seila0}
{_a}\Delta_{t}^{-\nu}\Delta f(t)=\Delta(_a\Delta_t^{-\nu}f(t))
-\frac{\nu}{\Gamma(\nu+1)}(t+\nu-a)^{(\nu-1)}f(a),\quad
t\in\mathbb{T}^\kappa,
\end{equation}
for all $\nu\geq 0$.
\end{remark}
Now, we prove the counterpart of Theorem~\ref{Z::thm2} for the right
fractional sum.
\begin{theorem}
Let $f\in\mathcal{F}$ and $\nu\geq 0$. Then, the equality
\begin{equation}\label{Z::naosei12}
{_t}\Delta_{\rho(b)}^{-\nu}\Delta
f(t)=\frac{\nu}{\Gamma(\nu+1)}(b+\nu-\sigma(t))^{(\nu-1)}f(b)
+\Delta(_t\Delta_b^{-\nu}f(t)),\quad t\in\mathbb{T}^\kappa,
\end{equation}
holds.
\end{theorem}

\begin{proof}
We only prove the case $\nu>0$ as the case $\nu=0$ is trivial (see
Remark~\ref{Z::rem0}). We start by fixing an arbitrary
$t\in\mathbb{T}^\kappa$ and prove that for all
$s\in\mathbb{T}^\kappa$,
\begin{equation}\label{prod:rule:s}
    \Delta_s\left((s+\nu-\sigma(t))^{(\nu-1)}f(s)\right)=(\nu-1)(s+\nu-\sigma(t))^{(\nu-2)}f^\sigma(s)
+(s+\nu-\sigma(t))^{(\nu-1)}\Delta f(s).
\end{equation}
By definition of forward difference with respect to variable $s$ we can write that
\begin{eqnarray}
\Delta_s\left((s+\nu-\sigma(t))^{(\nu-1)}f(s)\right)&=&\left(\sigma(s)+\nu
-\sigma(t)\right)^{(\nu-1)}f^\sigma(s)-\left(s+\nu-\sigma(t)\right)^{(\nu-1)}f(s)\nonumber\\
&=&\left(s+\nu-t\right)^{(\nu-1)}f^\sigma(s)-\left(s+\nu-t-1\right)^{(\nu-1)}f(s)\nonumber\\
&=&\frac{\Gamma(s+\nu-t+1)f^\sigma(s)}{\Gamma(s-t+2)}-\frac{\Gamma(s+\nu-t)f(s)}{\Gamma(s-t+1)}~.\label{simp:gamma:prod}
\end{eqnarray}
Applying~\eqref{Z::naosei9} to the numerator and denominator 
of first fraction we have that \eqref{simp:gamma:prod} is equal to
\[
\frac{(s+\nu-t)\Gamma(s+\nu-t)f^\sigma(s)}{(s-t+1)\Gamma(s-t+1)}-\frac{\Gamma(s+\nu-t)f(s)}{\Gamma(s-t+1)}
\]
\[
=\frac{\left[(s+\nu-t)f^\sigma(s)-(s-t+1)f(s)\right]\Gamma(s+\nu-t)}{(s-t+1)\Gamma(s-t+1)}
\]
\[
=\frac{\left[(s-t+1)f^\sigma(s)+(\nu-1)f^\sigma(s)-(s-t+1)f(s)\right]\Gamma(s+\nu-t)}{(s-t+1)\Gamma(s-t+1)}
\]
\[
=(\nu-1)\frac{\Gamma(s+\nu-\sigma(t)+1)}{\Gamma(s+\nu-\sigma(t)-(\nu-2)+1)}f^\sigma(s)
+\frac{\Gamma(s+\nu-\sigma(t)+1)}{\Gamma(s+\nu-\sigma(t)-(\nu-1)+1)}(f^\sigma(s)-f(s))
\]
\[
=(\nu-1)(s+\nu-\sigma(t))^{(\nu-2)}f^\sigma(s)
+(s+\nu-\sigma(t))^{(\nu-1)}\Delta f(s)
~.\]
Now, with~\eqref{prod:rule:s} proven, we can state that
\begin{equation*}
\begin{split}
\frac{1}{\Gamma(\nu)} & \sum_{s=t}^{b-1}(s+\nu-\sigma(t))^{(\nu-1)}\Delta f(s)\\
&=\left[\frac{(s+\nu-\sigma(t))^{(\nu-1)}}{\Gamma(\nu)}f(s)\right]_{s=t}^{s=b}
-\frac{1}{\Gamma(\nu)}\sum_{s=t}^{b-1}(\nu-1)(s+\nu-\sigma(t))^{(\nu-2)}
f^\sigma(s)\\
&=\frac{(b+\nu-\sigma(t))^{(\nu-1)}}{\Gamma(\nu)}f(b)
-\frac{(\nu-1)^{(\nu-1)}}{\Gamma(\nu)}f(t)\\
&\qquad
-\frac{1}{\Gamma(\nu)}\sum_{s=t}^{b-1}(\nu-1)(s+\nu-\sigma(t))^{(\nu-2)}
f^\sigma(s).
\end{split}
\end{equation*}
We now compute $\Delta(_t\Delta_b^{-\nu}f(t))$:
\begin{equation*}
\begin{split}
\Delta(_t\Delta_b^{-\nu}f(t))&=
\frac{1}{\Gamma(\nu)}\left[\sum_{s=\sigma(t)}^{b}(s+\nu-\sigma(t+1))^{(\nu-1)}
f(s)\right.\\
& \left. \qquad\qquad\quad -\sum_{s=t}^{b}(s+\nu-\sigma(t))^{(\nu-1)} f(s)\right]\\
&=\frac{1}{\Gamma(\nu)}\left[\sum_{s=\sigma(t)}^{b}(s+\nu-\sigma(t+1))^{(\nu-1)}
f(s)\right.\\
&\qquad\qquad\quad
\left.-\sum_{s=\sigma(t)}^{b}(s+\nu-\sigma(t))^{(\nu-1)}
f(s)\right]-\frac{(\nu-1)^{(\nu-1)}}{\Gamma(\nu)}f(t)\\
&=\frac{1}{\Gamma(\nu)}\sum_{s=\sigma(t)}^{b}\Delta_t(s+\nu-\sigma(t))^{(\nu-1)}
f(s)-\frac{(\nu-1)^{(\nu-1)}}{\Gamma(\nu)}f(t)\\
&=-\frac{1}{\Gamma(\nu)}\sum_{s=t}^{b-1}(\nu-1)(s+\nu-\sigma(t))^{(\nu-2)}
f^\sigma(s)-\frac{(\nu-1)^{(\nu-1)}}{\Gamma(\nu)}f(t).
\end{split}
\end{equation*}
Since $t$ is arbitrary, the theorem is proved.
\end{proof}

\begin{definition}\label{Z::def1}
Let $0<\alpha\leq 1$ and set $\mu=1-\alpha$. Then, the \emph{left
fractional difference} and the \emph{right fractional difference} of
order $\alpha$ of a function $f\in\mathcal{F}$ are defined,
respectively, by
$$_a\Delta_t^\alpha f(t)=\Delta(_a\Delta_t^{-\mu}f(t)),\quad t\in\mathbb{T}^\kappa,$$
and
$$_t\Delta_b^\alpha f(t)=-\Delta(_t\Delta_b^{-\mu}f(t))),\quad t\in\mathbb{T}^\kappa.$$
\end{definition}


\section{Main results}
\label{Z::sec1}

Our aim is to introduce the discrete-time (in time scale
$\mathbb{T}=\{a,a+1,\ldots,b\}$ ) fractional problem of the calculus
of variations and to prove corresponding necessary optimality
conditions. In order to obtain an analogue of the Euler--Lagrange
equation (\textrm{cf.} Theorem~\ref{Z::thm0}) we first prove a
fractional formula of summation by parts. Our results
give discrete analogues to the fractional Riemann--Liouville results
available in the literature: Theorem~\ref{Z::teor1} is the discrete
analog of fractional integration by parts \cite{Riewe,Samko};
Theorem~\ref{Z::thm0} is the discrete analog of the fractional
Euler--Lagrange equation of Agrawal \cite[Theorem~1]{agr0}; the
natural boundary conditions \eqref{Z::rui1} and \eqref{Z::rui2} are
the discrete fractional analogues of the transversality conditions
in \cite{A:TC:06,MyID:182}. However, to the best of the author's
knowledge, no counterpart to our Theorem~\ref{Z::thm1} exists in the
literature of continuous fractional variational problems.


\subsection{Fractional summation by parts}

The next lemma is used in the proof of Theorem~\ref{Z::teor1}.

\begin{lemma}\label{Z::lem0}
Let $f$ and $h$ be two functions defined on $\mathbb{T}^\kappa$ and
$g$ a function defined on
$\mathbb{T}^\kappa\times\mathbb{T}^\kappa$. Then, the equality
\begin{equation*}
\sum_{\tau=a}^{b-1}f(\tau)\sum_{s=a}^{\tau-1}g(\tau,s)h(s)
=\sum_{\tau=a}^{b-2}h(\tau)\sum_{s=\sigma(\tau)}^{b-1}g(s,\tau)f(s)
\end{equation*}
holds.
\end{lemma}
\begin{proof}
Choose $\mathbb{T} = \mathbb{Z}$ and $F(\tau,s) = f(\tau) g(\tau,s)
h(s)$ in Theorem~10 of \cite{Akin}.
\end{proof}

The next result gives a \emph{fractional summation by parts}
formula.

\begin{theorem}[Fractional summation by parts]
\label{Z::teor1} Let $f$ and $g$ be real valued functions defined on
$\mathbb{T}^k$ and $\mathbb{T}$, respectively. Fix $0<\alpha\leq 1$
and put $\mu=1-\alpha$. Then,
\begin{multline*}
\sum_{t=a}^{b-1}f(t)_a\Delta_t^\alpha
g(t)=f(b-1)g(b)-f(a)g(a)+\sum_{t=a}^{b-2}{_t\Delta_{\rho(b)}^\alpha
f(t)g^\sigma(t)}\\
+\frac{\mu}{\Gamma(\mu+1)}g(a)\left(\sum_{t=a}^{b-1}(t+\mu-a)^{(\mu-1)}f(t)
-\sum_{t=\sigma(a)}^{b-1}(t+\mu-\sigma(a))^{(\mu-1)}f(t)\right).
\end{multline*}
\end{theorem}

\begin{proof}
From \eqref{Z::seila0} we can write
\begin{align}
\sum_{t=a}^{b-1}f(t)_a\Delta_t^\alpha
g(t)&=\sum_{t=a}^{b-1}f(t)\Delta(_a\Delta_t^{-\mu} g(t))\nonumber\\
&=\sum_{t=a}^{b-1}f(t)\left[_a\Delta_t^{-\mu}\Delta
g(t)+\frac{\mu}{\Gamma(\mu+1)}(t+\mu-a)^{(\mu-1)}g(a)\right]\nonumber\\
&=\sum_{t=a}^{b-1}f(t)_a\Delta_t^{-\mu}\Delta
g(t)+\sum_{t=a}^{b-1}\frac{\mu}{\Gamma(\mu+1)}(t+\mu-a)^{(\mu-1)}f(t)g(a)\label{Z::rui0}.
\end{align}
Using \eqref{Z::seila1} we get
\begin{align*}
\sum_{t=a}^{b-1}&f(t)_a\Delta_t^{-\mu}\Delta
g(t)\\
&=\sum_{t=a}^{b-1}f(t)\Delta
g(t)+\frac{\mu}{\Gamma(\mu+1)}\sum_{t=a}^{b-1}f(t)\sum_{s=a}^{t-1}(t+\mu-\sigma(s))^{(\mu-1)}\Delta
g(s)\\
&=\sum_{t=a}^{b-1}f(t)\Delta
g(t)+\frac{\mu}{\Gamma(\mu+1)}\sum_{t=a}^{b-2}\Delta
g(t)\sum_{s=\sigma(t)}^{b-1}(s+\mu-\sigma(t))^{(\mu-1)}f(s)\\
&=f(b-1)[g(b)-g(b-1)]+\sum_{t=a}^{b-2}\Delta
g(t)_t\Delta_{\rho(b)}^{-\mu} f(t),
\end{align*}
where the third equality follows by Lemma~\ref{Z::lem0}. We proceed
to develop the right hand side of the last equality as follows:
\begin{equation*}
\begin{split}
f&(b-1)[g(b)-g(b-1)]+\sum_{t=a}^{b-2}\Delta
g(t)_t\Delta_{\rho(b)}^{-\mu} f(t)\\
&=f(b-1)[g(b)-g(b-1)] +\left[g(t)_t\Delta_{\rho(b)}^{-\mu}
f(t)\right]_{t=a}^{t=b-1}-\sum_{t=a}^{b-2}
g^\sigma(t)\Delta(_t\Delta_{\rho(b)}^{-\mu} f(t))\\
&=f(b-1)g(b)-f(a)g(a)-\frac{\mu}{\Gamma(\mu+1)}g(a)\sum_{s
=\sigma(a)}^{b-1}(s+\mu-\sigma(a))^{(\mu-1)}f(s)\\
&\qquad +\sum_{t=a}^{b-2}{\left(_t\Delta_{\rho(b)}^\alpha
f(t)\right)g^\sigma(t)}\, ,
\end{split}
\end{equation*}
where the first equality follows from the usual summation by parts
formula. Putting this into \eqref{Z::rui0}, we get:
\begin{multline*}
\sum_{t=a}^{b-1}f(t)_a\Delta_t^\alpha
g(t)=f(b-1)g(b)-f(a)g(a)+\sum_{t=a}^{b-2}{\left(_t\Delta_{\rho(b)}^\alpha
f(t)\right)g^\sigma(t)}\\
+\frac{g(a)\mu}{\Gamma(\mu+1)}\sum_{t=a}^{b-1}\frac{(t+\mu-a)^{(\mu-1)}}{\Gamma(\mu)}f(t)
-\frac{g(a)\mu}{\Gamma(\mu+1)}\sum_{s=\sigma(a)}^{b-1}(s+\mu-\sigma(a))^{(\mu-1)}f(s).
\end{multline*}
The theorem is proved.
\end{proof}


\subsection{Necessary optimality conditions}

We begin to fix two arbitrary real numbers $\alpha$ and $\beta$ such
that $\alpha,\beta\in(0,1]$. Further, we put $\mu=1-\alpha$ and
$\nu=1-\beta$.

Let a function $L(t,u,v,w):\mathbb{T}^\kappa\times\mathbb{R}\times
\mathbb{R}\times\mathbb{R}\rightarrow\mathbb{R}$ be given. We assume
that the second-order partial derivatives $L_{uu}$, $L_{uv}$,
$L_{uw}$, $L_{vw}$, $L_{vv}$, and $L_{ww}$ exist and are continuous.

Consider the functional
$\mathcal{L}:\mathcal{F}\rightarrow\mathbb{R}$ defined by
\begin{equation}
\label{Z::naosei7}
\mathcal{L}(y(\cdot))=\sum_{t=a}^{b-1}L(t,y^{\sigma}(t),{_a}\Delta_t^\alpha
y(t),{_t}\Delta_b^\beta y(t))
\end{equation}
and the problem, that we denote by (P), of minimizing
\eqref{Z::naosei7} subject to the boundary conditions $y(a)=A$ and
$y(b)=B$ ($A,B\in\mathbb{R}$). Our aim is to derive necessary
conditions of first and second order for problem (P).
\begin{definition}
\label{Z::def:norm} For $f\in\mathcal{F}$ we define the norm
$$\|f\|=\max_{t\in\mathbb{T}^\kappa}|f^\sigma(t)|
+\max_{t\in\mathbb{T}^\kappa}|_a\Delta_t^\alpha
f(t)|+\max_{t\in\mathbb{T}^\kappa}|_t\Delta_b^\beta f(t)|.$$ A
function $\tilde{y}\in\mathcal{F}$ with $\tilde{y}(a)=A$ and
$\tilde{y}(b)=B$ is called a local minimizer for problem (P)
provided there exists $\delta>0$ such that
$\mathcal{L}(\tilde{y})\leq\mathcal{L}(y)$ for all $y\in\mathcal{F}$
with $y(a)=A$ and $y(b)=B$ and $\|y-\tilde{y}\|<\delta$.
\end{definition}

\begin{remark}
It is easy to see that Definition~\ref{Z::def:norm} gives a norm in
$\mathcal{F}$. Indeed, it is clear that $||f||$ is nonnegative, and
for an arbitrary $f\in\mathcal{F}$ and $k\in\mathbb{R}$ we have
$\|kf\|=|k|\|f\|$. The triangle inequality is also easy to prove:
\begin{align*}
\|f+g\|&=\max_{t\in\mathbb{T}^\kappa}|f(t)+g(t)|
+\max_{t\in\mathbb{T}^\kappa}|_a\Delta_t^\alpha (f+g)(t)|
+\max_{t\in\mathbb{T}^\kappa}|_t\Delta_b^\alpha(f+g)(t)|\\
&\leq\max_{t\in\mathbb{T}^\kappa}\left[|f(t)|+|g(t)|\right]
+\max_{t\in\mathbb{T}^\kappa}\left[|_a\Delta_t^\alpha
f(t)|+|_a\Delta_t^\alpha
g(t)|\right]\\
& \qquad +\max_{t\in\mathbb{T}^\kappa}\left[|_t\Delta_b^\alpha
f(t)|+|_t\Delta_b^\alpha g(t)|\right]\\
&\leq\|f\|+\|g\|.
\end{align*}
The only possible doubt is to prove that $||f|| = 0$ implies that
$f(t) = 0$ for any $t \in \mathbb{T} = \{a, a+1,\ldots,b\}$. Suppose
$||f|| = 0$. It follows that
\begin{gather}
\max_{t\in\mathbb{T}^\kappa}|f^\sigma(t)|  = 0\, , \label{Z::rr2:1}\\
\max_{t\in\mathbb{T}^\kappa}|_a\Delta_t^\alpha f(t)| = 0\, ,\label{Z::rr2:2}\\
\max_{t\in\mathbb{T}^\kappa}|_t\Delta_b^\beta f(t)| = 0 \, .
\label{Z::rr2:3}
\end{gather}
From \eqref{Z::rr2:1} we conclude that $f(t) = 0$ for all  $t \in
\{a+1,\ldots,b\}$. It remains to prove that $f(a) = 0$. To prove
this we use \eqref{Z::rr2:2} (or \eqref{Z::rr2:3}). Indeed, from
\eqref{Z::rr2:1} we can write
\begin{equation*}
\begin{split}
_a\Delta_t^\alpha & f(t)
=  \Delta\left(\frac{1}{\Gamma(1-\alpha)}\sum_{s=a}^{t}(t+1
-\alpha-\sigma(s))^{(-\alpha)} f(s) \right)\\
&= \frac{1}{\Gamma(1-\alpha)}\left(\sum_{s=a}^{t+1}(t+2 -
\alpha-\sigma(s))^{(-\alpha)} f(s)
- \sum_{s=a}^{t}(t+1-\alpha-\sigma(s))^{(-\alpha)} f(s) \right)\\
&= \frac{1}{\Gamma(1-\alpha)} \left( (t+2 -
\alpha-\sigma(a))^{(-\alpha)} f(a)
- (t+1 - \alpha-\sigma(a))^{(-\alpha)} f(a)\right)\\
&= \frac{f(a)}{\Gamma(1-\alpha)} \Delta (t+1 -
\alpha-\sigma(a))^{(-\alpha)}
\end{split}
\end{equation*}
and since by \eqref{Z::rr2:2} $_a\Delta_t^\alpha f(t) = 0$, one
concludes that $f(a) = 0$ (because $(t+1 -
\alpha-\sigma(a))^{(-\alpha)}$ is not a constant).
\end{remark}

\begin{definition}
A function $\eta\in\mathcal{F}$ is called an admissible variation
for problem (P) provided $\eta\neq 0$ and $\eta(a)=\eta(b)=0$.
\end{definition}

The next theorem presents a first order necessary optimality condition for
problem (P).

\begin{theorem}[The fractional discrete-time Euler--Lagrange equation]
\label{Z::thm0}\index{Euler--Lagrange equation!discrete fractional case} 
If $\tilde{y}\in\mathcal{F}$ is a local minimizer
for problem (P), then
\begin{equation}\label{Z::EL}
L_u[\tilde{y}](t) +{_t}\Delta_{\rho(b)}^\alpha
L_v[\tilde{y}](t)+{_a}\Delta_t^\beta L_w[\tilde{y}](t)=0
\end{equation}
holds for all $t\in\mathbb{T}^{\kappa^2}$, where the operator
$[\cdot]$ is defined by
$$
[y](s) =(s,y^{\sigma}(s),{_a}\Delta_s^\alpha y(s),{_s}\Delta_b^\beta
y(s)) .
$$
\end{theorem}

\begin{proof}
Suppose that $\tilde{y}(\cdot)$ is a local minimizer of
$\mathcal{L}(\cdot)$. Let $\eta(\cdot)$ be an arbitrary fixed
admissible variation and define the function
$\Phi:\left(-\frac{\delta}{\|\eta(\cdot)\|},
\frac{\delta}{\|\eta(\cdot)\|}\right)\rightarrow\mathbb{R}$ by
\begin{equation}
\label{Z::fi}
\Phi(\varepsilon)=\mathcal{L}(\tilde{y}(\cdot)+\varepsilon\eta(\cdot)).
\end{equation}
This function has a minimum at $\varepsilon=0$, so we must have
$\Phi'(0)=0$, \textrm{i.e.},
$$\sum_{t=a}^{b-1}\left[L_u[\tilde{y}](t)\eta^\sigma(t)
+L_v[\tilde{y}](t){_a}\Delta_t^\alpha\eta(t)
+L_w[\tilde{y}](t){_t}\Delta_b^\beta\eta(t)\right]=0,$$ which we may
write, equivalently, as
\begin{multline}
\label{Z::rui3} L_u[\tilde{y}](t)\eta^\sigma(t)|_{t=\rho(b)}
+\sum_{t=a}^{b-2}L_u[\tilde{y}](t)\eta^\sigma(t)\\
+\sum_{t=a}^{b-1}L_v[\tilde{y}](t){_a}\Delta_t^\alpha\eta(t)
+\sum_{t=a}^{b-1}L_w[\tilde{y}](t){_t}\Delta_b^\beta\eta(t)=0.
\end{multline}
Using Theorem~\ref{Z::teor1}, and the fact that $\eta(a)=\eta(b)=0$,
we get for the third term in \eqref{Z::rui3} that
\begin{equation}
\label{Z::naosei5}
\sum_{t=a}^{b-1}L_v[\tilde{y}](t){_a}\Delta_t^\alpha\eta(t)
=\sum_{t=a}^{b-2}\left({_t}\Delta_{\rho(b)}^\alpha
L_v[\tilde{y}](t)\right)\eta^\sigma(t).
\end{equation}
Using \eqref{Z::naosei12} it follows that
\begin{equation}
\label{Z::naosei4}
\begin{aligned}
\sum_{t=a}^{b-1}&L_w[\tilde{y}](t){_t}\Delta_b^\beta\eta(t)\\
&=-\sum_{t=a}^{b-1}L_w[\tilde{y}](t)\Delta({_t}\Delta_b^{-\nu}\eta(t))\\
&=-\sum_{t=a}^{b-1}L_w[\tilde{y}](t)\left[{_t}\Delta_{\rho(b)}^{-\nu}\Delta
\eta(t)-\frac{\nu}{\Gamma(\nu+1)}(b+\nu-\sigma(t))^{(\nu-1)}\eta(b)\right]\\
&=-\left(\sum_{t=a}^{b-1}L_w[\tilde{y}](t){_t}\Delta_{\rho(b)}^{-\nu}\Delta
\eta(t)-\frac{\nu\eta(b)}{\Gamma(\nu+1)}\sum_{t=a}^{b-1}(b
+\nu-\sigma(t))^{(\nu-1)}L_w[\tilde{y}](t)\right).
\end{aligned}
\end{equation}
We now use Lemma~\ref{Z::lem0} to get
\begin{equation}
\label{Z::naosei2}
\begin{split}
\sum_{t=a}^{b-1} &
L_w[\tilde{y}](t){_t}\Delta_{\rho(b)}^{-\nu}\Delta
\eta(t)\\
&=\sum_{t=a}^{b-1}L_w[\tilde{y}](t)\Delta\eta(t)+\frac{\nu}{\Gamma(\nu+1)}\sum_{t=a}^{b-2}
L_w[\tilde{y}](t)\sum_{s=\sigma(t)}^{b-1}(s+\nu-\sigma(t))^{(\nu-1)}
\Delta\eta(s)\\
&=\sum_{t=a}^{b-1}L_w[\tilde{y}](t)\Delta\eta(t)
+\frac{\nu}{\Gamma(\nu+1)}\sum_{t=a}^{b-1}\Delta\eta(t)\sum_{s=a}^{t
-1}(t+\nu-\sigma(s))^{(\nu-1)}L_w[\tilde{y}](s)\\
&=\sum_{t=a}^{b-1}\Delta\eta(t){_a}\Delta^{-\nu}_t
L_w[\tilde{y}](t).
\end{split}
\end{equation}
We apply again the usual summation by parts formula, this time to
\eqref{Z::naosei2}, to obtain:
\begin{equation}
\label{Z::naosei3}
\begin{split}
\sum_{t=a}^{b-1} & \Delta\eta(t){_a}\Delta^{-\nu}_t
L_w[\tilde{y}](t)\\
&=\sum_{t=a}^{b-2}\Delta\eta(t){_a}\Delta^{-\nu}_t
L_w[\tilde{y}](t)+(\eta(b)-\eta(\rho(b))){_a}\Delta^{-\nu}_t
L_w[\tilde{y}](t)|_{t=\rho(b)}\\
&=\left[\eta(t){_a}\Delta^{-\nu}_t
L_w[\tilde{y}](t)\right]_{t=a}^{t=b-1}-\sum_{t=a}^{b-2}\eta^\sigma(t)\Delta({_a}\Delta^{-\nu}_t
L_w[\tilde{y}](t))\\
&\qquad +\eta(b){_a}\Delta^{-\nu}_t
L_w[\tilde{y}](t)|_{t=\rho(b)}-\eta(b-1){_a}\Delta^{-\nu}_t
L_w[\tilde{y}](t)|_{t=\rho(b)}\\
&=\eta(b){_a}\Delta^{-\nu}_t
L_w[\tilde{y}](t)|_{t=\rho(b)}-\eta(a){_a}\Delta^{-\nu}_t
L_w[\tilde{y}](t)|_{t=a}-\sum_{t=a}^{b-2}\eta^\sigma(t){_a}\Delta^{\beta}_t
L_w[\tilde{y}](t).
\end{split}
\end{equation}
Since $\eta(a)=\eta(b)=0$ it follows, from \eqref{Z::naosei2} and
\eqref{Z::naosei3}, that
$$
\sum_{t=a}^{b-1}L_w[\tilde{y}](t){_t}\Delta_{\rho(b)}^{-\nu}\Delta
\eta(t)=-\sum_{t=a}^{b-2}\eta^\sigma(t){_a}\Delta^{\beta}_t
L_w[\tilde{y}](t)
$$
and, after inserting in \eqref{Z::naosei4}, that
\begin{equation}\label{Z::naosei6}
\sum_{t=a}^{b-1}L_w[\tilde{y}](t){_t}\Delta_b^\beta\eta(t)
=\sum_{t=a}^{b-2}\eta^\sigma(t){_a}\Delta^{\beta}_t
L_w[\tilde{y}](t).
\end{equation}
By \eqref{Z::naosei5} and \eqref{Z::naosei6} we may write
\eqref{Z::rui3} as
$$
\sum_{t=a}^{b-2}\left[L_u[\tilde{y}](t) +{_t}\Delta_{\rho(b)}^\alpha
L_v[\tilde{y}](t)+{_a}\Delta_t^\beta
L_w[\tilde{y}](t)\right]\eta^\sigma(t)=0.
$$
Since the values of $\eta^\sigma(t)$ are arbitrary for
$t\in\mathbb{T}^{\kappa^2}$, the Euler--Lagrange equation
\eqref{Z::EL} holds along $\tilde{y}$.
\end{proof}

\begin{remark}
If the initial condition $y(a)=A$ is not present (\textrm{i.e.},
$y(a)$ is free), we can use standard techniques to show that the
following supplementary condition must be fulfilled:
\begin{multline}\label{Z::rui1}
-L_v(a)+\frac{\mu}{\Gamma(\mu+1)}\left(\sum_{t=a}^{b-1}(t+\mu-a)^{(\mu-1)}L_v[\tilde{y}](t)\right.\\
\left.-\sum_{t=\sigma(a)}^{b-1}(t+\mu-\sigma(a))^{(\mu-1)}L_v[\tilde{y}](t)\right)+
L_w(a)=0.
\end{multline}
Similarly, if $y(b)=B$ is not present (\textrm{i.e.}, $y(b)$ is
free), the equality
\begin{multline}\label{Z::rui2}
L_u(\rho(b))+L_v(\rho(b))-L_w(\rho(b))\\
+\frac{\nu}{\Gamma(\nu+1)}\Biggl(\sum_{t=a}^{b-1}(b+\nu-\sigma(t))^{(\nu-1)}L_w[\tilde{y}](t)\\
-\sum_{t=a}^{b-2}(\rho(b)+\nu-\sigma(t))^{(\nu-1)}L_w[\tilde{y}](t)\Biggr)=0
\end{multline}
holds. We just note that the first term in \eqref{Z::rui2} arises
from the first term on the left hand side of \eqref{Z::rui3}.
Equalities \eqref{Z::rui1} and \eqref{Z::rui2} are the fractional
discrete-time \emph{natural boundary conditions}.
\end{remark}
The next result is a particular case of our Theorem~\ref{Z::thm0}.

\begin{cor}[The discrete-time Euler--Lagrange equation
-- \textrm{cf.}, \textrm{e.g.}, \cite{CD:Bohner:2004,RD}] If
$\tilde{y}$ is a solution to the problem
\begin{equation}
\label{Z::eq:BPCV:DT}
\begin{gathered}
\mathcal{L}(y(\cdot))=\sum_{t=a}^{b-1}L(t,y(t+1),\Delta y(t)) \longrightarrow \min \\
y(a)=A \, , \quad y(b)=B \, ,
\end{gathered}
\end{equation}
then $L_u(t,\tilde{y}(t+1),\Delta\tilde{y}(t)) -\Delta
L_v(t,\tilde{y}(t+1),\Delta\tilde{y}(t))=0$ for all $t \in \{ a,
\ldots, b-2\}$.
\end{cor}
\begin{proof}
Follows from Theorem~\ref{Z::thm0} with $\alpha=1$ and a $L$ not
depending on $w$.
\end{proof}

We derive now the second order necessary condition for problem (P),
\textrm{i.e.}, we obtain Legendre's necessary condition for the
fractional difference setting.
\begin{theorem}[The fractional discrete-time Legendre condition]
\label{Z::thm1} If $\tilde{y}\in\mathcal{F}$ is a local minimizer
for problem (P), then the inequality
\begin{multline*}
L_{uu}[\tilde{y}](t)+2L_{uv}[\tilde{y}](t)
+L_{vv}[\tilde{y}](t)+L_{vv}[\tilde{y}](\sigma(t))(\mu-1)^2\\
+\sum_{s=\sigma(\sigma(t))}^{b-1}L_{vv}[\tilde{y}](s)\left(
\frac{\mu(\mu-1)\prod_{i=0}^{s-t-3}(\mu+i+1)}{(s-t)\Gamma(s-t)}\right)^2
+2L_{uw}[\tilde{y}](t)(\nu-1)\\
+2(\nu-1)L_{vw}[\tilde{y}](t)
+2(\mu-1)L_{vw}[\tilde{y}](\sigma(t))+L_{ww}[\tilde{y}](t)(1-\nu)^2\\
+L_{ww}[\tilde{y}](\sigma(t))
+\sum_{s=a}^{t-1}L_{ww}[\tilde{y}](s)\left(\frac{\nu(1-\nu)\prod_{i=0}^{t
-s-2}(\nu+i)}{(\sigma(t)-s)\Gamma(\sigma(t)-s)}\right)^2 \geq 0
\end{multline*}
holds for all $t\in\mathbb{T}^{\kappa^2}$, where
$[\tilde{y}](t)=(t,\tilde{y}^{\sigma}(t),{_a}\Delta_t^\alpha
\tilde{y}(t),{_t}\Delta_b^\beta\tilde{y}(t))$.
\end{theorem}

\begin{proof}
By the hypothesis of the theorem, and letting $\Phi$ be as in
\eqref{Z::fi}, we get
\begin{equation}
\label{Z::des0} \Phi''(0)\geq 0
\end{equation}
for an arbitrary admissible variation $\eta(\cdot)$. Inequality
\eqref{Z::des0} is equivalent to
\begin{multline*}
\sum_{t=a}^{b-1}\left[L_{uu}[\tilde{y}](t)(\eta^\sigma(t))^2
+2L_{uv}[\tilde{y}](t)\eta^\sigma(t){_a}\Delta_t^\alpha\eta(t)
+L_{vv}[\tilde{y}](t)({_a}\Delta_t^\alpha\eta(t))^2\right.\\
\left.+2L_{uw}[\tilde{y}](t)\eta^\sigma(t){_t}\Delta_b^\beta\eta(t)
+2L_{vw}[\tilde{y}](t){_a}\Delta_t^\alpha\eta(t){_t}\Delta_b^\beta\eta(t)
+L_{ww}[\tilde{y}](t)({_t}\Delta_b^\beta\eta(t))^2\right]\geq 0.
\end{multline*}
Let $\tau\in\mathbb{T}^{\kappa^2}$ be arbitrary and define
$\eta:\mathbb{T}\rightarrow\mathbb{R}$ by
\[ \eta(t) = \left\{ \begin{array}{ll}
1 & \mbox{if $t=\sigma(\tau)$};\\
0 & \mbox{otherwise}.\end{array} \right. \] It follows that
$\eta(a)=\eta(b)=0$, \textrm{i.e.}, $\eta$ is an admissible
variation. Using \eqref{Z::seila0} (note that $\eta(a)=0$), we get
\begin{equation*}
\begin{split}
\sum_{t=a}^{b-1}&\left[L_{uu}[\tilde{y}](t)(\eta^\sigma(t))^2
+2L_{uv}[\tilde{y}](t)\eta^\sigma(t){_a}\Delta_t^\alpha\eta(t)
+L_{vv}[\tilde{y}](t)({_a}\Delta_t^\alpha\eta(t))^2\right]\\
&=\sum_{t=a}^{b-1}\left\{L_{uu}[\tilde{y}](t)(\eta^\sigma(t))^2\right.\\
&\qquad \left.
+2L_{uv}[\tilde{y}](t)\eta^\sigma(t)\left[\Delta\eta(t)
+\frac{\mu}{\Gamma(\mu+1)}\sum_{s=a}^{t-1}(t
+\mu-\sigma(s))^{(\mu-1)}\Delta\eta(s)\right]\right.\\
&\qquad \left. +L_{vv}[\tilde{y}](t)\left(\Delta\eta(t)
+\frac{\mu}{\Gamma(\mu+1)}\sum_{s=a}^{t-1}(t
+\mu-\sigma(s))^{(\mu-1)}\Delta\eta(s)\right)^2\right\}\\
&= L_{uu}[\tilde{y}](\tau)+2L_{uv}[\tilde{y}](\tau)+L_{vv}[\tilde{y}](\tau)\\
&\qquad
+\sum_{t=\sigma(\tau)}^{b-1}L_{vv}[\tilde{y}](t)\left(\Delta\eta(t)
+\frac{\mu}{\Gamma(\mu+1)}\sum_{s=a}^{t-1}(t+\mu-\sigma(s))^{(\mu-1)}\Delta\eta(s)\right)^2.
\end{split}
\end{equation*}
Observe that
\begin{multline*}
\sum_{t=\sigma(\sigma(\tau))}^{b
-1}L_{vv}[\tilde{y}](t)\left(\frac{\mu}{\Gamma(\mu+1)}\sum_{s=a}^{t-1}(t
+\mu-\sigma(s))^{(\mu-1)}\Delta\eta(s)\right)^2
+ L_{vv}(\sigma(\tau))(-1+\mu)^2 \\
= \sum_{t=\sigma(\tau)}^{b-1}L_{vv}[\tilde{y}](t)\left(\Delta\eta(t)
+\frac{\mu}{\Gamma(\mu+1)}\sum_{s=a}^{t-1}(t+\mu-\sigma(s))^{(\mu-1)}\Delta\eta(s)\right)^2
\, .
\end{multline*}
We show next that
\begin{multline*}
\sum_{t=\sigma(\sigma(\tau))}^{b-1}L_{vv}[\tilde{y}](t)\left(\frac{\mu}{\Gamma(\mu
+1)}\sum_{s=a}^{t-1}(t+\mu-\sigma(s))^{(\mu-1)}\Delta\eta(s)\right)^2\\
=\sum_{t=\sigma(\sigma(\tau))}^{b-1}L_{vv}[\tilde{y}](t)\left(\frac{\mu(\mu-1)\prod_{i=0}^{t
-\tau-3}(\mu+i+1)}{(t-\tau)\Gamma(t-\tau)}\right)^2.
\end{multline*}
Let $t\in[\sigma(\sigma(\tau)),b-1]\cap\mathbb{Z}$. Then,
\begin{equation*}
\begin{split}
&\frac{\mu}{\Gamma(\mu+1)}\sum_{s=a}^{t-1}(t+\mu-\sigma(s))^{(\mu-1)}\Delta\eta(s)\\
&\quad
=\frac{\mu}{\Gamma(\mu+1)}\left[\sum_{s=a}^{\tau}(t+\mu-\sigma(s))^{(\mu-1)}\Delta\eta(s)
+\sum_{s=\sigma(\tau)}^{t-1}(t+\mu-\sigma(s))^{(\mu-1)}\Delta\eta(s)\right]\\
&\quad =\frac{\mu}{\Gamma(\mu+1)}\left[(t+\mu-\sigma(\tau))^{(\mu-1)}
-(t+\mu-\sigma(\sigma(\tau)))^{(\mu-1)}\right]\\
&\quad=\frac{\mu}{\Gamma(\mu+1)}\left[\frac{\Gamma(t+\mu-\sigma(\tau)+1)}{\Gamma(t
+\mu-\sigma(\tau)+1-(\mu-1))}-\frac{\Gamma(t+\mu-\sigma(\sigma(\tau))
+1)}{\Gamma(t+\mu-\sigma(\sigma(\tau))+1-(\mu-1))}\right]
\end{split}
\end{equation*}
\begin{equation}
\label{Z::rui10}
\begin{split}
&\quad=\frac{\mu}{\Gamma(\mu+1)}\left[\frac{\Gamma(t+\mu-\tau)}{\Gamma(t-\tau+1)}
-\frac{\Gamma(t-\tau+\mu-1)}{\Gamma(t-\tau)}\right]\\
&\quad=\frac{\mu}{\Gamma(\mu+1)}\left[\frac{(t+\mu-\tau-1)\Gamma(t+\mu-\tau-1)}{(t
-\tau)\Gamma(t-\tau)}-\frac{(t-\tau)\Gamma(t-\tau+\mu-1)}{(t-\tau)\Gamma(t-\tau)}\right]\\
&\quad=\frac{\mu}{\Gamma(\mu+1)}\frac{(\mu-1)\Gamma(t-\tau+\mu-1)}{(t-\tau)\Gamma(t-\tau)}\\
&\quad=\frac{\mu(\mu-1)\prod_{i=0}^{t-\tau-3}(\mu+i+1)}{(t-\tau)\Gamma(t-\tau)},
\end{split}
\end{equation}
which proves our claim. Observe that we can write
${_t}\Delta_b^\beta\eta(t)=-{_t}\Delta_{\rho(b)}^{-\nu}\Delta
\eta(t)$ since $\eta(b)=0$. It is not difficult to see that the
following equality holds:
$$\sum_{t=a}^{b-1}2L_{uw}[\tilde{y}](t)\eta^\sigma(t){_t}\Delta_b^\beta\eta(t)
=-\sum_{t=a}^{b-1}2L_{uw}[\tilde{y}](t)\eta^\sigma(t){_t}\Delta_{\rho(b)}^{-\nu}\Delta
\eta(t)=2L_{uw}[\tilde{y}](\tau)(\nu-1).$$ Moreover,
\begin{equation*}
\begin{split}
\sum_{t=a}^{b-1} & 2L_{vw}[\tilde{y}](t){_a}\Delta_t^\alpha\eta(t){_t}\Delta_b^\beta\eta(t)\\
&=-2\sum_{t=a}^{b-1}L_{vw}[\tilde{y}](t)\left\{\left(\Delta\eta(t)+\frac{\mu}{\Gamma(\mu+1)}
\cdot\sum_{s=a}^{t-1}(t+\mu-\sigma(s))^{(\mu-1)}\Delta\eta(s)\right)\right.\\
& \qquad \left. \cdot \left[\Delta\eta(t)
+\frac{\nu}{\Gamma(\nu+1)}\sum_{s=\sigma(t)}^{b-1}(s+\nu-\sigma(t))^{(\nu-1)}\Delta\eta(s)\right]\right\}\\
&=2(\nu-1)L_{vw}[\tilde{y}](\tau)+2(\mu-1)L_{vw}[\tilde{y}](\sigma(\tau)).
\end{split}
\end{equation*}
Finally, we have that
\begin{equation*}
\begin{split}
\sum_{t=a}^{b-1} & L_{ww}[\tilde{y}](t)({_t}\Delta_b^\beta\eta(t))^2\\
&=\sum_{t=a}^{\sigma(\tau)}L_{ww}[\tilde{y}](t)\left[\Delta\eta(t)
+\frac{\nu}{\Gamma(\nu+1)}\sum_{s=\sigma(t)}^{b-1}(s
+\nu-\sigma(t))^{(\nu-1)}\Delta\eta(s)\right]^2\\
&=\sum_{t=a}^{\tau-1}L_{ww}[\tilde{y}](t)\left[\frac{\nu}{\Gamma(\nu
+1)}\sum_{s=\sigma(t)}^{b-1}(s+\nu-\sigma(t))^{(\nu-1)}\Delta\eta(s)\right]^2\\
&\qquad +L_{ww}[\tilde{y}](\tau)(1-\nu)^2+L_{ww}[\tilde{y}](\sigma(\tau))\\
&=\sum_{t=a}^{\tau-1}L_{ww}[\tilde{y}](t)\left[\frac{\nu}{\Gamma(\nu+1)}\left\{(\tau
+\nu-\sigma(t))^{(\nu-1)}-(\sigma(\tau)+\nu-\sigma(t))^{(\nu-1)}\right\}\right]^2\\
&\qquad
+L_{ww}[\tilde{y}](\tau)(1-\nu)^2+L_{ww}[\tilde{y}](\sigma(\tau)).
\end{split}
\end{equation*}
Similarly as we have done in \eqref{Z::rui10}, we obtain that
$$\frac{\nu}{\Gamma(\nu+1)}\left[(\tau+\nu-\sigma(t))^{(\nu-1)}
-(\sigma(\tau)+\nu-\sigma(t))^{(\nu-1)}\right]
=\frac{\nu(1-\nu)\prod_{i=0}^{\tau-t-2}(\nu+i)}{(\sigma(\tau)-t)\Gamma(\sigma(\tau)-t)}.$$
We are done with the proof.
\end{proof}
A trivial corollary of our result gives the discrete-time version of
Legendre's necessary condition.
\begin{cor}[The discrete-time Legendre condition
-- \textrm{cf.}, \textrm{e.g.}, \cite{CD:Bohner:2004,Zeidan2}]
\label{Z::cor:2} If $\tilde{y}$ is a solution to the problem
\eqref{Z::eq:BPCV:DT}, then
$$L_{uu}[\tilde{y}](t)+2L_{uv}[\tilde{y}](t)+L_{vv}[\tilde{y}](t)
    +L_{vv}[\tilde{y}](\sigma(t))\geq 0$$
holds for all $t\in\mathbb{T}^{\kappa^2}$, where
$[\tilde{y}](t)=(t,\tilde{y}^{\sigma}(t),\Delta\tilde{y}(t))$.
\end{cor}
\begin{proof}
We consider problem (P) with $\alpha=1$ and $L$ not depending on
$w$. The choice $\alpha=1$ implies $\mu=0$, and the result follows
immediately from Theorem~\ref{Z::thm1}.
\end{proof}


\section{Examples}
\label{Z::sec2}

In this section we present three illustrative examples. The results
were obtained using the open source Computer Algebra System
\textsf{Maxima}.\footnote{\url{http://maxima.sourceforge.net}} All
computations were done running \textsf{Maxima} on an
Intel$\circledR$ Core$^{TM}$2 Duo, CPU of 2.27GHz with 3Gb of RAM.
Our \textsf{Maxima} definitions are given in Appendix~\ref{Z::Maxima}.

\begin{example}
\label{Z::ex:1} Let us consider the following problem:
\begin{equation}
\label{Z::eq:ex1} J_{\alpha}(y)=\sum_{t=0}^{b-1}
\left({_0}\Delta_t^\alpha y(t)\right)^2 \longrightarrow \min \, ,
\quad y(0) = A \, , \quad y(b) = B \, .
\end{equation}
In this case Theorem~\ref{Z::thm1} is trivially satisfied. We obtain
the solution $\tilde{y}$ to our Euler--Lagrange equation
\eqref{Z::EL} for the case $b = 2$ using the computer algebra system
\textsf{Maxima}. Using our \textsf{Maxima} package (see the
definition of the command \texttt{extremal} in Appendix~\ref{Z::Maxima}) we do
\begin{verbatim}
         L1:v^2$
         extremal(L1,0,2,A,B,alpha,alpha);
\end{verbatim}
to obtain (2 seconds)
\begin{equation}
\label{Z::eq:sol:ex1} \tilde{y}(1) = \frac{2\,\alpha\,B+\left(
{\alpha}^{3}-{\alpha}^{2}+2\,\alpha\right) \,A}{2\,{\alpha}^{2}+2}
\, .
\end{equation}
For the particular case $\alpha = 1$ the equality
\eqref{Z::eq:sol:ex1} gives $\tilde{y}(1) = \frac{A+B}{2}$, which
coincides with the solution to the (non-fractional) discrete problem
\begin{equation*}
\sum_{t=0}^{1} \left(\Delta y(t)\right)^2 = \sum_{t=0}^{1}
\left(y(t+1)-y(t)\right)^2 \longrightarrow \min \, , \quad y(0) = A
\, , \quad y(2) = B \, .
\end{equation*}
Similarly, we can obtain exact formulas of the extremal on bigger
intervals (for bigger values of $b$). For example, the solution of
problem \eqref{Z::eq:ex1} with $b = 3$ is (35 seconds)
\begin{equation*}
\begin{split}
\tilde{y}(1) &= \frac{\left( 6\,{\alpha}^{2}+6\,\alpha\right)
\,B+\left( 2\,{\alpha}^{5}
+2\,{\alpha}^{4}+10\,{\alpha}^{3}-2\,{\alpha}^{2} +12\,\alpha\right)
\,A}{3\,{\alpha}^{4}+6\,{\alpha}^{3}
+15\,{\alpha}^{2}+12} \, ,\\
\tilde{y}(2) &= \frac{\left(
12\,{\alpha}^{3}+12\,{\alpha}^{2}+24\,\alpha\right) \,B+\left(
{\alpha}^{6}+{\alpha}^{5}+7\,{\alpha}^{4}-{\alpha}^{3}+4\,{\alpha}^{2}+12\,\alpha\right)
\,A}{6\,{\alpha}^{4}+12\,{\alpha}^{3}+30\,{\alpha}^{2}+24} \, ;
\end{split}
\end{equation*}
and the solution of problem \eqref{Z::eq:ex1} with $b = 4$ is (72
seconds)
\begin{equation*}
\begin{split}
\tilde{y}(1) &=
\frac{3\,{\alpha}^{7}+15\,{\alpha}^{6}+57\,{\alpha}^{5}+69\,{\alpha}^{4}
+156\,{\alpha}^{3}-12\,{\alpha}^{2}+144\,\alpha}{\xi} A\\
&\qquad + \frac{24\,{\alpha}^{3}+72\,{\alpha}^{2}+48\,\alpha}{\xi} B \, ,\\
\tilde{y}(2) &=
\frac{{\alpha}^{8}+5\,{\alpha}^{7}+22\,{\alpha}^{6}+32\,{\alpha}^{5}
+67\,{\alpha}^{4}+35\,{\alpha}^{3}+54\,{\alpha}^{2}+72\,\alpha}{\xi} A\\
&\qquad + \frac{24\,{\alpha}^{4}+72\,{\alpha}^{3}+120\,{\alpha}^{2}
+72\,\alpha}{\xi} B \, ,\\
\tilde{y}(3) &=
\frac{{\alpha}^{9}+6\,{\alpha}^{8}+30\,{\alpha}^{7}+60\,{\alpha}^{6}+117\,{\alpha}^{5}
+150\,{\alpha}^{4}-4\,{\alpha}^{3}+216\,{\alpha}^{2}+288\,\alpha}{\zeta} A\\
&\qquad + \frac{72\,{\alpha}^{5}+288\,{\alpha}^{4}+792\,{\alpha}^{3}
+576\,{\alpha}^{2}+864\,\alpha}{\zeta} B \, ,
\end{split}
\end{equation*}
where
\begin{equation*}
\begin{split}
\xi &= 4\,{\alpha}^{6}
+24\,{\alpha}^{5}+88\,{\alpha}^{4}+120\,{\alpha}^{3}
+196\,{\alpha}^{2}+144 \, , \\
\zeta &= 24\,{\alpha}^{6}+144\,{\alpha}^{5}+528\,{\alpha}^{4}
+720\,{\alpha}^{3}+1176\,{\alpha}^{2}+864 \, .
\end{split}
\end{equation*}
Consider now problem \eqref{Z::eq:ex1} with $b = 4$, $A=0$, and
$B=1$. In Table~\ref{Z::tab:1} we show the extremal values
$\tilde{y}(1)$, $\tilde{y}(2)$, $\tilde{y}(3)$, and corresponding
$\tilde{J}_{\alpha}$, for some values of $\alpha$. Our numerical
results show that the fractional extremal converges to the classical
(integer order) extremal when $\alpha$ tends to one. This is
illustrated in Figure~\ref{Z::Fig:0}. The numerical results from
Table~\ref{Z::tab:1} and Figure~\ref{Z::Fig:2} show that for this
problem the smallest value of $\tilde{J}_{\alpha}$,
$\alpha\in]0,1]$, occur for $\alpha=1$ (\textrm{i.e.}, the smallest
value of $\tilde{J}_{\alpha}$ occurs for the classical
non-fractional case).
\begin{figure}[htp]
\begin{center}
\includegraphics[scale=0.6]{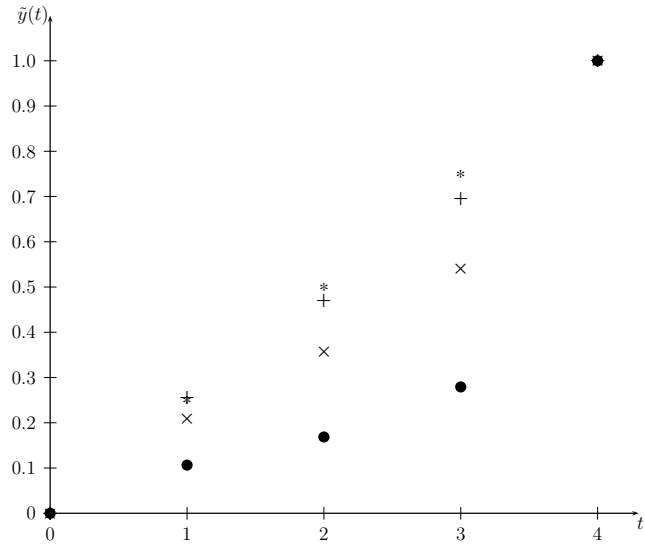}
  \caption{Extremal $\tilde{y}(t)$ of Example~\ref{Z::ex:1}
  with $b = 4$, $A=0$, $B=1$, and different $\alpha$'s
  ($\bullet$: $\alpha=0.25$; $\times$: $\alpha=0.5$;
  $+$: $\alpha=0.75$; $\ast$: $\alpha=1$).}\label{Z::Fig:0}
\end{center}
\end{figure}
{\small
\begin{table}[!htbp]
  \centering
  \begin{tabular}{|c|c|c|c|c|}
    \hline
    $\alpha$ & $\tilde{y}(1)$ & $\tilde{y}(2)$ & $\tilde{y}(3)$ & $\tilde{J}_{\alpha}$\\
    \hline
    0.25 & 0.10647146897355 & 0.16857982587479 & 0.2792657904952 & 0.90855653524095 \\
    \hline
    0.50 & 0.20997375328084 & 0.35695538057743 & 0.54068241469816 & 0.67191601049869 \\
    \hline
    0.75 & 0.25543605027861 & 0.4702345471038 & 0.69508876506414 & 0.4246209666969\\
    \hline
    1 & 0.25 & 0.5 & 0.75 & 0.25\\
    \hline
  \end{tabular}
\caption{The extremal values $\tilde{y}(1)$, $\tilde{y}(2)$ and
$\tilde{y}(3)$ of problem \eqref{Z::eq:ex1} with $b = 4$, $A=0$, and
$B=1$ for different $\alpha$'s.}\label{Z::tab:1}
\end{table}
}
\begin{figure}[htp]
\begin{center}
\includegraphics[scale=0.6]{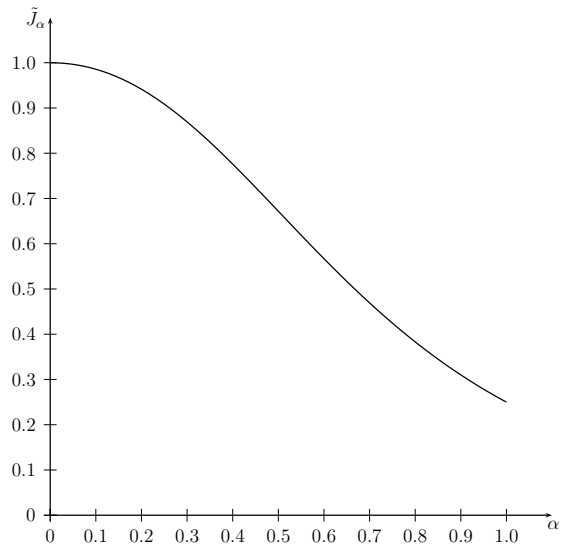}
\caption{Function $\tilde{J}_{\alpha}$ of Example~\ref{Z::ex:1} with
$b = 4$, $A=0$, and $B=1$.}\label{Z::Fig:2}
\end{center}
\end{figure}
\end{example}

\begin{example}
\label{Z::ex:2} In this example we generalize problem
\eqref{Z::eq:ex1} to
\begin{equation}
\label{Z::eq:ex2}
 J_{\alpha,\beta}=\sum_{t=0}^{b-1} \gamma_1
\Bigl({_0}\Delta_t^\alpha y(t)\Bigr)^2 + \gamma_2
\Bigl({_t}\Delta_b^\beta y(t)\Bigr)^2 \longrightarrow \min \, ,
\quad y(0) = A \, , \quad y(b) = B \, .
\end{equation}
As before, we solve the associated Euler--Lagrange equation
\eqref{Z::EL} for the case $b = 2$ with the help of our
\textsf{Maxima} package (35 seconds):
\begin{verbatim}
         L2:(gamma[1])*v^2+(gamma[2])*w^2$
         extremal(L2,0,2,A,B,alpha,beta);
\end{verbatim}
\begin{equation*}
\tilde{y}(1) =
\frac{\left(2\,\gamma_{2}\,\beta+\gamma_{1}\,{\alpha}^{3}-\gamma_{1}\,{\alpha}^{2}
+2\,\gamma_{1}\,\alpha\right)
\,A+\left(\gamma_{2}\,{\beta}^{3}-\gamma_{2}\,{\beta}^{2}
+2\,\gamma_{2}\,\beta+2\,\gamma_{1}\,\alpha\right)
\,B}{2\,\gamma_{2}\,{\beta}^{2}+2\,
\gamma_{1}\,{\alpha}^{2}+2\,\gamma_{2}+2\,\gamma_{1}} \, .
\end{equation*}
Consider now problem \eqref{Z::eq:ex2} with $\gamma_1=\gamma_2=1$,
$b=2$, $A=0$, $B=1$, and $\beta=\alpha$. In Table~\ref{Z::tab:2} we
show the values of $\tilde{y}(1)$ and $\tilde{J}_{\alpha} :=
J_{\alpha,\alpha}(\tilde{y}(1))$ for some values of $\alpha$. We
concluded, numerically, that the fractional extremal $\tilde{y}(1)$
tends to the classical (non-fractional) extremal when $\alpha$ tends
to one. Differently from Example~\ref{Z::ex:1}, the smallest value
of $\tilde{J}_{\alpha}$, $\alpha\in]0,1]$, does not occur here for
$\alpha=1$ (see Figure~\ref{Z::Fig:4}). The smallest value of
$\tilde{J}_{\alpha}$, $\alpha\in]0,1]$, occurs for
$\alpha=0.61747447161482$. {\small
\begin{table}[!htbp]
  \centering
  \begin{tabular}{|c|c|c|}
    \hline
    $\alpha$ & $\tilde{y}(1)$  & $\tilde{J}_{\alpha}$\\
    \hline
    0.25 & 0.22426470588235 &  0.96441291360294 \\
    \hline
    0.50 & 0.375 & 0.9140625 \\
    \hline
    0.75 & 0.4575 &  0.91720703125\\
    \hline
    1 & 0.5 & 1\\
    \hline
  \end{tabular}
  \caption{The extremal $\tilde{y}(1)$ of problem \eqref{Z::eq:ex2}
  for different values of $\alpha$
  ($\gamma_1=\gamma_2=1$, $b=2$, $A=0$, $B=1$, and $\beta = \alpha$).}\label{Z::tab:2}
\end{table}
}
\begin{figure}[htp]
\begin{center}
\includegraphics[scale=0.6]{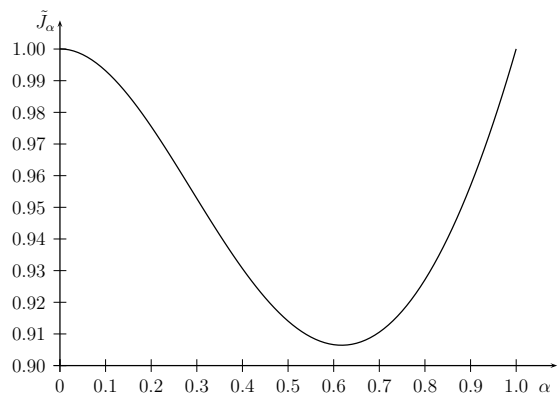}
  \caption{Function $\tilde{J}_{\alpha}$ of Example~\ref{Z::ex:2}
with $\gamma_1=\gamma_2=1$, $b=2$, $A=0$, $B=1$, and $\beta =
\alpha$.}\label{Z::Fig:4}
\end{center}
  \end{figure}
\end{example}

\begin{example}
\label{Z::ex:3} Our last example is a discrete version of the
fractional continuous problem \cite[Example~2]{agr2}:
\begin{equation}
\label{Z::eq:ex3} J_{\alpha}=\sum_{t=0}^{1}
\frac{1}{2}\left({_0}\Delta_t^\alpha y(t)\right)^2-y^{\sigma}(t)
\longrightarrow \min \, , \quad y(0) = 0 \, , \quad y(2) = 0 \, .
\end{equation}
The Euler--Lagrange extremal of \eqref{Z::eq:ex3} is easily obtained
with our \textsf{Maxima} package (4 seconds):
\begin{verbatim}
         L3:(1/2)*v^2-u;$
         extremal(L3,0,2,0,0,alpha,beta);
\end{verbatim}
\begin{equation}
\label{Z::eq:sol:ex3} \tilde{y}(1) =  \frac{1}{{\alpha}^{2}+1} \, .
\end{equation}
For the particular case $\alpha = 1$ the equality
\eqref{Z::eq:sol:ex3} gives $\tilde{y}(1) = \frac{1}{2}$, which
coincides with the solution to the non-fractional discrete problem
\begin{gather*}
\sum_{t=0}^{1} \frac{1}{2}\left(\Delta y(t)\right)^2-y^{\sigma}(t) =
\sum_{t=0}^{1} \frac{1}{2}\left(y(t+1)-y(t)\right)^2-y(t+1)
\longrightarrow \min \, , \\
y(0) = 0 \, , \quad y(2) = 0 \, .
\end{gather*}
In Table~\ref{Z::tab:3} we show the values of $\tilde{y}(1)$ and
$\tilde{J}_{\alpha}$ for some $\alpha$'s. As seen in
Figure~\ref{Z::Fig:6}, for $\alpha=1$ one gets the maximum value of
$\tilde{J}_{\alpha}$, $\alpha\in]0,1]$. {\small
\begin{table}[!htbp]
  \centering
  \begin{tabular}{|c|c|c|}
    \hline
    $\alpha$ & $\tilde{y}(1)$  & $\tilde{J}_{\alpha}$\\
    \hline
    0.25 & 0.94117647058824 &  -0.47058823529412 \\
    \hline
    0.50 & 0.8 & -0.4 \\
    \hline
    0.75 & 0.64 &  -0.32\\
    \hline
    1 & 0.5 & -0.25\\
    \hline
  \end{tabular}
  \caption{Extremal values $\tilde{y}(1)$ of \eqref{Z::eq:ex3}
  for different $\alpha$'s}\label{Z::tab:3}
\end{table}
}
\begin{figure}[!htbp]
\begin{center}
\includegraphics[scale=0.6]{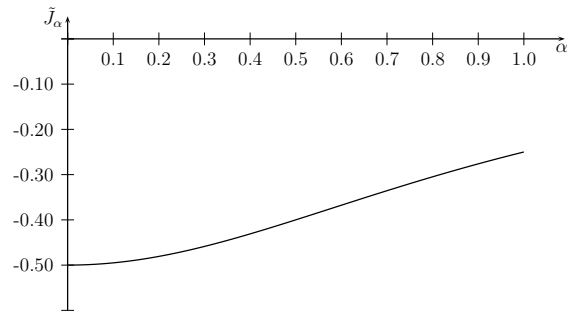}
\caption{Function $\tilde{J}_{\alpha}$ of
Example~\ref{Z::ex:3}.}\label{Z::Fig:6}
\end{center}
\end{figure}
\end{example}


\section{Conclusion}
\label{Z::sec:conc}

In this chapter we introduce the study of fractional discrete-time
problems of the calculus of variations of order $\alpha$,
$0 < \alpha \le 1$, with left and right discrete operators
of Riemann--Liouville type. For $\alpha = 1$
we obtain the classical discrete-time results
of the calculus of variations \cite{book:DCV}.

Main results of the chapter include
a fractional summation by parts formula (Theorem~\ref{Z::teor1}),
a fractional discrete-time Euler--Lagrange equation (Theorem~\ref{Z::thm0}),
transversality conditions \eqref{Z::rui1} and \eqref{Z::rui2},
and a fractional discrete-time Legendre condition (Theorem~\ref{Z::thm1}).
From the analysis of the results obtained from computer experiments,
we conclude that when the value of $\alpha$ approaches one,
the optimal value of the fractional discrete functional converges
to the optimal value of the classical (non-fractional) discrete problem.
On the other hand, the value of $\alpha$ for which the functional
attains its minimum varies with the concrete problem under consideration.


\section{State of the Art}

The discrete-time calculus is also a very important tool in
practical applications and in the modeling of real phenomena.
Therefore, it is not a surprise that fractional discrete calculus
is recently under strong development (see, \textrm{e.g.}, 
\cite{B:J:06, 7, Cresson:Frederico:Torres, Goodrich1, Goodrich2} and references therein).
There are some recent theses~\cite{Holm,Sengul} with results related with discrete calculus 
on time scale $\mathbb{T}=\left\{a,a+1,a+2,\ldots\right\}$. Thesis~\cite{Sengul} links the use 
of fractional difference equations with tumor growth modeling, and addresses a very important question: 
by changing the order of the difference equations from integers to fractional, in which conditions, 
we are able to provide more accurate models for real world problems?
For more on the subject see~\cite{sofia}.\\\\
In this chapter we use fractional difference operators of Riemann--Liouville type. 
However, there are some results related with fractional difference operators in Caputo sense 
\cite{Anastassiou2} like in \cite{Chen}, where the existence of solutions 
for IVP involving nonlinear fractional difference equations is discussed.
\\\\
The results of this chapter are published in~\cite{comNuno:Rui:Z} and were presented 
by the author at the Workshop on Control, Nonsmooth Analysis and Optimization, 
Celebrating Francis Clarke's and Richard Vinter's 60th Birthday, Porto, Portugal, 
May 4--8, 2009 in a contributed talk entitled {\it Necessary Optimality Conditions 
for Fractional Difference Problems of the Calculus of Variations}.

\clearpage{\thispagestyle{empty}\cleardoublepage}


\chapter{Fractional Variational Problems in $\mathbb{T}=(\lowercase{h}\mathbb{Z})_{\lowercase{a}}$}
\label{chap4}


We introduce a discrete-time fractional calculus of variations on
the time scale $(h\mathbb{Z})_a$, $h > 0$. First and second order
necessary optimality conditions are established. Examples
illustrating the use of the new Euler-Lagrange and Legendre type
conditions are given.

\section{Introduction}
\label{hZ::int}

Although the fractional
Euler--Lagrange equations are obtained in a similar manner as in the
standard variational calculus \cite{R:N:H:M:B:07}, some classical
results are extremely difficult to be proved in a fractional
context. This explains, for example, why a fractional Legendre type
condition is absent from the literature of fractional variational
calculus. In this chapter we give a first result in this direction
(\textrm{cf.} Theorem~\ref{hZ::thm1}).

Despite its importance in applications, less is known for
discrete-time fractional systems \cite{R:N:H:M:B:07}.
Our objective is proceed to develop the
theory of \emph{fractional difference calculus}, namely, we
introduce the concept of left and right fractional sum/difference
(\textrm{cf.} Definition~\ref{hZ::def0}). In
Section~\ref{hZ::sec0} we introduce notations, we give necessary
definitions, and prove some preliminary results needed in the
sequel. Main results of the paper appear in Section~\ref{hZ::sec1}: we
prove a fractional formula of $h$-summation by parts
(Theorem~\ref{hZ::teor1}), and necessary optimality conditions of
first and second order (Theorems~\ref{hZ::thm0} and \ref{hZ::thm1},
respectively) for the proposed $h$-fractional problem of the
calculus of variations \eqref{hZ::naosei7}. Before the end of the chapter we present
in Section~\ref{hZ::sec2} some illustrative examples and give some conclusions 
in Section~\ref{hZ::sec:conc} and the state of the art in Section~\ref{hZ::sec:state art}.


\section{Preliminaries}
\label{hZ::sec0}

One way to approach the Riemann-Liouville fractional calculus is
through the theory of linear differential equations \cite{Podlubny}.
Miller and Ross \cite{Miller} use an analogous methodology to
introduce fractional discrete operators for the case
$\mathbb{T}=\mathbb{Z}=\{a,a+1,a+2,\ldots\}, a\in\mathbb{R}$. That
was our starting point for the results presented in Chapter
\ref{chap3}. Because the results obtained in the previous chapter
were good, we start thinking how to generalize them to an arbitrary
time scale. However, since we found several difficulties when we try to
get results for a time-scale where the graininess function is not
constant, here we keep the graininess function constant but not,
necessarily, equal to one and we give a step further: we use the
theory of time scales in order to introduce fractional discrete
operators to the more general case
$\mathbb{T}=(h\mathbb{Z})_a=\{a,a+h,a+2h,\ldots\}$,
$a\in\mathbb{R}$, $h>0$.

Before going further we need to collected more results on the theory of
time-scales.

Until we state the opposite, the following results are true for a arbitrary time scale.

For $n\in \mathbb{N}_0$ and rd-continuous functions
$p_i:\mathbb{T}\rightarrow \mathbb{R}$, $1\leq i\leq n$, let us
consider the $n$th order linear dynamic equation
\begin{equation}
\label{hZ::linearDiffequa} Ly=0\, , \quad \text{ where }
Ly=y^{\Delta^n}+\sum_{i=1}^n p_iy^{\Delta^{n-i}} \, .
\end{equation}
A function $y:\mathbb{T}\rightarrow \mathbb{R}$ is said to be a
solution of equation (\ref{hZ::linearDiffequa}) on $\mathbb{T}$
provided $y$ is $n$ times delta differentiable on
$\mathbb{T}^{\kappa^n}$ and satisfies $Ly(t)=0$ for all
$t\in\mathbb{T}^{\kappa^n}$.

\begin{lemma}\cite[p.~239]{livro:2001}
\label{hZ::8:88 Bohner} If $z=\left(z_1, \ldots,z_n\right) :
\mathbb{T} \rightarrow \mathbb{R}^n$ satisfies for all $t\in
\mathbb{T}^\kappa$
\begin{equation}\label{hZ::5.86}
z^\Delta=A(t)z(t),\qquad\mbox{where}\qquad A=\left(
                                               \begin{array}{ccccc}
                                                 0 & 1 & 0 & \ldots & 0 \\
                                                 \vdots & 0 & 1 & \ddots & \vdots \\
                                                 \vdots &  & \ddots & \ddots & 0 \\
                                                 0 & \ldots & \ldots & 0 & 1 \\
                                                 -p_n & \ldots & \ldots & -p_2 & -p_1 \\
                                               \end{array}
                                             \right)
\end{equation}
then $y=z_1$ is a solution of equation \eqref{hZ::linearDiffequa}.
Conversely, if $y$ solves \eqref{hZ::linearDiffequa} on
$\mathbb{T}$, then $z=\left(y, y^\Delta,\ldots,
y^{\Delta^{n-1}}\right) : \mathbb{T}\rightarrow \mathbb{R}$
satisfies \eqref{hZ::5.86} for all $t\in\mathbb{T}^{\kappa^n}$
\end{lemma}

\begin{definition}\cite[p.~239]{livro:2001}
We say that equation (\ref{hZ::linearDiffequa}) is \emph{regressive}\index{Regressive!equation}
provided $I + \mu(t) A(t)$ is invertible for all $t \in
\mathbb{T}^\kappa$, where $A$ is the matrix in \eqref{hZ::5.86}.
\end{definition}

\begin{definition}\cite[p.~250]{livro:2001}
We define the Cauchy function $y:\mathbb{T} \times
\mathbb{T}^{\kappa^n}\rightarrow \mathbb{R}$ for the linear dynamic
equation~(\ref{hZ::linearDiffequa}) to be, for each fixed
$s\in\mathbb{T}^{\kappa^n}$, the solution of the initial value
problem
\begin{equation}
\label{hZ::IVP} Ly=0,\quad
y^{\Delta^i}\left(\sigma(s),s\right)=0,\quad 0\leq i \leq n-2,\quad
y^{\Delta^{n-1}}\left(\sigma(s),s\right)=1\, .
\end{equation}
\end{definition}

\begin{theorem}\cite[p.~251]{livro:2001}\label{hZ::eqsol}
Suppose $\{y_1,\ldots,y_n\}$ is a fundamental system of the
regressive equation~(\ref{hZ::linearDiffequa}). Let $f\in C_{rd}$.
Then the solution of the initial value problem
$$
Ly=f(t),\quad y^{\Delta^i}(t_0)=0,\quad 0\leq i\leq n-1 \, ,
$$
is given by $y(t)=\int_{t_0}^t y(t,s)f(s)\Delta s$, where $y(t,s)$
is the Cauchy function for ~(\ref{hZ::linearDiffequa}).
\end{theorem}

It is known that $y(t,s):=H_{n-1}(t,\sigma(s))$ is the Cauchy
function for $y^{\Delta^n}=0$, where $H_{n-1}$ is a time scale
generalized polynomial \cite[Example~5.115]{livro:2001}. The
generalized polynomials $H_{k}$ are the functions
$H_k:\mathbb{T}^2\rightarrow \mathbb{R}$, $k\in \mathbb{N}_0$,
defined recursively as follows:
\begin{equation*}
H_0(t,s)\equiv 1 \, , \quad H_{k+1}(t,s)=\int_s^t
H_k(\tau,s)\Delta\tau \, , \quad k = 1, 2, \ldots
\end{equation*}
for all $s,t\in \mathbb{T}$. If we let $H_k^\Delta(t,s)$ denote, for
each fixed $s$, the derivative of $H_k(t,s)$ with respect to $t$,
then (\textrm{cf.} \cite[p.~38]{livro:2001})
$$
H_k^\Delta(t,s)=H_{k-1}(t,s)\quad \text{for } k\in \mathbb{N}, \
t\in \mathbb{T}^\kappa \, .
$$

Let $a\in\mathbb{R}$ and $h>0$,
$(h\mathbb{Z})_a=\{a,a+h,a+2h,\ldots\}$, and $b=a+kh$ for some
$k\in\mathbb{N}$.

From now on we restrict ourselves to the time scale
$\mathbb{T}=(h\mathbb{Z})_a$, $h > 0$. Our main goal is to propose
and develop a discrete-time fractional variational theory in
$\mathbb{T}=(h\mathbb{Z})_a$. We borrow the notations from the
recent calculus of variations on time scales
\cite{CD:Bohner:2004,RD,J:B:M:08}. How to generalize our results to
an arbitrary time scale $\mathbb{T}$, with the graininess function
$\mu$ depending on time, is not clear and remains a challenging
question.

 We have $\sigma(t)=t+h$,
$\rho(t)=t-h$, $\mu(t) \equiv h$, and we will frequently write
$f^\sigma(t)=f(\sigma(t))$. We put $\mathbb{T}=[a,b]\cap
(h\mathbb{Z})_a$, so that $\mathbb{T}^\kappa=[a,\rho(b)]\cap
(h\mathbb{Z})_a$ and $\mathbb{T}^{\kappa^2}=[a,\rho^2(b)] \cap
(h\mathbb{Z})_a$.\\The delta derivative coincides in this case with
the forward $h$-difference:
$$\displaystyle{f^{\Delta}(t)=\frac{f^\sigma(t)-f(t)}{\mu(t)}}.$$\\If
$h=1$, then we have the usual discrete forward difference $\Delta
f(t)$.\\The delta integral gives the $h$-sum (or $h$-integral) of
$f$: $\displaystyle{\int_a^b f(t)\Delta
t=\sum_{k=\frac{a}{h}}^{\frac{b}{h}-1}f(kh)h}$. If we have a
function $f$ of two variables, $f(t,s)$, its partial forward
$h$-differences will be denoted by $\Delta_{t,h}$ and
$\Delta_{s,h}$, respectively. We make use of the standard
conventions $\sum_{t=c}^{c-1}f(t)=0$, $c\in\mathbb{Z}$, and
$\prod_{i=0}^{-1}f(i)=1$. Often, \emph{left fractional delta
integration} (resp., \emph{right fractional delta integration}) of
order $\nu>0$ is denoted by $_a\Delta_t^{-\nu}f(t)$ (resp.
$_t\Delta_b^{-\nu}f(t)$). Here, similarly as in Ross \textit{et
al.} \cite{Ross}, where the authors omit the subscript $t$ on the
operator (the operator itself cannot depend on $t$), we write
$_a\Delta_h^{-\nu}f(t)$ (resp. $_h\Delta_b^{-\nu}f(t)$).

Before giving an explicit formula for the generalized polynomials
$H_{k}$ on $h\mathbb{Z}$ we introduce the following definition:
\begin{definition}\label{powerHZ}
For arbitrary $x,y\in\mathbb{R}$ the $h$-factorial function is
defined by
\begin{equation*}
x_h^{(y)}:=h^y\frac{\Gamma(\frac{x}{h}+1)}{\Gamma(\frac{x}{h}+1-y)}\,
,
\end{equation*}
where $\Gamma$ is the well-known Euler gamma function, and we use
the convention that division at a pole yields zero.
\end{definition}

\begin{remark}
Before proposing Definition~\ref{powerHZ} we have tried other
possibilities. One that seemed most obvious is just to replace $1$ by
$h$ in formula \eqref{powerZ} because if we do $h=1$ on the time scale
$(h\mathbb{Z})_a$ we have the time scale $\mathbb{Z}$. This
possibility did not reveal a good choice.
\end{remark}

\begin{remark}
For $h = 1$, and in accordance with the previous similar definition
on Chapter~\ref{chap3}, we write $x^{(y)}$ to denote $x_h^{(y)}$.
\end{remark}

\begin{prop}
\label{hZ::prop:d} For the time-scale $\mathbb{T}=(h\mathbb{Z})_a$
one has
\begin{equation}
\label{hZ::hn} H_{k}(t,s):=\frac{(t-s)_h^{(k)}}{k!}\quad\mbox{for
all}\quad s,t\in \mathbb{T} \text{ and } k\in \mathbb{N}_0 \, .
\end{equation}
\end{prop}

To prove \eqref{hZ::hn} we use the following technical lemma.
Throughout this chapter the basic property \eqref{Z::naosei9}
of the gamma function will be frequently used.

\begin{lemma}
\label{hZ::lem:tl} Let $s \in \mathbb{T}$. Then, for all $t \in
\mathbb{T}^\kappa$ one has
\begin{equation*}
\Delta_{t,h} \left\{\frac{(t-s)_h^{(k+1)}}{(k+1)!}\right\} =
\frac{(t-s)_h^{(k)}}{k!} \, .
\end{equation*}
\end{lemma}
\begin{proof}
The equality follows by direct computations:
\begin{equation*}
\begin{split}
\Delta_{t,h} &\left\{\frac{(t-s)_h^{(k+1)}}{(k+1)!}\right\}
=\frac{1}{h}\left\{\frac{(\sigma(t)-s)_h^{(k+1)}}{(k+1)!}-\frac{(t-s)_h^{(k+1)}}{(k+1)!}\right\}\\
&=\frac{h^{k+1}}{h(k+1)!}\left\{\frac{\Gamma((t+h-s)/h+1)}{\Gamma((t+h-s)/h+1-(k+1))}
-\frac{\Gamma((t-s)/h+1)}{\Gamma((t-s)/h+1-(k+1))}\right\}\\
&=\frac{h^k}{(k+1)!}\left\{\frac{((t-s)/h+1)\Gamma((t-s)/h+1)}{((t-s)/h-k)\Gamma((t-s)/h-k)}
-\frac{\Gamma((t-s)/h+1)}{\Gamma((t-s)/h-k)}\right\}\\
&=\frac{h^k}{k!}\left\{\frac{\Gamma((t-s)/h+1)}{\Gamma((t-s)/h+1-k)}\right\}
=\frac{(t-s)_h^{(k)}}{k!} \, .
\end{split}
\end{equation*}
\end{proof}
\begin{proof}(of Proposition~\ref{hZ::prop:d})
We proceed by mathematical induction. For $k=0$
$$
H_0(t,s)=\frac{1}{0!}h^0\frac{\Gamma(\frac{t-s}{h}+1)}{\Gamma(\frac{t-s}{h}+1-0)}
=\frac{\Gamma(\frac{t-s}{h}+1)}{\Gamma(\frac{t-s}{h}+1)}=1 \, .
$$
Assume that (\ref{hZ::hn}) holds for $k$ replaced by $m$. Then by
Lemma~\ref{hZ::lem:tl}
\begin{eqnarray*}
H_{m+1}(t,s) &=& \int_s^t H_m(\tau,s)\Delta\tau = \int_s^t
\frac{(\tau-s)_h^{(m)}}{m!} \Delta\tau =
\frac{(t-s)_h^{(m+1)}}{(m+1)!},
\end{eqnarray*}
which is (\ref{hZ::hn}) with $k$ replaced by $m+1$.
\end{proof}

Let $y_1(t),\ldots,y_n(t)$ be $n$ linearly independent solutions of
the linear homogeneous dynamic equation $y^{\Delta^n}=0$. From
Theorem~\ref{hZ::eqsol} we know that the solution of \eqref{hZ::IVP}
(with $L=\Delta^n$ and $t_0=a$) is
\begin{equation*}
y(t) = \Delta^{-n} f(t)=\int_a^t
\frac{(t-\sigma(s))_h^{(n-1)}}{\Gamma(n)}f(s)\Delta s\\
=\frac{1}{\Gamma(n)}\sum_{k=a/h}^{t/h-1} (t-\sigma(kh))_h^{(n-1)}
f(kh) h \, .
\end{equation*}
Since $y^{\Delta_i}(a)=0$, $i = 0,\ldots,n-1$, then we can write
that
\begin{equation}
\label{hZ::eq:derDh:int}
\begin{split}
\Delta^{-n} f(t) &= \frac{1}{\Gamma(n)}\sum_{k=a/h}^{t/h-n}
(t-\sigma(kh))_h^{(n-1)} f(kh) h \\
&=
\frac{1}{\Gamma(n)}\int_a^{\sigma(t-nh)}(t-\sigma(s))_h^{(n-1)}f(s)
\Delta s \, .
\end{split}
\end{equation}
Note that function $t \rightarrow (\Delta^{-n} f)(t)$ is defined for
$t=a+n h \mbox{ mod}(h)$ while function $t \rightarrow f(t)$ is
defined for $t=a \mbox{ mod}(h)$. Extending \eqref{hZ::eq:derDh:int}
to any positive real value $\nu$, and having as an analogy the
continuous left and right fractional derivatives \cite{Miller1}, we
define the left fractional $h$-sum and the right fractional $h$-sum
as follows. We denote by $\mathcal{F}_\mathbb{T}$ the set of all
real valued functions defined on a given time scale $\mathbb{T}$.

\begin{definition}
\label{hZ::def0} Let $a\in\mathbb{R}$, $h>0$, $b=a+kh$ with
$k\in\mathbb{N}$, and put $\mathbb{T}=[a,b]\cap(h\mathbb{Z})_a$.
Consider $f\in\mathcal{F}_\mathbb{T}$. The left and right fractional
$h$-sum of order $\nu>0$ are, respectively, the operators
$_a\Delta_h^{-\nu} : \mathcal{F}_\mathbb{T} \rightarrow
\mathcal{F}_{\tilde{\mathbb{T}}_\nu^+}$ and $_h\Delta_b^{-\nu} :
\mathcal{F}_\mathbb{T} \rightarrow
\mathcal{F}_{\tilde{\mathbb{T}}_\nu^-}$, $\tilde{\mathbb{T}}_\nu^\pm
= \{t \pm \nu h : t \in \mathbb{T}\}$, defined by
\begin{equation*}
\begin{split}
_a\Delta_h^{-\nu}f(t) &= \frac{1}{\Gamma(\nu)}\int_{a}^{\sigma(t-\nu
h)}(t-\sigma(s))_h^{(\nu-1)}f(s)\Delta s
=\frac{1}{\Gamma(\nu)}\sum_{k=\frac{a}{h}}^{\frac{t}{h}-\nu}(t-\sigma(kh))_h^{(\nu-1)}f(kh)h\\
_h\Delta_b^{-\nu}f(t) &= \frac{1}{\Gamma(\nu)}\int_{t+\nu
h}^{\sigma(b)}(s-\sigma(t))_h^{(\nu-1)}f(s)\Delta s
=\frac{1}{\Gamma(\nu)}\sum_{k=\frac{t}{h}+\nu}^{\frac{b}{h}}(kh-\sigma(t))_h^{(\nu-1)}f(kh)h.
\end{split}
\end{equation*}
\end{definition}

\begin{remark}
In Definition~\ref{hZ::def0} we are using summations with limits
that are reals. For example, the summation that appears in the
definition of operator $_a\Delta_h^{-\nu}$ has the following
meaning:
$$
\sum_{k = \frac{a}{h}}^{\frac{t}{h} - \nu} G(k) = G(a/h) + G(a/h+1)
+ G(a/h+2) + \cdots + G(t/h - \nu),
$$
where $t \in \{ a + \nu h, a + h + \nu h , a + 2 h + \nu h , \ldots,
\underbrace{a+kh}_b + \nu h\}$ with $k\in\mathbb{N}$.
\end{remark}

\begin{lemma}
Let $\nu>0$ be an arbitrary positive real number. For any $t \in
\mathbb{T}$ we have: (i) $\lim_{\nu\rightarrow
0}{_a}\Delta_h^{-\nu}f(t+\nu h)=f(t)$; (ii) $\lim_{\nu\rightarrow
0}{_h}\Delta_b^{-\nu}f(t-\nu h)=f(t)$.
\end{lemma}
\begin{proof}
Since
\begin{align*}
{_a}\Delta_h^{-\nu}f(t+\nu
h)&=\frac{1}{\Gamma(\nu)}\int_{a}^{\sigma(t)}(t+\nu
h-\sigma(s))_h^{(\nu-1)}f(s)\Delta s\\
&=\frac{1}{\Gamma(\nu)}\sum_{k=\frac{a}{h}}^{\frac{t}{h}}(t+\nu h-\sigma(kh))_h^{(\nu-1)}f(kh)h\\
&=h^{\nu}f(t)+\frac{\nu}{\Gamma(\nu+1)}\sum_{k=\frac{a}{h}}^{\frac{\rho(t)}{h}}(t+\nu
h-\sigma(kh))_h^{(\nu-1)}f(kh)h\, ,
\end{align*}
it follows that $\lim_{\nu\rightarrow 0}{_a}\Delta_h^{-\nu}f(t+\nu
h)=f(t)$. The proof of (ii) is similar.
\end{proof}

For any $t\in\mathbb{T}$ and for any $\nu\geq 0$ we define
$_a\Delta_h^{0}f(t) := {_h}\Delta_b^{0}f(t) := f(t)$ and write
\begin{equation}
\label{hZ::seila1}
\begin{gathered}
{_a}\Delta_h^{-\nu}f(t+\nu h) = h^\nu f(t)
+\frac{\nu}{\Gamma(\nu+1)}\int_{a}^{t}(t+\nu
h-\sigma(s))_h^{(\nu-1)}f(s)\Delta s\, , \\
{_h}\Delta_b^{-\nu}f(t)=h^\nu f(t-\nu h) +
\frac{\nu}{\Gamma(\nu+1)}\int_{\sigma(t)}^{\sigma(b)}(s+\nu
h-\sigma(t))_h^{(\nu-1)}f(s)\Delta s \, .
\end{gathered}
\end{equation}

\begin{theorem}\label{hZ::thm2}
Let $f\in\mathcal{F}_\mathbb{T}$ and $\nu\geq0$. For all
$t\in\mathbb{T}^\kappa$ we have
\begin{equation}
\label{hZ::naosei1} {_a}\Delta_{h}^{-\nu} f^{\Delta}(t+\nu
h)=(_a\Delta_h^{-\nu}f(t+\nu h))^{\Delta} -\frac{\nu}{\Gamma(\nu +
1)}(t+\nu h-a)_h^{(\nu-1)}f(a) \, .
\end{equation}
\end{theorem}
To prove Theorem~\ref{hZ::thm2} we make use of a technical lemma:
\begin{lemma}
\label{hZ::lemma:tl} Let $t\in\mathbb{T}^\kappa$. The following
equality holds for all $s\in\mathbb{T}^\kappa$:
\begin{multline}
\label{hZ::proddiff}
\Delta_{s,h}\left((t+\nu h-s)_h^{(\nu-1)}f(s))\right)\\
=(t+\nu h-\sigma(s))_h^{(\nu-1)}f^{\Delta}(s) -(v-1)(t+\nu
h-\sigma(s))_h^{(\nu-2)}f(s) \, .
\end{multline}
\end{lemma}
\begin{proof}
Direct calculations give the intended result:
\begin{equation*}
\begin{split}
\Delta&_{s,h} \left((t+\nu h-s)_h^{(\nu-1)}f(s)\right)\\
&=\Delta_{s,h}\left((t+\nu h-s)_h^{(\nu-1)}\right)f(s)+\left(t+\nu h
-\sigma(s)\right)_h^{(\nu-1)}f^{\Delta}(s)\\
&=\frac{f(s)}{h}\left[h^{\nu-1}\frac{\Gamma\left(\frac{t+\nu
h-\sigma(s)}{h}+1\right)}{\Gamma\left(\frac{t+\nu
h-\sigma(s)}{h}+1-(\nu-1)\right)}-h^{\nu-1}\frac{\Gamma\left(\frac{t+\nu
h-s}{h}+1\right)}{\Gamma\left(\frac{t+\nu
h-s}{h}+1-(\nu-1)\right)}\right]\\
&\qquad +\left(t+\nu h -
\sigma(s)\right)_h^{(\nu-1)}f^{\Delta}(s)\\
&=f(s)\left[h^{\nu-2}\left[\frac{\Gamma(\frac{t+\nu
h-s}{h})}{\Gamma(\frac{t-s}{h}+1)}-\frac{\Gamma(\frac{t+\nu
h-s}{h}+1)}{\Gamma(\frac{t-s}{h}+2)}\right]\right]+\left(t+\nu h -
\sigma(s)\right)_h^{(\nu-1)}f^{\Delta}(s)\\
&=f(s)h^{\nu-2}\frac{\Gamma(\frac{t+\nu
h-s-h}{h}+1)}{\Gamma(\frac{t-s+\nu h-h}{h}+1-(\nu-2))}(-(\nu-1))+
\left(t+\nu h - \sigma(s)\right)_h^{(\nu-1)}f^{\Delta}(s)\\
&=-(\nu-1)(t+\nu h -\sigma(s))_h^{(\nu-2)}f(s)+\left(t+\nu h -
\sigma(s)\right)_h^{(\nu-1)}f^{\Delta}(s) \, ,
\end{split}
\end{equation*}
where the first equality follows directly from
\eqref{produto}.
\end{proof}

\begin{remark}
Given an arbitrary $t\in\mathbb{T}^\kappa$ it is easy to prove, in a
similar way as in the proof of Lemma~\ref{hZ::lemma:tl}, the
following equality analogous to \eqref{hZ::proddiff}: for all
$s\in\mathbb{T}^\kappa$
\begin{multline}
\label{hZ::eq:semlhante}
\Delta_{s,h}\left((s+\nu h-\sigma(t))_h^{(\nu-1)}f(s))\right)\\
=(\nu-1)(s+\nu h-\sigma(t))_h^{(\nu-2)}f^\sigma(s) + (s+\nu
h-\sigma(t))_h^{(\nu-1)}f^{\Delta}(s) \, .
\end{multline}
\end{remark}
\begin{proof}(of Theorem~\ref{hZ::thm2})
From Lemma~\ref{hZ::lemma:tl} we obtain that
\begin{equation}
\label{hZ::naosei}
\begin{split}
{_a}\Delta_{h}^{-\nu} & f^{\Delta}(t+\nu h) = h^\nu
f^\Delta(t)+\frac{\nu}{\Gamma(\nu+1)}\int_{a}^{t}(t+\nu
h-\sigma(s))_h^{(\nu-1)}f^{\Delta}(s)\Delta s\\
&=h^\nu f^\Delta(t)+\frac{\nu}{\Gamma(\nu+1)}\left[(t+\nu
h-s)_h^{(\nu-1)}f(s)\right]_{s=a}^{s=t}\\
&\qquad +\frac{\nu}{\Gamma(\nu+1)}\int_{a}^{\sigma(t)}(\nu-1)(t+\nu
h-\sigma(s))_h^{(\nu-2)}
f(s)\Delta s\\
&=-\frac{\nu(t+\nu h-a)_h^{(\nu-1)}}{\Gamma(\nu+1)}f(a)
+h^{\nu}f^\Delta(t)+\nu h^{\nu-1}f(t)\\
&\qquad +\frac{\nu}{\Gamma(\nu+1)}\int_{a}^{t}(\nu-1)(t+\nu
h-\sigma(s))_h^{(\nu-2)} f(s)\Delta s.
\end{split}
\end{equation}
We now show that $(_a\Delta_h^{-\nu}f(t+\nu h))^{\Delta}$ equals
\eqref{hZ::naosei}:
\begin{equation*}
\begin{split}
(_a\Delta_h^{-\nu} & f(t+\nu h))^\Delta = \frac{1}{h}\left[h^\nu
f(\sigma(t))+\frac{\nu}{\Gamma(\nu+1)}\int_{a}^{\sigma(t)}(\sigma(t)+\nu
h-\sigma(s))_h^{(\nu-1)}
f(s)\Delta s\right.\\
&\qquad \left.-h^\nu f(t)-\frac{\nu}{\Gamma(\nu+1)}\int_{a}^{t}(t
+\nu h-\sigma(s))_h^{(\nu-1)} f(s)\Delta s\right]\\
&=h^\nu
f^\Delta(t)+\frac{\nu}{h\Gamma(\nu+1)}\left[\int_{a}^{t}(\sigma(t)+\nu
h-\sigma(s))_h^{(\nu-1)}
f(s)\Delta s\right.\\
&\qquad\left.-\int_{a}^{t}(t+\nu h-\sigma(s))_h^{(\nu-1)}
f(s)\Delta s\right]+h^{\nu-1}\nu f(t)\\
&=h^\nu
f^\Delta(t)+\frac{\nu}{\Gamma(\nu+1)}\int_{a}^{t}\Delta_{t,h}\left((t+\nu
h -\sigma(s))_h^{(\nu-1)}
\right)f(s)\Delta s+h^{\nu-1}\nu f(t)\\
&=h^\nu
f^\Delta(t)+\frac{\nu}{\Gamma(\nu+1)}\int_{a}^{t}(\nu-1)(t+\nu
h-\sigma(s))_h^{(\nu-2)} f(s)\Delta s+\nu h^{\nu-1}f(t) \, .
\end{split}
\end{equation*}
\end{proof}
Follows the counterpart of Theorem~\ref{hZ::thm2} for the right
fractional $h$-sum:
\begin{theorem}\label{hZ::thm3} Let $f\in\mathcal{F}_\mathbb{T}$ and $\nu\geq 0$.
For all $t\in\mathbb{T}^\kappa$ we have
\begin{equation}
\label{hZ::naosei12} {_h}\Delta_{\rho(b)}^{-\nu} f^{\Delta}(t-\nu
h)=\frac{\nu}{\Gamma(\nu+1)}(b+\nu
h-\sigma(t))_h^{(\nu-1)}f(b)+(_h\Delta_b^{-\nu}f(t-\nu h))^{\Delta}
\, .
\end{equation}
\end{theorem}
\begin{proof}
From \eqref{hZ::eq:semlhante} we obtain from integration by parts
(item 2 of Lemma~\ref{integracao:partes}) that
\begin{equation}
\label{hZ::naosei99}
\begin{split}
{_h}\Delta_{\rho(b)}^{-\nu} & f^{\Delta}(t-\nu h) =\frac{\nu(b+\nu
h-\sigma(t))_h^{(\nu-1)}}{\Gamma(\nu+1)}f(b)
+ h^\nu f^\Delta(t) -\nu h^{\nu-1}f(\sigma(t))\\
&\qquad -\frac{\nu}{\Gamma(\nu+1)}\int_{\sigma(t)}^{b}(\nu-1)(s+\nu
h-\sigma(t))_h^{(\nu-2)} f^\sigma(s)\Delta s.
\end{split}
\end{equation}
We show that $(_h\Delta_b^{-\nu}f(t-\nu h))^{\Delta}$ equals \eqref{hZ::naosei99}:
\begin{equation*}
\begin{split}
(_h&\Delta_b^{-\nu} f(t-\nu h))^{\Delta}\\
&=h^{\nu}f^\Delta(t)+\frac{\nu}{h\Gamma(\nu+1)}\left[\int_{\sigma^2(t)}^{\sigma(b)}(s+\nu
h-\sigma^2(t)))_h^{(\nu-1)}
f(s)\Delta s\right.\\
&\qquad \left.-\int_{\sigma^2(t)}^{\sigma(b)}(s+\nu
h-\sigma(t))_h^{(\nu-1)}
f(s)\Delta s\right]-\nu h^{\nu-1} f(\sigma(t))\\
&=h^{\nu}f^\Delta(t)+\frac{\nu}{\Gamma(\nu+1)}\int_{\sigma^2(t)}^{\sigma(b)}\Delta_{t,h}\left((s+\nu
h-\sigma(t))_h^{(\nu-1)}\right)
f(s)\Delta s-\nu h^{\nu-1} f(\sigma(t))
\end{split}
\end{equation*}
\begin{equation*}
\begin{split}
&=h^{\nu}f^\Delta(t)-\frac{\nu}{\Gamma(\nu+1)}\int_{\sigma^2(t)}^{\sigma(b)}(\nu-1)(s+\nu
h-\sigma^2(t))_h^{(\nu-2)}
f(s)\Delta s-\nu h^{\nu-1} f(\sigma(t))\\
&=h^{\nu}f^\Delta(t)-\frac{\nu}{\Gamma(\nu+1)}\int_{\sigma(t)}^{b}(\nu-1)(s+\nu
h-\sigma(t))_h^{(\nu-2)} f(s)\Delta s-\nu h^{\nu-1} f(\sigma(t)).
\end{split}
\end{equation*}
\end{proof}

\begin{definition}
\label{hZ::def1} Let $0<\alpha\leq 1$ and set $\gamma := 1-\alpha$.
The \emph{left fractional difference} $_a\Delta_h^\alpha f(t)$ and
the \emph{right fractional difference} $_h\Delta_b^\alpha f(t)$ of
order $\alpha$ of a function $f\in\mathcal{F}_\mathbb{T}$ are
defined as
\begin{equation*}
_a\Delta_h^\alpha f(t) := (_a\Delta_h^{-\gamma}f(t+\gamma
h))^{\Delta}\ \text{ and } \ _h\Delta_b^\alpha
f(t):=-(_h\Delta_b^{-\gamma}f(t-\gamma h))^{\Delta}
\end{equation*}
for all $t\in\mathbb{T}^\kappa$.
\end{definition}


\section{Main Results}
\label{hZ::sec1}

Our aim is to introduce the $h$-fractional problem of the calculus
of variations and to prove corresponding necessary optimality
conditions. In order to obtain an Euler-Lagrange type equation
(\textrm{cf.} Theorem~\ref{hZ::thm0}) we first prove a fractional
formula of $h$-summation by parts.


\subsection{Fractional $h$-summation by parts}

A big challenge was to discover a fractional $h$-summation by parts
formula within the time scale setting. Indeed, there is no clue of
what such a formula should be. We found it eventually, making use of
the following lemma.

\begin{lemma}
\label{hZ::lem1} Let $f$ and $k$ be two functions defined on
$\mathbb{T}^\kappa$ and $\mathbb{T}^{\kappa^2}$, respectively, and
$g$ a function defined on
$\mathbb{T}^\kappa\times\mathbb{T}^{\kappa^2}$. The following
equality holds:
\begin{equation*}
\int_{a}^{b}f(t)\left[\int_{a}^{t}g(t,s)k(s)\Delta s\right]\Delta
t=\int_{a}^{\rho(b)}k(t)\left[\int_{\sigma(t)}^{b}g(s,t)f(s)\Delta
s\right]\Delta t \, .
\end{equation*}
\end{lemma}
\begin{proof}
Consider the matrices $R = \left[ f(a+h), f(a+2h), \cdots, f(b-h)
\right]$,
\begin{equation*}
C_1 = \left[
\begin{array}{c}
g(a+h,a)k(a) \\
g(a+2h,a)k(a)+g(a+2h,a+h)k(a+h) \\
\vdots \\
g(b-h,a)k(a)+g(b-h,a+h)k(a+h)+\cdots+ g(b-h,b-2h)k(b-2h)
\end{array}
\right],
\end{equation*}
\begin{gather*}
C_2 = \left[
\begin{array}{c}
g(a+h,a) \\
g(a+2h,a) \\
\vdots \\
g(b-h,a)  \end{array} \right], \ \  C_3 = \left[
\begin{array}{c}
0 \\
g(a+2h,a+h) \\
\vdots \\
g(b-h,a+h)  \end{array} \right], \ \ C_k = \left[
\begin{array}{c}
0 \\
0 \\
\vdots \\
g(b-h,b-2h)
\end{array}
\right] .
\end{gather*}
Direct calculations show that
\begin{equation*}
\begin{split}
\int_{a}^{b}&f(t)\left[\int_{a}^{t}g(t,s)k(s)\Delta s\right]\Delta t
=h^2\sum_{i=a/h}^{b/h-1} f(ih)\sum_{j=a/h}^{i-1}g(ih,jh)k(jh) = h^2 R \cdot C_1\\
&=h^2 R \cdot \left[k(a) C_2 + k(a+h)C_3 +\cdots +k(b-2h) C_k \right]\\
&=h^2\left[k(a)\sum_{j=a/h+1}^{b/h-1}g(jh,a)f(jh)+k(a+h)\sum_{j=a/h+2}^{b/h-1}g(jh,a+h)f(jh)\right.\\
&\left.\qquad \qquad +\cdots+k(b-2h)\sum_{j=b/h-1}^{b/h-1}g(jh,b-2h)f(jh)\right]\\
&=\sum_{i=a/h}^{b/h-2}k(ih)h\sum_{j=\sigma(ih)/h}^{b/h-1}g(jh,ih)f(jh)
h =\int_a^{\rho(b)}k(t)\left[\int_{\sigma(t)}^b g(s,t)f(s)\Delta
s\right]\Delta t.
\end{split}
\end{equation*}
\end{proof}

\begin{theorem}[fractional $h$-summation by parts]\label{hZ::teor1}
Let $f$ and $g$ be real valued functions defined on
$\mathbb{T}^\kappa$ and $\mathbb{T}$, respectively. Fix
$0<\alpha\leq 1$ and put $\gamma := 1-\alpha$. Then,
\begin{multline}
\label{hZ::delf:sumPart} \int_{a}^{b}f(t)_a\Delta_h^\alpha
g(t)\Delta t=h^\gamma f(\rho(b))g(b)-h^\gamma
f(a)g(a)+\int_{a}^{\rho(b)}{_h\Delta_{\rho(b)}^\alpha
f(t)g^\sigma(t)}\Delta t\\
+\frac{\gamma}{\Gamma(\gamma+1)}g(a)\left(\int_{a}^{b}(t+\gamma
h-a)_h^{(\gamma-1)}f(t)\Delta t -\int_{\sigma(a)}^{b}(t+\gamma
h-\sigma(a))_h^{(\gamma-1)}f(t)\Delta t\right).
\end{multline}
\end{theorem}
\begin{proof}
By \eqref{hZ::naosei1} we can write
\begin{equation}
\label{hZ::rui0}
\begin{split}
\int_{a}^{b} &f(t)_a\Delta_h^\alpha g(t)\Delta t
=\int_{a}^{b}f(t)(_a\Delta_h^{-\gamma} g(t+\gamma h))^{\Delta}\Delta t\\
&=\int_{a}^{b}f(t)\left[_a\Delta_h^{-\gamma}
g^{\Delta}(t+\gamma h)+\frac{\gamma}{\Gamma(\gamma
+1)}(t+\gamma h-a)_h^{(\gamma-1)}g(a)\right]\Delta t\\
&=\int_{a}^{b}f(t)_a\Delta_h^{-\gamma}g^{\Delta}(t+\gamma h)\Delta t
+\int_{a}^{b}\frac{\gamma}{\Gamma(\gamma+1)}(t+\gamma
h-a)_h^{(\gamma-1)}f(t)g(a)\Delta t.
\end{split}
\end{equation}
Using \eqref{hZ::seila1} we get
\begin{equation*}
\begin{split}
\int_{a}^{b} &f(t)_a\Delta_h^{-\gamma} g^{\Delta}(t+\gamma h) \Delta t\\
&=\int_{a}^{b}f(t)\left[h^\gamma g^{\Delta}(t) +
\frac{\gamma}{\Gamma(\gamma+1)}\int_{a}^{t}(t+\gamma h
-\sigma(s))_h^{(\gamma-1)} g^{\Delta}(s)\Delta s\right]\Delta t\\
&=h^\gamma\int_{a}^{b}f(t)g^{\Delta}(t)\Delta
t+\frac{\gamma}{\Gamma(\gamma+1)}\int_{a}^{\rho(b)}
g^{\Delta}(t)\int_{\sigma(t)}^{b}(s+\gamma h
-\sigma(t))_h^{(\gamma-1)}f(s)\Delta s \Delta t\\
&=h^\gamma f(\rho(b))[g(b)-g(\rho(b))]+\int_{a}^{\rho(b)}
g^{\Delta}(t)_h\Delta_{\rho(b)}^{-\gamma} f(t-\gamma h)\Delta t,
\end{split}
\end{equation*}
where the third equality follows by Lemma~\ref{hZ::lem1}. We proceed
to develop the right hand side of the last equality as follows:
\begin{equation*}
\begin{split}
h^\gamma & f(\rho(b))[g(b)-g(\rho(b))]+\int_{a}^{\rho(b)}
g^{\Delta}(t)_h\Delta_{\rho(b)}^{-\gamma} f(t-\gamma h)\Delta t\\
&=h^\gamma f(\rho(b))[g(b)-g(\rho(b))]
+\left[g(t)_h\Delta_{\rho(b)}^{-\gamma}
f(t-\gamma h)\right]_{t=a}^{t=\rho(b)}\\
&\quad -\int_{a}^{\rho(b)} g^\sigma(t)(_h\Delta_{\rho(b)}^{-\gamma} 
f(t-\gamma h))^{\Delta}\Delta t\\
&=h^\gamma f(\rho(b))g(b)-h^\gamma f(a)g(a)\\
&\quad
-\frac{\gamma}{\Gamma(\gamma+1)}g(a)\int_{\sigma(a)}^{b}(s+\gamma
h-\sigma(a))_h^{(\gamma-1)}f(s)\Delta s
+\int_{a}^{\rho(b)}{\left(_h\Delta_{\rho(b)}^\alpha
f(t)\right)g^\sigma(t)}\Delta t,
\end{split}
\end{equation*}
where the first equality follows from
Lemma~\ref{integracao:partes}. Putting this into
(\ref{hZ::rui0}) we get \eqref{hZ::delf:sumPart}.
\end{proof}


\subsection{Necessary optimality conditions}

We begin to fix two arbitrary real numbers $\alpha$ and $\beta$ such
that $\alpha,\beta\in(0,1]$. Further, we put $\gamma := 1-\alpha$
and $\nu :=1-\beta$.

Let a function
$L(t,u,v,w):\mathbb{T}^\kappa\times\mathbb{R}\times\mathbb{R}\times\mathbb{R}\rightarrow\mathbb{R}$
be given. We consider the problem of minimizing (or maximizing) a
functional $\mathcal{L}:\mathcal{F}_\mathbb{T}\rightarrow\mathbb{R}$
subject to given boundary conditions:
\begin{equation}
\label{hZ::naosei7}
\mathcal{L}(y(\cdot))=\int_{a}^{b}L(t,y^{\sigma}(t),{_a}\Delta_h^\alpha
y(t),{_h}\Delta_b^\beta y(t))\Delta t \longrightarrow \min, \
y(a)=A, \ y(b)=B \, .
\end{equation}
Our main aim is to derive necessary optimality conditions for
problem \eqref{hZ::naosei7}.
\begin{definition}
For $f\in\mathcal{F}_\mathbb{T}$ we define the norm
$$\|f\|=\max_{t\in\mathbb{T}^\kappa}|f^\sigma(t)|+\max_{t\in\mathbb{T}^\kappa}|_a\Delta_h^\alpha
f(t)|+\max_{t\in\mathbb{T}^\kappa}|_h\Delta_b^\beta f(t)|.$$ A
function $\hat{y}\in\mathcal{F}_\mathbb{T}$ with $\hat{y}(a)=A$ and
$\hat{y}(b)=B$ is called a local minimizer for problem
\eqref{hZ::naosei7} provided there exists $\delta>0$ such that
$\mathcal{L}(\hat{y})\leq\mathcal{L}(y)$ for all
$y\in\mathcal{F}_\mathbb{T}$ with $y(a)=A$ and $y(b)=B$ and
$\|y-\hat{y}\|<\delta$.
\end{definition}

\begin{definition}
A function $\eta\in\mathcal{F}_\mathbb{T}$ is called an admissible
variation provided $\eta \neq 0$ and $\eta(a)=\eta(b)=0$.
\end{definition}

From now on we assume that the second-order partial derivatives
$L_{uu}$, $L_{uv}$, $L_{uw}$, $L_{vw}$, $L_{vv}$, and $L_{ww}$ exist
and are continuous.


\subsubsection{First order optimality condition}

Next theorem gives a first order necessary condition for problem
\eqref{hZ::naosei7}, \textrm{i.e.}, an Euler-Lagrange type equation
for the fractional $h$-difference setting.
\begin{theorem}[The $h$-fractional Euler-Lagrange equation for problem \eqref{hZ::naosei7}]
\label{hZ::thm0} If $\hat{y}\in\mathcal{F}_\mathbb{T}$ is a local
minimizer for problem \eqref{hZ::naosei7}, then the equality
\begin{equation}
\label{hZ::EL} L_u[\hat{y}](t) +{_h}\Delta_{\rho(b)}^\alpha
L_v[\hat{y}](t)+{_a}\Delta_h^\beta L_w[\hat{y}](t)=0
\end{equation}
holds for all $t\in\mathbb{T}^{\kappa^2}$ with operator $[\cdot]$
defined by $[y](s) =(s,y^{\sigma}(s),{_a}\Delta_s^\alpha
y(s),{_s}\Delta_b^\beta y(s))$.
\end{theorem}
\begin{proof}
Suppose that $\hat{y}(\cdot)$ is a local minimizer of
$\mathcal{L}[\cdot]$. Let $\eta(\cdot)$ be an arbitrarily fixed
admissible variation and define a function
$\Phi:\left(-\frac{\delta}{\|\eta(\cdot)\|},\frac{\delta}{\|\eta(\cdot)\|}\right)\rightarrow\mathbb{R}$
by
\begin{equation}
\label{hZ::fi}
\Phi(\varepsilon)=\mathcal{L}[\hat{y}(\cdot)+\varepsilon\eta(\cdot)].
\end{equation}
This function has a minimum at $\varepsilon=0$, so we must have
$\Phi'(0)=0$, i.e.,
$$\int_{a}^{b}\left[L_u[\hat{y}](t)\eta^\sigma(t)
+L_v[\hat{y}](t){_a}\Delta_h^\alpha\eta(t)
+L_w[\hat{y}](t){_h}\Delta_b^\beta\eta(t)\right]\Delta t=0,$$ which
we may write equivalently as
\begin{multline}
\label{hZ::rui3} h
L_u[\hat{y}](t)\eta^\sigma(t)|_{t=\rho(b)}+\int_{a}^{\rho(b)}L_u[\hat{y}](t)\eta^\sigma(t)\Delta
t +\int_{a}^{b}L_v[\hat{y}](t){_a}\Delta_h^\alpha\eta(t)\Delta
t\\+\int_{a}^{b}L_w[\hat{y}](t){_h}\Delta_b^\beta\eta(t)\Delta t=0.
\end{multline}
Using Theorem~\ref{hZ::teor1} and the fact that $\eta(a)=\eta(b)=0$,
we get
\begin{equation}
\label{hZ::naosei5}
\int_{a}^{b}L_v[\hat{y}](t){_a}\Delta_h^\alpha\eta(t)\Delta
t=\int_{a}^{\rho(b)}\left({_h}\Delta_{\rho(b)}^\alpha
\left(L_v[\hat{y}]\right)(t)\right)\eta^\sigma(t)\Delta t
\end{equation}
for the third term in \eqref{hZ::rui3}. Using \eqref{hZ::naosei12}
it follows that
\begin{equation}
\label{hZ::naosei4}
\begin{split}
\int_{a}^{b} & L_w[\hat{y}](t){_h}\Delta_b^\beta\eta(t)\Delta t\\
=&-\int_{a}^{b}L_w[\hat{y}](t)({_h}\Delta_b^{-\nu}\eta(t-\nu h))^{\Delta}\Delta t\\
=&-\int_{a}^{b}L_w[\hat{y}](t)\left[{_h}\Delta_{\rho(b)}^{-\nu}
\eta^{\Delta}(t-\nu h)-\frac{\nu}{\Gamma(\nu+1)}(b+\nu h-\sigma(t))_h^{(\nu-1)}\eta(b)\right]\Delta t\\
=&-\int_{a}^{b}L_w[\hat{y}](t){_h}\Delta_{\rho(b)}^{-\nu}
\eta^{\Delta}(t-\nu h)\Delta t
+\frac{\nu\eta(b)}{\Gamma(\nu+1)}\int_{a}^{b}(b+\nu
h-\sigma(t))_h^{(\nu-1)}L_w[\hat{y}](t)\Delta t .
\end{split}
\end{equation}
We now use Lemma~\ref{hZ::lem1} to get
\begin{equation}
\label{hZ::naosei2}
\begin{split}
\int_{a}^{b} &L_w[\hat{y}](t){_h}\Delta_{\rho(b)}^{-\nu} \eta^{\Delta}(t-\nu h)\Delta t\\
&=\int_{a}^{b}L_w[\hat{y}](t)\left[h^\nu\eta^{\Delta}(t)\right.\\
&\quad+\left.\frac{\nu}{\Gamma(\nu+1)}\int_{\sigma(t)}^{b}(s+\nu
h-\sigma(t))_h^{(\nu-1)} \eta^{\Delta}(s)\Delta s\right]\Delta t\\
&=\int_{a}^{b}h^\nu L_w[\hat{y}](t)\eta^{\Delta}(t)\Delta t\\
&\qquad
+\frac{\nu}{\Gamma(\nu+1)}\int_{a}^{\rho(b)}\left[L_w[\hat{y}](t)\int_{\sigma(t)}^{b}(s+\nu
h-\sigma(t))_h^{(\nu-1)} \eta^{\Delta}(s)\Delta s\right]\Delta t\\
&=\int_{a}^{b}h^\nu L_w[\hat{y}](t)\eta^{\Delta}(t)\Delta t\\
&\qquad
+\frac{\nu}{\Gamma(\nu+1)}\int_{a}^{b}\left[\eta^{\Delta}(t)\int_{a}^{t}(t+\nu
h
-\sigma(s))_h^{(\nu-1)}L_w[\hat{y}](s)\Delta s\right]\Delta t\\
&=\int_{a}^{b}\eta^{\Delta}(t){_a}\Delta^{-\nu}_h
\left(L_w[\hat{y}]\right)(t+\nu h)\Delta t.
\end{split}
\end{equation}
We apply again the time scale integration by parts formula
(Lemma~\ref{integracao:partes}), this time to
\eqref{hZ::naosei2}, to obtain,
\begin{equation}
\label{hZ::naosei3}
\begin{split}
\int_{a}^{b} & \eta^{\Delta}(t){_a}\Delta^{-\nu}_h
\left(L_w[\hat{y}]\right)(t+\nu h)\Delta t\\
&=\int_{a}^{\rho(b)}\eta^{\Delta}(t){_a}\Delta^{-\nu}_h
\left(L_w[\hat{y}]\right)(t+\nu h)\Delta t\\
&\qquad +(\eta(b)-\eta(\rho(b))){_a}\Delta^{-\nu}_h
\left(L_w[\hat{y}]\right)(t+\nu h)|_{t=\rho(b)}\\
&=\left[\eta(t){_a}\Delta^{-\nu}_h \left(L_w[\hat{y}]\right)(t+\nu
h)\right]_{t=a}^{t=\rho(b)}
-\int_{a}^{\rho(b)}\eta^\sigma(t)({_a}\Delta^{-\nu}_h
\left(L_w[\hat{y}]\right)(t+\nu h))^\Delta \Delta t\\
&\qquad +\eta(b){_a}\Delta^{-\nu}_h \left(L_w[\hat{y}]\right)(t+\nu
h)|_{t=\rho(b)}-\eta(\rho(b)){_a}\Delta^{-\nu}_h
\left(L_w[\hat{y}]\right)(t+\nu h)|_{t=\rho(b)}\\
&=\eta(b){_a}\Delta^{-\nu}_h \left(L_w[\hat{y}]\right)(t+\nu
h)|_{t=\rho(b)}-\eta(a){_a}\Delta^{-\nu}_h
\left(L_w[\hat{y}]\right)(t+\nu h)|_{t=a}\\
&\qquad -\int_{a}^{\rho(b)}\eta^\sigma(t){_a}\Delta^{\beta}_h
\left(L_w[\hat{y}]\right)(t)\Delta t.
\end{split}
\end{equation}
Since $\eta(a)=\eta(b)=0$ we obtain, from \eqref{hZ::naosei2} and
\eqref{hZ::naosei3}, that
$$\int_{a}^{b}L_w[\hat{y}](t){_h}\Delta_{\rho(b)}^{-\nu}
\eta^\Delta(t)\Delta t
=-\int_{a}^{\rho(b)}\eta^\sigma(t){_a}\Delta^{\beta}_h
\left(L_w[\hat{y}]\right)(t)\Delta t\, ,$$ and after inserting in
\eqref{hZ::naosei4}, that
\begin{equation}
\label{hZ::naosei6}
\int_{a}^{b}L_w[\hat{y}](t){_h}\Delta_b^\beta\eta(t)\Delta t
=\int_{a}^{\rho(b)}\eta^\sigma(t){_a}\Delta^{\beta}_h
\left(L_w[\hat{y}]\right)(t) \Delta t.
\end{equation}
By \eqref{hZ::naosei5} and \eqref{hZ::naosei6} we may write
\eqref{hZ::rui3} as
$$\int_{a}^{\rho(b)}\left[L_u[\hat{y}](t)
+{_h}\Delta_{\rho(b)}^\alpha
\left(L_v[\hat{y}]\right)(t)+{_a}\Delta_h^\beta
\left(L_w[\hat{y}]\right)(t)\right]\eta^\sigma(t) \Delta t =0\, .$$
Since the values of $\eta^\sigma(t)$ are arbitrary for
$t\in\mathbb{T}^{\kappa^2}$, the Euler-Lagrange equation
\eqref{hZ::EL} holds along $\hat{y}$.
\end{proof}

The next result is a direct corollary of Theorem~\ref{hZ::thm0}.

\begin{cor}[The $h$-Euler-Lagrange equation
-- \textrm{cf.}, \textrm{e.g.}, \cite{CD:Bohner:2004,RD}]
\label{hZ::ELCor} Let $\mathbb{T}$ be the time scale $h \mathbb{Z}$,
$h > 0$, with the forward jump operator $\sigma$ and the delta
derivative $\Delta$. Assume $a, b \in \mathbb{T}$, $a < b$. If
$\hat{y}$ is a solution to the problem
\begin{equation*}
\mathcal{L}(y(\cdot))=\int_{a}^{b}L(t,y^{\sigma}(t),y^\Delta(t))\Delta
t \longrightarrow \min, \  y(a)=A, \  y(b)=B\, ,
\end{equation*}
then the equality
$L_u(t,\hat{y}^{\sigma}(t),\hat{y}^\Delta(t))
-\left(L_v(t,\hat{y}^{\sigma}(t),\hat{y}^\Delta(t))\right)^\Delta
=0$ holds for all $t\in\mathbb{T}^{\kappa^2}$.
\end{cor}
\begin{proof}
Choose $\alpha=1$ and a $L$ that does not depend on $w$ in
Theorem~\ref{hZ::thm0}.
\end{proof}

\begin{remark}
If we take $h=1$ in Corollary~\ref{hZ::ELCor} we have that
$$L_u(t,\hat{y}^{\sigma}(t),\Delta\hat{y}(t))
-\Delta L_v(t,\hat{y}^{\sigma}(t),\Delta\hat{y}(t)) =0$$
holds for all $t\in\mathbb{T}^{\kappa^2}$. This equation is usually
called \emph{the discrete Euler-Lagrange equation}, and can be
found, \textrm{e.g.}, in \cite[Chap.~8]{book:DCV}.
\end{remark}


\subsubsection{Natural boundary conditions}

If the initial condition $y(a)=A$ is not present in problem
\eqref{hZ::naosei7} (\textrm{i.e.}, $y(a)$ is free), besides the
$h$-fractional Euler-Lagrange equation \eqref{hZ::EL} the following
supplementary condition must be fulfilled:
\begin{multline}\label{hZ::rui1}
-h^\gamma L_v[\hat{y}](a)+\frac{\gamma}{\Gamma(\gamma+1)}\left(
\int_{a}^{b}(t+\gamma h-a)_h^{(\gamma-1)}L_v[\hat{y}](t)\Delta t\right.\\
\left.-\int_{\sigma(a)}^{b}(t+\gamma
h-\sigma(a))_h^{(\gamma-1)}L_v[\hat{y}](t)\Delta t\right)+
L_w[\hat{y}](a)=0.
\end{multline}
Similarly, if $y(b)=B$ is not present in \eqref{hZ::naosei7} ($y(b)$
is free), the extra condition
\begin{multline}\label{hZ::rui2}
h L_u[\hat{y}](\rho(b))+h^\gamma L_v[\hat{y}](\rho(b))-h^\nu L_w[\hat{y}](\rho(b))\\
+\frac{\nu}{\Gamma(\nu+1)}\left(\int_{a}^{b}(b+\nu
h-\sigma(t))_h^{(\nu-1)}L_w[\hat{y}](t)\Delta t \right.\\ \left.
-\int_{a}^{\rho(b)}(\rho(b)+\nu
h-\sigma(t))_h^{(\nu-1)}L_w[\hat{y}](t)\Delta t\right)=0
\end{multline}
is added to Theorem~\ref{hZ::thm0}. We leave the proof of the
\emph{natural boundary conditions} \eqref{hZ::rui1} and
\eqref{hZ::rui2} to the reader. We just note here that the first
term in \eqref{hZ::rui2} arises from the first term of the left hand
side of \eqref{hZ::rui3}.


\subsubsection{Second order optimality condition}

We now obtain a second order necessary condition for problem
\eqref{hZ::naosei7}, \textrm{i.e.}, we prove a Legendre optimality
type condition for the fractional $h$-difference setting.
\begin{theorem}[The $h$-fractional Legendre necessary condition]
\label{hZ::thm1} If $\hat{y}\in\mathcal{F}_\mathbb{T}$ is a local
minimizer for problem \eqref{hZ::naosei7}, then the inequality
\begin{equation}
\label{hZ::eq:LC}
\begin{split}
h^2
&L_{uu}[\hat{y}](t)+2h^{\gamma+1}L_{uv}[\hat{y}](t)+2h^{\nu+1}(\nu-1)L_{uw}[\hat{y}](t)
+h^{2\gamma}(\gamma -1)^2 L_{vv}[\hat{y}](\sigma(t))\\
&+2h^{\nu+\gamma}(\gamma-1)L_{vw}[\hat{y}](\sigma(t))
+2h^{\nu+\gamma}(\nu-1)L_{vw}[\hat{y}](t)+h^{2\nu}(\nu-1)^2 L_{ww}[\hat{y}](t)\\
&+h^{2\nu}L_{ww}[\hat{y}](\sigma(t))
+\int_{a}^{t}h^3L_{ww}[\hat{y}](s)\left(\frac{\nu(1-\nu)}{\Gamma(\nu+1)}(t+\nu
h - \sigma(s))_h^{(\nu-2)}\right)^2\Delta s\\
&+h^{\gamma}L_{vv}[\hat{y}](t)
+\int_{\sigma(\sigma(t))}^{b}h^3L_{vv}[\hat{y}](s)\left(\frac{\gamma(\gamma-1)}{\Gamma(\gamma+1)}(s+\gamma
h -\sigma(\sigma(t)))_h^{(\gamma-2)}\right)^2\Delta s \geq 0
\end{split}
\end{equation}
holds for all $t\in\mathbb{T}^{\kappa^2}$, where
$[\hat{y}](t)=(t,\hat{y}^{\sigma}(t),{_a}\Delta_t^\alpha
\hat{y}(t),{_t}\Delta_b^\beta\hat{y}(t))$.
\end{theorem}
\begin{proof}
By the hypothesis of the theorem, and letting $\Phi$ be as in
\eqref{hZ::fi}, we have as necessary optimality condition that
$\Phi''(0)\geq 0$ for an arbitrary admissible variation
$\eta(\cdot)$. Inequality $\Phi''(0)\geq 0$ is equivalent to
\begin{multline}
\label{hZ::des1}
\int_{a}^{b}\left[L_{uu}[\hat{y}](t)(\eta^\sigma(t))^2
+2L_{uv}[\hat{y}](t)\eta^\sigma(t){_a}\Delta_h^\alpha\eta(t)
+2L_{uw}[\hat{y}](t)\eta^\sigma(t){_h}\Delta_b^\beta\eta(t)\right.\\
\left. +L_{vv}[\hat{y}](t)({_a}\Delta_h^\alpha\eta(t))^2
+2L_{vw}[\hat{y}](t){_a}\Delta_h^\alpha\eta(t){_h}\Delta_b^\beta\eta(t)
+L_{ww}(t)({_h}\Delta_b^\beta\eta(t))^2\right]\Delta t\geq 0.
\end{multline}
Let $\tau\in\mathbb{T}^{\kappa^2}$ be arbitrary, and choose
$\eta:\mathbb{T}\rightarrow\mathbb{R}$ given by $\eta(t) = \left\{
\begin{array}{ll}
h & \mbox{if $t=\sigma(\tau)$};\\
0 & \mbox{otherwise}.\end{array} \right.$ It follows that
$\eta(a)=\eta(b)=0$, \textrm{i.e.}, $\eta$ is an admissible
variation. Using \eqref{hZ::naosei1} we get
\begin{equation*}
\begin{split}
\int_{a}^{b}&\left[L_{uu}[\hat{y}](t)(\eta^\sigma(t))^2
+2L_{uv}[\hat{y}](t)\eta^\sigma(t){_a}\Delta_h^\alpha\eta(t)
+L_{vv}[\hat{y}](t)({_a}\Delta_h^\alpha\eta(t))^2\right]\Delta t\\
&=\int_{a}^{b}\Biggl[L_{uu}[\hat{y}](t)(\eta^\sigma(t))^2\\
&\qquad\quad +2L_{uv}[\hat{y}](t)\eta^\sigma(t)\left(h^\gamma
\eta^{\Delta}(t)+
\frac{\gamma}{\Gamma(\gamma+1)}\int_{a}^{t}(t+\gamma h
-\sigma(s))_h^{(\gamma-1)}\eta^{\Delta}(s)\Delta s\right)\\
&\qquad\quad +L_{vv}[\hat{y}](t)\left(h^\gamma \eta^{\Delta}(t)
+\frac{\gamma}{\Gamma(\gamma+1)}\int_{a}^{t}(t+\gamma h
-\sigma(s))_h^{(\gamma-1)}\eta^{\Delta}(s)\Delta s\right)^2\Biggr]\Delta t\\
&=h^3L_{uu}[\hat{y}](\tau)+2h^{\gamma+2}L_{uv}[\hat{y}](\tau)+h^{\gamma+1}L_{vv}[\hat{y}](\tau)\\
&\quad
+\int_{\sigma(\tau)}^{b}L_{vv}[\hat{y}](t)\left(h^\gamma\eta^{\Delta}(t)
+\frac{\gamma}{\Gamma(\gamma+1)}\int_{a}^{t}(t+\gamma
h-\sigma(s))_h^{(\gamma-1)}\eta^{\Delta}(s)\Delta s\right)^2\Delta
t.
\end{split}
\end{equation*}
Observe that
\begin{multline*}
h^{2\gamma+1}(\gamma -1)^2 L_{vv}[\hat{y}](\sigma(\tau))\\
+\int_{\sigma^2(\tau)}^{b}L_{vv}[\hat{y}](t)\left(\frac{\gamma}{\Gamma(\gamma+1)}\int_{a}^{t}(t+\gamma
h-\sigma(s))_h^{(\gamma-1)}\eta^{\Delta}(s)\Delta s\right)^2\Delta t\\
=\int_{\sigma(\tau)}^{b}L_{vv}[\hat{y}](t)\left(h^\gamma
\eta^\Delta(t)+\frac{\gamma}{\Gamma(\gamma+1)}\int_{a}^{t}(t+\gamma
h-\sigma(s))_h^{(\gamma-1)}\eta^{\Delta}(s)\Delta s\right)^2\Delta
t.
\end{multline*}
Let $t\in[\sigma^2(\tau),\rho(b)]\cap h\mathbb{Z}$. Since
\begin{equation}
\label{hZ::rui10}
\begin{split}
\frac{\gamma}{\Gamma(\gamma+1)}&\int_{a}^{t}(t+\gamma h 
-\sigma(s))_h^{(\gamma-1)}\eta^{\Delta}(s)\Delta s\\
&= \frac{\gamma}{\Gamma(\gamma+1)}\left[
\int_{a}^{\sigma(\tau)}(t+\gamma h-\sigma(s))_h^{(\gamma-1)}\eta^{\Delta}(s)\Delta s\right.\\
&\qquad\qquad\qquad\qquad \left.+\int_{\sigma(\tau)}^{t}(t+\gamma h
-\sigma(s))_h^{(\gamma-1)}\eta^{\Delta}(s)\Delta s\right]\\
&=h\frac{\gamma}{\Gamma(\gamma+1)}\left[(t+\gamma h
-\sigma(\tau))_h^{(\gamma-1)}-(t+\gamma h-\sigma(\sigma(\tau)))_h^{(\gamma-1)}\right]\\
&=\frac{\gamma h^\gamma}{\Gamma(\gamma+1)}\left[
\frac{\left(\frac{t-\tau}{h}+\gamma-1\right)\Gamma\left(\frac{t-\tau}{h}+\gamma-1\right)
-\left(\frac{t-\tau}{h}\right)\Gamma\left(\frac{t-\tau}{h}+\gamma-1\right)}
{\left(\frac{t-\tau}{h}\right)\Gamma\left(\frac{t-\tau}{h}\right)}\right]\\
&=h^{2}\frac{\gamma(\gamma-1)}{\Gamma(\gamma+1)}(t+\gamma h
-\sigma(\sigma(\tau)))_h^{(\gamma-2)},
\end{split}
\end{equation}
we conclude that
\begin{multline*}
\int_{\sigma^2(\tau)}^{b}L_{vv}[\hat{y}](t)\left(\frac{\gamma}{\Gamma(\gamma+1)}\int_{a}^{t}(t
+\gamma h-\sigma(s))_h^{(\gamma-1)}\eta^{\Delta}(s)\Delta s\right)^2\Delta t\\
=\int_{\sigma^2(\tau)}^{b}L_{vv}[\hat{y}](t)\left(h^2\frac{\gamma(\gamma-1)}{\Gamma(\gamma+1)}(t
+\gamma h-\sigma^2(\tau))_h^{(\gamma-2)}\right)^2\Delta t.
\end{multline*}
Note that we can write
${_t}\Delta_b^\beta\eta(t)=-{_h}\Delta_{\rho(b)}^{-\nu}
\eta^\Delta(t-\nu h)$ because $\eta(b)=0$. It is not difficult to
see that the following equality holds:
\begin{equation*}
\begin{split}
\int_{a}^{b}2L_{uw}[\hat{y}](t)\eta^\sigma(t){_h}\Delta_b^\beta\eta(t)\Delta
t
&=-\int_{a}^{b}2L_{uw}[\hat{y}](t)\eta^\sigma(t){_h}\Delta_{\rho(b)}^{-\nu}
\eta^\Delta(t-\nu h)\Delta t\\
&=2h^{2+\nu}L_{uw}[\hat{y}](\tau)(\nu-1) \, .
\end{split}
\end{equation*}
Moreover,
\begin{equation*}
\begin{split}
\int_{a}^{b} &2L_{vw}[\hat{y}](t){_a}\Delta_h^\alpha\eta(t){_h}\Delta_b^\beta\eta(t)\Delta t\\
&=-2\int_{a}^{b}L_{vw}[\hat{y}](t)\left\{\left(h^\gamma\eta^{\Delta}(t)+\frac{\gamma}{\Gamma(\gamma+1)}
\cdot\int_{a}^{t}(t+\gamma h-\sigma(s))_h^{(\gamma-1)}\eta^{\Delta}(s)\Delta s\right)\right.\\
&\qquad\qquad
\left.\cdot\left[h^\nu\eta^{\Delta}(t)+\frac{\nu}{\Gamma(\nu+1)}\int_{\sigma(t)}^{b}(s
+\nu h-\sigma(t))_h^{(\nu-1)}\eta^{\Delta}(s)\Delta s\right]\right\}\Delta t\\
&=2h^{\gamma+\nu+1}(\nu-1)L_{vw}[\hat{y}](\tau)+2h^{\gamma+\nu+1}(\gamma-1)L_{vw}[\hat{y}](\sigma(\tau)).
\end{split}
\end{equation*}
Finally, we have that
\begin{equation*}
\begin{split}
&\int_{a}^{b} L_{ww}[\hat{y}](t)({_h}\Delta_b^\beta\eta(t))^2\Delta t\\
&=\int_{a}^{\sigma(\sigma(\tau))}L_{ww}[\hat{y}](t)\left[h^\nu\eta^{\Delta}(t)+\frac{\nu}{\Gamma(\nu+1)}
\int_{\sigma(t)}^{b}(s+\nu h-\sigma(t))_h^{(\nu-1)}\eta^{\Delta}(s)\Delta s\right]^2\Delta t\\
&=\int_{a}^{\tau}L_{ww}[\hat{y}](t)\left[\frac{\nu}{\Gamma(\nu+1)}\int_{\sigma(t)}^{b}(s
+\nu h-\sigma(t))_h^{(\nu-1)}\eta^{\Delta}(s)\Delta s\right]^2\Delta t\\
&\qquad +hL_{ww}[\hat{y}](\tau)(h^\nu-\nu h^\nu)^2+h^{2\nu+1}L_{ww}[\hat{y}](\sigma(\tau))\\
&=\int_{a}^{\tau}L_{ww}[\hat{y}](t)\left[h\frac{\nu}{\Gamma(\nu+1)}\left\{(\tau+\nu
h
-\sigma(t))_h^{(\nu-1)}-(\sigma(\tau)+\nu h-\sigma(t))_h^{(\nu-1)}\right\}\right]^2\\
&\qquad + hL_{ww}[\hat{y}](\tau)(h^\nu-\nu
h^\nu)^2+h^{2\nu+1}L_{ww}[\hat{y}](\sigma(\tau)).
\end{split}
\end{equation*}
Similarly as we did in \eqref{hZ::rui10}, we can prove that
\begin{multline*}
h\frac{\nu}{\Gamma(\nu+1)}\left\{(\tau+\nu
h-\sigma(t))_h^{(\nu-1)}-(\sigma(\tau)+\nu
h-\sigma(t))_h^{(\nu-1)}\right\}\\
=h^{2}\frac{\nu(1-\nu)}{\Gamma(\nu+1)}(\tau+\nu
h-\sigma(t))_h^{(\nu-2)}.
\end{multline*}
Thus, we have that inequality \eqref{hZ::des1} is equivalent to
\begin{multline}
\label{hZ::des2}
h\Biggl\{h^2L_{uu}[\hat{y}](t)+2h^{\gamma+1}L_{uv}[\hat{y}](t)
+h^{\gamma}L_{vv}[\hat{y}](t)+L_{vv}(\sigma(t))(\gamma h^\gamma-h^\gamma)^2\\
+\int_{\sigma(\sigma(t))}^{b}h^3L_{vv}(s)\left(\frac{\gamma(\gamma-1)}{\Gamma(\gamma+1)}(s
+\gamma h -\sigma(\sigma(t)))_h^{(\gamma-2)}\right)^2\Delta s\\
+2h^{\nu+1}L_{uw}[\hat{y}](t)(\nu-1)+2h^{\gamma+\nu}(\nu-1)L_{vw}[\hat{y}](t)\\
+2h^{\gamma+\nu}(\gamma-1)L_{vw}(\sigma(t))
+h^{2\nu}L_{ww}[\hat{y}](t)(1-\nu)^2+h^{2\nu}L_{ww}[\hat{y}](\sigma(t))\\
+\int_{a}^{t}h^3L_{ww}[\hat{y}](s)\left(\frac{\nu(1-\nu)}{\Gamma(\nu+1)}(t+\nu
h - \sigma(s))^{\nu-2}\right)^2\Delta s\Biggr\}\geq 0.
\end{multline}
Because $h>0$, \eqref{hZ::des2} is equivalent to \eqref{hZ::eq:LC}.
The theorem is proved.
\end{proof}

The next result is a simple corollary of Theorem~\ref{hZ::thm1}.
\begin{cor}[The $h$-Legendre necessary condition -- \textrm{cf.} 
Result~1.3 of \cite{CD:Bohner:2004}]
\label{hZ::CorDis:Bohner} Let $\mathbb{T}$ be the time scale $h
\mathbb{Z}$, $h > 0$, with the forward jump operator $\sigma$ and
the delta derivative $\Delta$. Assume $a, b \in \mathbb{T}$, $a <
b$. If $\hat{y}$ is a solution to the problem
\begin{equation*}
\mathcal{L}(y(\cdot))=\int_{a}^{b}L(t,y^{\sigma}(t),y^\Delta(t))\Delta
t \longrightarrow \min, \  y(a)=A, \ y(b)=B \, ,
\end{equation*}
then the inequality
\begin{equation}
\label{hZ::LNCBohner}
h^2L_{uu}[\hat{y}](t)+2hL_{uv}[\hat{y}](t)+L_{vv}[\hat{y}](t)+L_{vv}[\hat{y}](\sigma(t))
\geq 0
\end{equation}
holds for all $t\in\mathbb{T}^{\kappa^2}$, where
$[\hat{y}](t)=(t,\hat{y}^{\sigma}(t),\hat{y}^\Delta(t))$.
\end{cor}
\begin{proof}
Choose $\alpha=1$ and a Lagrangian $L$ that does not depend on $w$.
Then, $\gamma=0$ and the result follows immediately from
Theorem~\ref{hZ::thm1}.
\end{proof}

\begin{remark}
When $h$ goes to zero we have $\sigma(t) = t$ and inequality
\eqref{hZ::LNCBohner} coincides with Legendre's classical necessary
optimality condition $L_{vv}[\hat{y}](t) \ge 0$ (\textrm{cf.},
\textrm{e.g.}, \cite{vanBrunt}).
\end{remark}


\section{Examples}
\label{hZ::sec2}

In this section we present some illustrative examples.
\begin{example}
\label{hZ::ex:2} Let us consider the following problem:
\begin{equation}
\label{hZ::eq:ex2} \mathcal{L}(y)=\frac{1}{2} \int_{0}^{1}
\left({_0}\Delta_h^{\frac{3}{4}} y(t)\right)^2\Delta t
\longrightarrow \min \, , \quad y(0)=0 \, , \quad y(1)=1 \, .
\end{equation}
We consider (\ref{hZ::eq:ex2}) with different values of $h$.
Numerical results show that when $h$ tends to zero the
$h$-fractional Euler-Lagrange extremal tends to the fractional
continuous extremal: when $h \rightarrow 0$ (\ref{hZ::eq:ex2}) tends
to the fractional continuous variational problem in the
Riemann-Liouville sense studied in \cite[Example~1]{agr0}, with
solution given by
\begin{equation}
\label{hZ::solEx2}
y(t)=\frac{1}{2}\int_0^t\frac{dx}{\left[(1-x)(t-x)\right]^{\frac{1}{4}}}
\, .
\end{equation}
This is illustrated in Figure~\ref{hZ::Fig:2}.
\begin{figure}[!htbp]
\begin{center}
\includegraphics[scale=0.6]{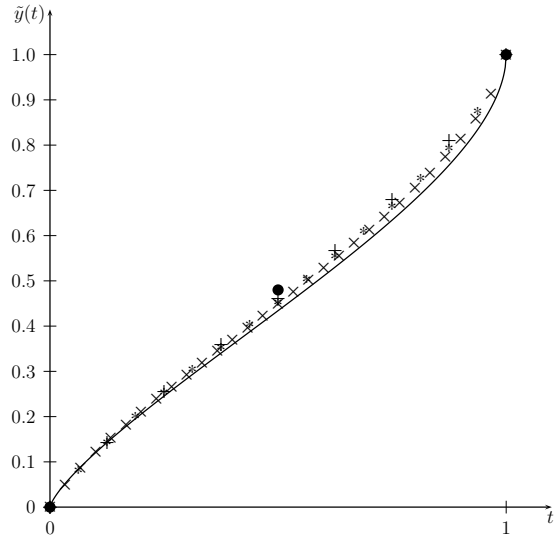}
  \caption{Extremal $\tilde{y}(t)$ for problem of Example~\ref{hZ::ex:2}
  with different values of $h$:
  $h=0.50$ ($\bullet$); $h=0.125$ ($+$);
  $h=0.0625$ ($\ast$); $h=1/30$ ($\times$).
  The continuous line represent function
  (\ref{hZ::solEx2}).}\label{hZ::Fig:2}
\end{center}
\end{figure}\\
In this example for each value of $h$ there is a unique
$h$-fractional Euler-Lagrange extremal, solution of \eqref{hZ::EL},
which always verifies the $h$-fractional Legendre necessary
condition \eqref{hZ::eq:LC}.
\end{example}

\begin{example}
\label{hZ::ex:1} Let us consider the following problem:
\begin{equation}
\label{hZ::eq:ex1} \mathcal{L}(y)=\int_{0}^{1}
\left[\frac{1}{2}\left({_0}\Delta_h^\alpha
y(t)\right)^2-y^{\sigma}(t)\right]\Delta t \longrightarrow \min \, ,
\quad y(0) = 0 \, , \quad y(1) = 0 \, .
\end{equation}
We begin by considering problem (\ref{hZ::eq:ex1}) with a fixed
value for $\alpha$ and different values of $h$. The extremals
$\tilde{y}$ are obtained using our Euler-Lagrange equation
(\ref{hZ::EL}). As in Example~\ref{hZ::ex:2} the numerical results
show that when $h$ tends to zero the extremal of the problem tends
to the extremal of the corresponding continuous fractional problem
of the calculus of variations in the Riemann-Liouville sense. More
precisely, when $h$ approximates zero problem (\ref{hZ::eq:ex1})
tends to the fractional continuous problem studied in
\cite[Example~2]{agr2}. For $\alpha=1$ and $h \rightarrow 0$ the
extremal of (\ref{hZ::eq:ex1}) is given by $y(t)=\frac{1}{2} t
(1-t)$, which coincides with the extremal of the classical problem
of the calculus of variations
\begin{equation*}
\mathcal{L}(y)=\int_{0}^{1} \left(\frac{1}{2} y'(t)^2-y(t)\right) dt
\longrightarrow \min \, , \quad y(0) = 0 \, , \quad y(1) = 0 \, .
\end{equation*}
This is illustrated in Figure~\ref{hZ::Fig:0} for $h =
\frac{1}{2^i}$, $i = 1, 2, 3, 4$.\\
\begin{figure}[!htbp]
\begin{center}
\includegraphics[scale=0.6]{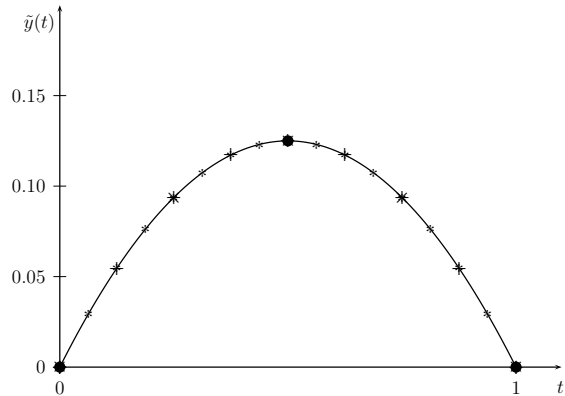}
  \caption{Extremal $\tilde{y}(t)$ for problem \eqref{hZ::eq:ex1}
  with $\alpha=1$ and different values of $h$:
  $h=0.5$ ($\bullet$); $h=0.25$ ($\times$);
  $h=0.125$ ($+$); $h=0.0625$ ($\ast$).}\label{hZ::Fig:0}
\end{center}
\end{figure}\\

In this example, for each value of $\alpha$ and $h$, we only have
one extremal (we only have one solution to (\ref{hZ::EL}) for each
$\alpha$ and $h$). Our Legendre condition \eqref{hZ::eq:LC} is
always verified along the extremals. Figure~\ref{hZ::Fig:1} shows
the extremals of problem \eqref{hZ::eq:ex1} for a fixed value of $h$
($h=1/20$) and different values of $\alpha$. The numerical results
show that when $\alpha$ tends to one the extremal tends to the
solution of the classical (integer order) discrete-time problem.\\
\end{example}
\begin{figure}[!htbp]
\begin{center}
\includegraphics[scale=0.6]{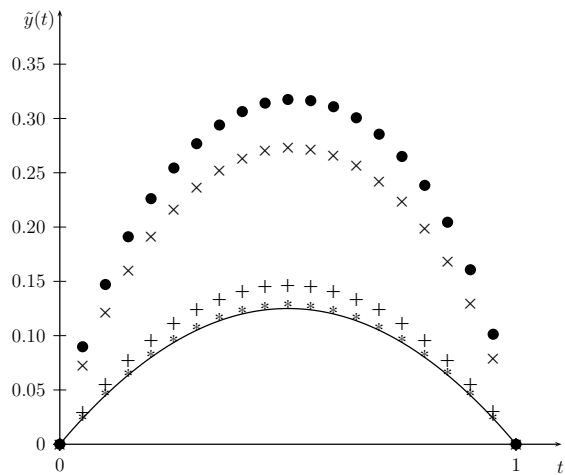}
  \caption{Extremal $\tilde{y}(t)$ for \eqref{hZ::eq:ex1}
with $h=0.05$ and different values of $\alpha$: $\alpha=0.70$
($\bullet$); $\alpha=0.75$ ($\times$); $\alpha=0.95$ ($+$);
$\alpha=0.99$ ($\ast$). The continuous line is $y(t)=\frac{1}{2} t
(1-t)$.}\label{hZ::Fig:1}
\end{center}
\end{figure}
Our last example shows that the $h$-fractional Legendre necessary
optimality condition can be a very useful tool. In
Example~\ref{hZ::ex:3} we consider a problem for which the
$h$-fractional Euler-Lagrange equation gives several candidates but
just a few of them verify the Legendre condition \eqref{hZ::eq:LC}.

\begin{example}
\label{hZ::ex:3} Let us consider the following problem:
\begin{equation}
\label{hZ::eq:ex3} \mathcal{L}(y)=\int_{a}^{b}
\left({_a}\Delta_h^\alpha
y(t)\right)^3+\theta\left({_h}\Delta_b^\alpha y(t)\right)^2\Delta t
\longrightarrow \min \, , \quad y(a)=0 \, , \quad y(b)=1 \, .
\end{equation}
For $\alpha=0.8$, $\beta=0.5$, $h=0.25$, $a=0$, $b=1$, and
$\theta=1$, problem (\ref{hZ::eq:ex3}) has eight different
Euler-Lagrange extremals. As we can see on
Table~\ref{hZ::candidates:ex3} only two of the candidates verify the
Legendre condition. To determine the best candidate we compare the
values of the functional $\mathcal{L}$ along the two good
candidates. The extremal we are looking for is given by the
candidate number five on Table~\ref{hZ::candidates:ex3}.\\
\begin{table}[!htbp]
\footnotesize \centering
\begin{tabular}{|c|c|c|c|c|c|}\hline
\# & $\tilde{y}\left(\frac{1}{4}\right)$ &
$\tilde{y}\left(\frac{1}{2}\right)$
& $\tilde{y}\left(\frac{3}{4}\right)$ & $\mathcal{L}(\tilde{y})$ & Legendre condition \eqref{hZ::eq:LC}\\
\hline
         1 & -0.5511786 &  0.0515282 &  0.5133134 &  9.3035911 &         Not verified \\
\hline
         2 &  0.2669091 &  0.4878808 &  0.7151924 &  2.0084203 &        Verified \\
\hline
         3 & -2.6745703 &  0.5599360 & -2.6730125 & 698.4443232 &         Not verified \\
\hline
         4 &  0.5789976 &  1.0701515 &  0.1840377 & 12.5174960 &         Not verified \\
\hline
         5 &  1.0306820 &  1.8920322 &  2.7429222 & -32.7189756 &        Verified \\
\hline
         6 &  0.5087946 & -0.1861431 &  0.4489196 & 10.6730959 &         Not verified \\
\hline
         7 &  4.0583690 & -1.0299054 & -5.0030989 & 2451.7637948 &         Not verified \\
\hline
         8 & -1.7436106 & -3.1898449 & -0.8850511 & 238.6120299 &         Not verified \\
\hline
  \end{tabular}
\smallskip
  \caption{There exist 8 Euler-Lagrange extremals for problem \eqref{hZ::eq:ex3}
  with $\alpha=0.8$, $\beta=0.5$, $h=0.25$, $a=0$, $b=1$, and $\theta=1$,
  but only 2 of them satisfy the fractional Legendre condition \eqref{hZ::eq:LC}.}
  \label{hZ::candidates:ex3}
\end{table}

For problem (\ref{hZ::eq:ex3}) with $\alpha=0.3$, $h=0.1$, $a=0$,
$b=0.5$, and $\theta=0$, we obtain the results of
Table~\ref{hZ::16dados}: there exist sixteen Euler-Lagrange
extremals but only one satisfy the fractional Legendre condition.
The extremal we are looking for is given by the candidate
number six on Table~\ref{hZ::16dados}.\\
\end{example}
\begin{table}[!htbp]
\footnotesize \centering
\begin{tabular}{|c|c|c|c|c|c|c|} \hline
\# & $\tilde{y}(0.1)$ & $\tilde{y}(0.2)$ & $\tilde{y}(0.3)$ &
$\tilde{y}(0.4)$ &  $\mathcal{L}(\tilde{y})$ & \eqref{hZ::eq:LC}\\
\hline

         1 & -0.305570704 & -0.428093486 & 0.223708338 & 0.480549114 & 12.25396166 &         No \\\hline

         2 & -0.427934654 & -0.599520948 & 0.313290997 & -0.661831134 & 156.2317667 &         No \\\hline

         3 & 0.284152257 & -0.227595659 & 0.318847274 & 0.531827387 & 8.669645848 &         No \\\hline

         4 & -0.277642565 & 0.222381632 & 0.386666793 & 0.555841555 & 6.993518478 &         No \\\hline

         5 & 0.387074742 & -0.310032839 & 0.434336603 & -0.482903047 & 110.7912605 &         No \\\hline

         6 & 0.259846344 & 0.364035314 & 0.463222456 & 0.597907505 & 5.104389191 &        Yes \\\hline

         7 & -0.375094681 & 0.300437245 & 0.522386246 & -0.419053781 & 93.95316858 &         No \\\hline

         8 & 0.343327771 & 0.480989769 & 0.61204299 & -0.280908953 & 69.23497954 &         No \\\hline

         9 & 0.297792192 & 0.417196073 & -0.218013689 & 0.460556635 & 14.12227593 &         No \\\hline

        10 & 0.41283304 & 0.578364133 & -0.302235104 & -0.649232892 & 157.8272685 &         No \\\hline

        11 & -0.321401682 & 0.257431098 & -0.360644857 & 0.400971272 & 19.87468886 &         No \\\hline

        12 & 0.330157414 & -0.264444122 & -0.459803086 & 0.368850105 & 24.84475504 &         No \\\hline

        13 & -0.459640837 & 0.368155651 & -0.515763025 & -0.860276767 & 224.9964788 &         No \\\hline

        14 & -0.359429958 & -0.50354835 & -0.640748011 & 0.294083676 & 34.43515839 &         No \\\hline

        15 & 0.477760586 & -0.382668914 & -0.66536683 & -0.956478654 & 263.3075289 &         No \\\hline

        16 & -0.541587541 & -0.758744525 & -0.965476394 & -1.246195157 & 392.9592508 &         No \\\hline
\end{tabular}
\smallskip
\caption{There exist 16 Euler-Lagrange extremals for problem
\eqref{hZ::eq:ex3}
  with $\alpha=0.3$, $h=0.1$, $a=0$, $b=0.5$, and $\theta=0$,
  but only 1 (candidate \#6) satisfy the fractional Legendre condition \eqref{hZ::eq:LC}.}\label{hZ::16dados}
\end{table}

\section{Conclusion}
\label{hZ::sec:conc}

In this chapter we introduce a new fractional difference variational calculus
in the time-scale $(h\mathbb{Z})_a$, $h > 0$ and $a$ a real number,
for Lagrangians depending on left and right discrete-time fractional derivatives.
Our objective was to introduce the concept of left and right fractional sum/difference
(\textrm{cf.} Definition~\ref{hZ::def0}) and to develop the theory of fractional difference calculus.
An Euler--Lagrange type equation \eqref{hZ::EL},
fractional natural boundary conditions \eqref{hZ::rui1} and \eqref{hZ::rui2},
and a second order Legendre type necessary optimality condition \eqref{hZ::eq:LC},
were obtained. The results are based on a new discrete fractional summation by parts formula
\eqref{hZ::delf:sumPart} for $(h\mathbb{Z})_a$. 
Obtained first and second order necessary optimality conditions
were implemented computationally in the computer algebra systems
\textsf{Maple} and \textsf{Maxima} (the \textsf{Maxima} 
code is found in Appendix~\ref{hZ::Maxima}).
Our numerical results show that solutions to the considered 
fractional problems become the classical discrete-time solutions
when the fractional order of the discrete-derivatives are integer
values, and that they converge to the fractional continuous-time
solutions when $h$ tends to zero. Our Legendre type condition is
useful to eliminate false candidates identified via the
Euler-Lagrange fractional equation.
The results of the chapter are formulated using standard notations
of the theory of time scales \cite{livro:2001,J:B:M:08,malina} because 
we keep in our mind the desire to generalize the present 
results to an arbitrary time scale $\mathbb{T}$.
Undoubtedly, much remains to be done in the development of the theory
of discrete fractional calculus of variations in $(h\mathbb{Z})_a$
here initiated.

\section{State of the Art}
\label{hZ::sec:state art}

The results of this chapter are published in~\cite{7} and were presented 
by the author at the FSS'09, Symposium on Fractional Signals and Systems, 
Lisbon, Portugal, November 4--6, 2009, in a contributed talk entitled 
{\it The fractional difference calculus of variations}.

Recently, Ferreira and Torres~\cite{Rui:arising} continued the development 
of the theory presented here with handy tools for the explicit solution 
of discrete equations involving left and right fractional difference operators.

\clearpage{\thispagestyle{empty}\cleardoublepage}



\chapter{Fractional Derivatives and Integrals on arbitrary $\mathbb{T}$}
\label{chap5}

In this chapter we introduce a fractional calculus on time scales using the theory
of delta dynamic equations. The basic notions of fractional order
integral and fractional order derivative on an arbitrary time scale
are proposed, using the inverse Laplace transform on time scales.
Useful properties of the new fractional operators are proved.

\section{Introduction}

Recently, two attempts have been made to provide a general
definition of fractional derivative on an arbitrary time scale
(see~\cite{emptyPaper,Anastassiou}).
These two works address a very interesting question, but
unfortunately there is a small inconsistency in the
very beginning of both studies. Indeed, 
investigations~\cite{emptyPaper,Anastassiou} are based on the following definition of
generalized polynomials on time scales $h_\alpha : \mathbb{T} \times
\mathbb{T}\rightarrow \mathbb{R}$:
\begin{equation}
\label{e:d}
\begin{gathered}
h_0(t,s) = 1 ,\\
h_{\alpha + 1}(t,s) = \int_s^t h_\alpha(\tau,s)\Delta\tau.
\end{gathered}
\end{equation}

Although the recurrence formula above address a very interesting
question about how to define fractional generalized polynomials on
time scales, recursion \eqref{e:d} provides a
definition only in the case $\alpha \in \mathbb{N}_0$,
and there is no hope to define polynomials $h_\alpha$ 
for real or complex indices $\alpha$ with~\eqref{e:d}.

Here we propose a different approach to provide a general definition
of fractional derivative on an arbitrary time scale based on the
Laplace transform~\cite{lap_Bohner}.

The chapter is organized as follows. In Section~\ref{InvLap::sec:prl}
we review the basic notions of Laplace transform on $\mathbb{R}$
(Section~\ref{Laplace transform:R}) and we introduce some necessary tools from time
scales (Section~\ref{InvLap::prl:ts}). Our results are then given in
Section~\ref{InvLap::sec:mr}: we introduce the concept of fractional
integral and fractional derivative on an arbitrary time scale $\T$
(Section~\ref{InvLap::mr:frac}); we then prove some important
properties of the fractional integrals and derivatives
(Section~\ref{InvLap::mr:prop}).


\section{Preliminaries}
\label{InvLap::sec:prl}

We start this section recalling results on classical Laplace transforms.
All necessary concepts to understand the Laplace transform definition
on time scales are supplied.


\subsection{Laplace transform on $\mathbb{R}$ as motivation}
\label{Laplace transform:R}

In order to motivate our idea we start with some 
concepts of continuous classical Laplace transform.\\

\begin{definition}\label{def:Lapl:R}
The \emph{Laplace transform}\index{Laplace transform!continuous case} 
of $f:\mathbb{R}\rightarrow \mathbb{R}$ is defined by
\begin{equation}\label{Lap:R}
\mathcal{L}[f](z)=F(z)=\int_0^{\infty} e^{-zt}f(t)dt,
\end{equation}
where $z\in\mathbb{C}$ is chosen so that the integral converges absolutely.
\end{definition}

\begin{remark}
It is well known that if
\begin{enumerate}
  \item $f$ is piecewise continuous on the interval $0\leq t \leq A$ for any positive $A$;
  \item $|f(t)|\leq M e^{at}$ when $t\geq T$, for any real constant $a$ and any positive 
  constants $M$ and $T$ (this means that $f$ is of exponential order, \textrm{i.e.}, 
  its rate of growth is not faster than that of exponential functions);
\end{enumerate}
then the Laplace transform~\eqref{Lap:R} exists for $z>a$.
\end{remark}

Because of the usefulness of the Laplace transform of derivatives in our work, 
we now refer what happens in the continuous case, using the following theorem 
that can be easily proved using the definition 
of Laplace transform and integration by parts.

\begin{theorem}
Suppose f is of exponential
order, and that f is continuous and $f'$ is piecewise continuous on any interval
$0\leq t\leq A$. Then
\begin{equation}\label{Lap:R:der:order:1}
\mathcal{L}[f']=z\mathcal{L}[f](z)-f(0)~.
\end{equation}
\end{theorem}

\begin{remark}
Applying the theorem multiple times yield
\begin{equation}\label{Lap:R:derivati}
\mathcal{L}[f^{(n)}](z)=z^n \mathcal{L}[f]-z^{n-1}f(0)-z^{n-2}f'(0)
-\ldots-z^2 f^{(n-3)}(0)-zf^{(n-2)}(0)-f^{(n-1)}(0).
\end{equation}
The above equality is an extremely useful tool of the Laplace transform
for solving linear ODEs with constant coefficients because it converts 
linear differential equations to linear algebraic equations that can be solved easily.
\end{remark}

Next proposition gives the Laplace transform of the Caputo
fractional derivative and was an inspiration for our work.

\begin{prop}\cite{book:Kilbas}
\label{InvLap::prop:Kilbas} Let $\alpha>0$, $n$ be the integer such
that $n-1<\alpha\leq n$, and $f$ a function satisfying $f \in
C^n(\R^+)$, $f^{(n)}\in L_1(0,t_1)$, $t_1>0$, and $|f^{(n)}(t)|\leq
B \mathrm{e}^{q_0 t}$, $t>t_1>0$. If the Laplace transforms
$\mathcal{L}[f](z)$ and $\mathcal{L}[f^{(n)}](z)$ exist, and
$\lim\limits_{t\rightarrow+\infty} f^{(k)}(t)=0$ for
$k=0,\ldots,n-1$, then
\begin{equation*}
\mathcal{L}\left[{}_0^C D^{\alpha}_{x}f\right](z)
=z^{\alpha}\mathcal{L}\left[f\right](z)
-\sum_{k=0}^{n-1}f^{(k)}(0)z^{\alpha-k-1}\,.
\end{equation*}
\end{prop}

\begin{remark}\label{Laplace:Caputo}
If $\alpha\in(0,1]$, then $\mathcal{L}\left[{}_0^C
D_{x}^{\alpha}f\right](z)
=z^{\alpha}\mathcal{L}\left[f\right](z)-f(0)z^{\alpha-1}$.
\end{remark}

\subsection{The Laplace transform on time scales}
\label{InvLap::prl:ts}

Looking to Definition~\ref{def:Lapl:R} we can observe that, 
in classical calculus, the exponential function assumes 
an important role in the theory of Laplace transforms.

Following the same direction, we would like to have a function on
time scales that serve the same purpose as the exponential function
does in the real case.

Before we can define such function, we need several concepts in
order for the definition to make sense. The first concept that we need is
the concept of regressivity.

\begin{definition}
A function $p:\mathbb{T}^\kappa\rightarrow\mathbb{R}$ is
\emph{regressive}\index{Regressive!function} provided
$$1+\mu(t)p(t)\neq 0$$
holds for all $t\in\mathbb{T}^\kappa$. We denote by $\mathcal{R}$
the set of all regressive and rd-continuous functions. The set of
all \emph{positively regressive} functions is defined by
$$\mathcal{R}^+=\{p\in\mathcal{R}:1+\mu(t)p(t)> 0,\ 
\mbox{for all}\ t\in\mathbb{T}^\kappa\}.$$
\end{definition}

\begin{definition}
The function $(\ominus p)(t)=-\frac{p(t)}{1+\mu(t)p(t)}$ 
for all $t\in\mathbb{T}^\kappa$ and $p\in \mathcal{R}$.
\end{definition}

Now we define the exponential function on time scales (also called, in some literature, 
the generalized exponential function) as the solution of the IVP~\eqref{exp}.

\begin{theorem}\cite[Theorem~1.37]{livro:2003}
Suppose $p\in\mathcal{R}$ and fix $t_0\in\mathbb{T}$. Then the
initial value problem
\begin{equation}\label{exp}
y^\Delta=p(t)y,\ y(t_0)=1
\end{equation}
has a unique solution on $\mathbb{T}$.
\end{theorem}
The solution of the IVP \eqref{exp} is the \emph{exponential function on
time scales}\index{Exponential function on
time scales},
and is given by $e_p(\cdot,t_0)$, where
\begin{equation} \label{expon}
e_p(t,s)=exp\left\{\int_s^t \xi_{\mu(\tau)}(p(\tau))\Delta
\tau\right\}\quad\text{with}\quad\xi_h(z)=\left\{
                                          \begin{array}{ll}
                                            \frac{Log(1+hz)}{h} & \hbox{if $h\neq 0$} \\
                                            z & \hbox{if $h=0$}
                                          \end{array}
                                       \right.
~.\end{equation}

In \eqref{expon} $\xi_h$ is the cylinder transformation defined for
$h>0$ from the set $\mathbb{C}_h:=\left\{z\in
\mathbb{C}:z\neq -\frac{1}{h}\right\}$ to the set
$\mathbb{Z}_h:=\left\{z\in \mathbb{C}:
-\frac{\pi}{h}<Im(z)\leq\frac{\pi}{h}\right\}$. The set
$\mathbb{C}_h$ is called the Hilger complex plane~\cite{livro:2001}.\\
\begin{remark}
For $h=0$, $\mathbb{C}_0:=\mathbb{C}$.
\end{remark}

\begin{definition}\cite[p.~52]{livro:2001}
Let $h>0$ and $z\in \mathbb{C}_h$. The Hilger real part of $z$ is defined by
\begin{equation}\label{Hil:Re}
Re_h(z):=\frac{|zh+1|-1}{h}
\end{equation}
and the Hilger imaginary part of $z$ by
\begin{equation}\label{Hil:Im}
Im_h(z):=\frac{Arg(zh+1)}{h},
\end{equation}
where $Arg(z)$ denotes, as usual, the principal argument of $z$.
\end{definition}

\begin{remark}
If we consider a time scale with graininess function $\mu(t)>0$
and a $z$ belonging to $\mathbb{C}_\mu:=\left\{z\in \mathbb{C}:z\neq -\frac{1}{\mu(t)}\right\}$,
we can replace $h$ by $\mu(t)$ in the right-hand sides of~\eqref{Hil:Re} and~\eqref{Hil:Im}
and represent them, respectively, by $Re_\mu(z)$ and $Im_\mu(z)$.
\end{remark}

Now we recall some properties of the exponential function stated in Theorem 2.36 of \cite{livro:2001}.
\begin{theorem}\cite{livro:2001}\label{exp:chap5}
If $p\in \mathcal{R}$, then,
\begin{enumerate}
  \item $e_0(t,s)\equiv 1$ and $e_p(t,t)\equiv 1$;
  \item $e_p(\sigma(t),s)=(1+\mu(t)p(t))e_p(t,s)$;
  \item $\frac{1}{e_p(t,s)}=e_{\ominus p}(t,s)$~.
\end{enumerate}
\end{theorem}
\begin{remark}
Item 2 of Theorem~\ref{exp:chap5} has special importance because it provides 
the ability to rewrite $e_p(\sigma(t),s)$ without the symbol $\sigma(t)$  
and on that way simplify expressions of exponential function of elementary 
functions and by consequence the Laplace transform.
\end{remark}

With all previous necessary concepts already given we are in conditions 
to understand the Laplace transform on time scales and provide 
some of their properties which are needed across this chapter.

The main results presented in this section were taken 
from \cite{lap_Bohner} and/or \cite{laplace}.

\begin{definition}\cite{lap_Bohner}\label{InvLap::def:ltts}
We define the \emph{generalized Laplace transform}\index{Laplace transform!time scales case} 
of a regulated function $f:\mathbb{T}\rightarrow \mathbb{C}$, 
where $\mathbb{T}$ denotes a time scale which 
is unbounded above and contains zero, by
\begin{equation}
\mathcal{L}_{\mathbb{T}}[f](z)=F(z):=\int_0^\infty f(t)
e_{\ominus z}(\sigma(t),0)\Delta t\quad\text{for}\quad z\in D[f],
\end{equation}
where $D[f]$ consists of all $z\in \mathbb{C}$ for which 
the improper integral exists and for which 
$1+\mu(t)z\neq 0$ for all $t\in \mathbb{T}$.
\end{definition}

\begin{remark}
In view of Definition~\ref{InvLap::def:ltts}, the Laplace transform
$\mathcal{L}$ of Proposition~\ref{InvLap::prop:Kilbas} can be
written as $\mathcal{L}_\R$.
\end{remark}


Throughout the rest of chapter, $\T$ is an arbitrary time scale with bounded
graininess, \textrm{i.e.}, $0<\mu_{min}\leq \mu(t)\leq \mu_{max}$
for all $ t \in \T$. Let $t_0 \in \T$ be fixed.

\begin{definition}\cite{laplace}
The function $f:\T\rightarrow \R$ is said to be of \emph{exponential
type I} if there exist constants $M, c>0$ such that $|f(t)|\leq M
\textrm{e}^{c t}$. Furthermore, $f$ is said to be of
\emph{exponential type II} if there exist constants $M, c>0$ such
that $|f(t)|\leq M \mathrm{e}_c(t,0)$.
\end{definition}

The time scale exponential function $\mathrm{e}_c(t,0)$ is of type
II while generalized polynomials $h_k(t,0)$ are of type I.

\begin{theorem}\cite{laplace}
If $f$ is of exponential type II with exponential constant $c$, then
the delta integral $\int_0^{\infty} f(t)
\mathrm{e}_{\ominus z}(\sigma(t),0)\Delta t$ converges absolutely for
$z\in D$.
\end{theorem}

Similarly to the classical Laplace transform 
(\textrm{cf.} formula~\eqref{Lap:R:der:order:1})
we can state the following theorem:

\begin{theorem}\cite[Theorem 1.2]{lap_Bohner}\label{Laplace:of:derivative:thm}
If $f:\mathbb{T}\rightarrow \mathbb{C}$ is such that $f^\Delta$ is regulated, then
\begin{equation}\label{Laplace:of:derivative}
\mathcal{L}_\mathbb{T}[f^\Delta](z)=z\mathcal{L}_\mathbb{T}[f](z)-f(0)
\end{equation}
for all $z\in \mathcal{D}[f]$ such that $\lim_{t\rightarrow \infty}\left(f(t)e_{\ominus z}(t)\right)=0$.
\end{theorem}

\begin{remark}
To prove Theorem \ref{Laplace:of:derivative:thm} we first make use 
of the Laplace transform, next apply integration by parts formula, 
and then use some of the properties of the exponential function 
mentioned before to rewrite the result and get what the theorem claims.
\end{remark}

Using the last result and mathematical induction one
obtains the following result:
\begin{prop}
Let $f:\T\rightarrow \R$ be such that $f^{\Delta^n}$ is regulated 
and $F$ represents the generalized Laplace transform of $f$. Then,
\begin{equation}
\label{InvLap::lap:derivative}
\mathcal{L}_{\T}[f^{\Delta^n}](z)=z^nF(z)
-\sum_{k=0}^{n-1}z^{n-k-1}f^{\Delta^{k}}(0)\,
\end{equation}
for all $z\in \mathcal{D}[f]$ such that 
$\lim_{t\rightarrow \infty}\left(f^{\Delta^k} (t)e_{\ominus z}(t)\right)=0$
$(0\leq k<n)$.
\end{prop}

\begin{proof}
The proof will be done by mathematical induction on $n$.
Before starting the induction process it's important to note 
that the fact that $f^{\Delta^n}$ is regulated implies 
that all delta derivatives with order below $n$ are also regulated.
We know from Theorem~\ref{Laplace:of:derivative:thm} that
$$\mathcal{L}_\mathbb{T}[f^\Delta](z)=z\mathcal{L}_\mathbb{T}[f](z)-f(0)\,.$$
Assuming that~\eqref{InvLap::lap:derivative} is true we have
\begin{eqnarray*}
\mathcal{L}_\mathbb{T}[f^{\Delta^{n+1}}](z)&=&z\mathcal{L}_\mathbb{T}[f^{\Delta^n}](z)-f^n(0)\\
&=&z\left(z^n F(z)-\sum_{k=0}^{n-1}z^{n-k-1}f^{\Delta^{k}}(0)\right)-f^{\Delta^n}(0)\\
&=&z^{n+1} F(z)-\sum_{k=0}^{n-1}z^{n-k}f^{\Delta^{k}}(0)-z^0 f^{\Delta^n}(0)\\
&=&z^{n+1} F(z)-\sum_{k=0}^{n}z^{n-k}f^{\Delta^{k}}(0)\,.
\end{eqnarray*}
\end{proof}

Using  Theorem~\ref{Laplace:transform:propert},
Davis \emph{et al.}~\cite{laplace}
established Theorem~\ref{laplace:inversion:time:scales} below.

\begin{theorem}\cite[Theorem 1.2]{laplace}\label{Laplace:transform:propert}
Let $F$ denote the generalized transform for $f:\mathbb{T}\rightarrow \mathbb{R}$.
Then:
\begin{enumerate}
  \item $F(z)$ is analytic in $Re_{\mu}(z)>Re_{\mu}(c)$;
  \item $F(z)$ is bounded in $Re_{\mu}(z)>Re_{\mu}(c)$;
  \item $\lim_{|z|\rightarrow \infty}F(z)=0$.
\end{enumerate}
\end{theorem}

Since our idea is to apply the Laplace transform theory on time scales, 
particularly the inverse of Laplace transform on time scales, 
to define the fractional integral (\textrm{cf.} Definition~\ref{F:I:T:S}) 
and fractional derivative (\textrm{cf.} Definition~\ref{F:D:T:S}) 
on time scales, Theorem \ref{laplace:inversion:time:scales} plays 
an important role in our work by providing sufficient conditions 
for the existence of such inverse and a formula to evaluate it.

\begin{theorem}[Inversion formula of the Laplace transform \cite{laplace}]
\label{laplace:inversion:time:scales}
Suppose that $F$ is analytic in the region $R \mathrm{e}_{\mu}(z)>R
\mathrm{e}_{\mu}(c)$ and $F(z)\rightarrow 0$ uniformly as
$|z|\rightarrow\infty$ in this region. Assume $F$ has finitely many
regressive poles of finite order $\{z_1,z_2,\ldots, z_n\}$ and
$\tilde{F}_{\R}(z)$ is the transform of the function $\tilde{f}(t)$
on $\R$ that corresponds to the transform $F(z)=F_{\T}(z)$ of $f(t)$
on $\T$. If
\[
\int_{c-i\infty}^{c+i\infty}\left|\tilde{F}_{\R}(z)\right||dz|<\infty\,,
\]
then
\begin{equation*}
f(t)=\sum_{i=1}^n\mbox{Res}_{z=z_i} \mathrm{e}_z(t,0)F(z)
\end{equation*}
has transform $F(z)$ for all $z$ with $Re(z)>c$.
\end{theorem}
\begin{example}\cite[Example 1.1]{laplace}\label{example:inverse}
Suppose $F(z)=\frac{1}{z^2}$. Then
$$\mathcal{L}_{\mathbb{T}}^{-1}[F](z)=f(t)=Res_{z=0}\frac{e_z(t,0)}{z^2}=t.$$\\
Applying Theorem~\ref{laplace:inversion:time:scales} for $F(z)=\frac{1}{z^3}$ we get
$$\mathcal{L}_{\mathbb{T}}^{-1}[F](z)=f(t)=h_2(t,0)\,.$$
\end{example}
\begin{remark}
Using induction arguments over the results of Example~\ref{example:inverse} 
it's easy to prove that the inverse transform of $F(z)=\frac{1}{z^{k+1}}$, 
$k\in\N$, is $h_k(\cdot,0)$.
\end{remark}

\begin{remark}
The inversion formula also gives the claimed
inverses for any of the elementary functions that were presented in
the table of Laplace transform in \cite{lap_Bohner}.
\end{remark}
The following lemma establish a relation between the solution of~\eqref{exp} with power series on time scales.

\begin{lemma}\cite[Lemma 4.4]{Bohnerlaplace:convol}
\label{Laplace:IVP}
For all $z\in\mathbb{C}$ and $t\in \mathbb{T}$ with $t\geq \alpha$, 
the initial value problem
\begin{equation}
y^\Delta=zy,\quad\quad y(\alpha)=1
\end{equation}
has a unique solution $y$ that is represented in the form
\begin{equation}\label{IVP:powerseries}
y(t)=\sum_{k=0}^{+\infty}z^k h_k(t,\alpha)
\end{equation}
and satisfies the inequality
\begin{equation}
\left|y(t)\right|\leq e^{|z|(t-\alpha)}\,.
\end{equation}
\end{lemma}
\begin{remark}
Let $\alpha=0$. If we use jointly Lemma~\ref{Laplace:IVP} and the well-known fact
that the unique solution of IVP \eqref{exp} is $e_z(t,0)$ we have the equality
$$e_z(t,0)=\sum_{k=0}^{+\infty}z^kh_k(t,0)$$
that allows us to see that generalized polynomials $h_k(t,0)$ is the reason 
to have different inverse Laplace images when we have different time scales.
\end{remark}


\section{Main Results}
\label{InvLap::sec:mr}

We begin by introducing the definition of fractional integral and
fractional derivative on an arbitrary time scale $\T$.


\subsection{Fractional derivative and integral on time scales}
\label{InvLap::mr:frac}

Simultaneously, the generalized Laplace transform on time scales
gives unification and extension of the classical results. Important
to us, the Laplace transform of the $\Delta$-derivative is given by
the formula $\mathcal{L}_{\T}[f^{\Delta}](z)=zF(z)-f(0)$. Our idea
is to define the fractional derivative on time scales via the
inverse Laplace transform formula for the complex function
$G(z)=z^{\alpha}\mathcal{L}_{\T}[f](z)-f(0^+)z^{\alpha-1}$.
Furthermore, for $\alpha\in(n-1,n]$, $n\in\N$, we use a
generalization of \eqref{InvLap::lap:derivative} to define
fractional derivatives on times scales for higher orders $\alpha$.

\begin{definition}[Fractional integral on time scales]
\label{F:I:T:S}\index{Fractional integral on time scales}
Let $\alpha>0$, $\T$ be a time scale, and $f : \mathbb{T}
\rightarrow \mathbb{R}$. The fractional integral of $f$ of order
$\alpha$ on the time scale $\T$, denoted by $I_{\T}^{\alpha}f$, is
defined by
\[
I_{\T}^{\alpha}f(t)
=\mathcal{L}_{\T}^{-1}\left[\frac{F(z)}{z^{\alpha}}\right](t)\,.
\]
\end{definition}

\begin{definition}[Fractional derivative on time scales]
\label{F:D:T:S}\index{Fractional derivative on time scales} 
Let $\T$ be a time scale,
$F(z)=\mathcal{L}_{\T}[f](z)$, and $\alpha\in (n-1,n]$, $n\in\N$.
The fractional derivative of function $f$ of order $\alpha$ on the
time scale $\T$, denoted by $f^{(\alpha)}$, is defined by
\begin{equation}
\label{InvLap::frac:def:n}
f^{(\alpha)}(t)=\mathcal{L}^{-1}_{\T}\left[z^{\alpha}F(z)
-\sum_{k=0}^{n-1}f^{\Delta^k}(0^+)z^{\alpha-k-1}\right](t)\,.
\end{equation}
\end{definition}

\begin{remark}
For $\alpha\in (0,1]$ we have
\begin{equation*}
f^{(\alpha)}(t)=\mathcal{L}^{-1}_{\T}\left[z^{\alpha}F(z)
-f(0^+)z^{\alpha-1}\right](t)\,.
\end{equation*}
Moreover, if we use $\mathbb{T}=\mathbb{R}$ we have
\begin{equation*}
f^{(\alpha)}(t)=\mathcal{L}^{-1}_{\R}\left[z^{\alpha}F(z)
-f(0^+)z^{\alpha-1}\right](t)=\mathcal{L}^{-1}_{\R}\left[
\mathcal{L}[{}_{0^+}^C D^{\alpha}_{x}f](z)\right](t)
=\left({}_0^C D^{\alpha}_{x}f\right)(t)\,.
\end{equation*}
\end{remark}


\subsection{Properties}
\label{InvLap::mr:prop}

We begin with two trivial but important remarks about the fractional
integral and the fractional derivative operators just introduced.

\begin{remark}
As the inverse Laplace transform is linear, we also have linearity
for the new fractional integral and derivative:
\begin{equation*}
\begin{split}
I_{\T}^{\alpha}(af+bg)(t)
&=aI_{\T}^{\alpha}f(t)+bI_{\T}^{\alpha}g(t),\\
(af+bg)^{(\alpha)}(t) &=a f^{(\alpha)}(t) + b g^{(\alpha)}(t)\,.
\end{split}
\end{equation*}
\end{remark}

\begin{remark}
Since $h_0(t) \equiv 1$ for any time scale $\T$, from the definition
of Laplace transform and the fractional derivative we conclude that
$h_0^{(\alpha)}(t)=0$. For $\alpha>1$ one has $h_1^{(\alpha)}(t)=0$.
\end{remark}

We now prove several important properties of the fractional
integrals and fractional derivatives on arbitrary time scales.

\begin{prop}
Let $\alpha\in(n-1,n]$, $n \in \N$. If $k\leq n-1$, then
\[
h_k^{(\alpha)}(t,0)=0\,.
\]
\end{prop}

\begin{proof}
From \eqref{InvLap::frac:def:n} it follows that
\begin{equation*}
\begin{split}
h_k^{(\alpha)}(t,0)&=\mathcal{L}^{-1}_{\T}\left[\frac{z^{\alpha}}{z^{k+1}}
-\sum_{i=0}^{n-1}h_k^{\Delta^i}(0)z^{\alpha-i-1}\right](t)
=\mathcal{L}^{-1}_{\T}\left[\frac{z^{\alpha}}{z^{k+1}}
-\sum_{i=0}^{k}h_{k-i}(0)z^{\alpha-i-1}\right](t)\\
&=\mathcal{L}^{-1}_{\T}\left[\frac{z^{\alpha}}{z^{k+1}}-z^{\alpha-k-1}\right](t)=0\,.
\end{split}
\end{equation*}
\end{proof}

\begin{prop}
Let $\alpha\in(n-1,n]$, $n \in \N$. If $k\geq n$, then
\[
h_k^{(\alpha)}(t,0)
=\mathcal{L}^{-1}_{\T}\left[\frac{1}{z^{k+1-\alpha}}\right](t) \,.
\]
\end{prop}

\begin{proof}
From \eqref{InvLap::frac:def:n} we have
\begin{equation*}
\begin{split}
h_k^{(\alpha)}(t)&=\mathcal{L}^{-1}_{\T}\left[\frac{z^{\alpha}}{z^{k+1}}
-\sum_{i=0}^{n-1}h_k^{\Delta^i}(0)z^{\alpha-i-1}\right](t)
=\mathcal{L}^{-1}_{\T}\left[\frac{z^{\alpha}}{z^{k+1}}
-\sum_{i=0}^{n-1}h_{k-i}(0)z^{\alpha-i-1}\right](t)\\
&=\mathcal{L}^{-1}_{\T}\left[\frac{z^{\alpha}}{z^{k+1}}\right](t)
=\mathcal{L}^{-1}_{\T}\left[\frac{1}{z^{k+1-\alpha}}\right](t)\,.
\end{split}
\end{equation*}
\end{proof}

\begin{prop}
Let $\alpha\in(n-1,n]$, $n \in \N$. If $c(t) \equiv m$, $m\in
\mathbb{R}$, then
\[
c^{(\alpha)}(t)=0.
\]
\end{prop}

\begin{proof}
From the linearity of the inverse Laplace transform and the fact
that $$h^{(\alpha)}_0(t,0)=0 $$it follows that
$$c^{(\alpha)}(t)=(m\cdot 1)^{(\alpha)}=(m\cdot
\overline{}h_0(t,0))^{(\alpha)}=m\cdot h^{(\alpha)}_0(t,0)=m\cdot 0 = 0\,.$$
\end{proof}

\begin{prop}
\label{InvLap::prop:comp_integrals} 
Let $\alpha,\beta>0$. Then,
\begin{equation*}
I_{\T}^{\beta}\left(I_{\T}^{\alpha}f\right)(t)
=I_{\T}^{\alpha+\beta}f(t)\,.
\end{equation*}
\end{prop}

\begin{proof}
By definition we have
\[
I_{\T}^{\beta}\left(I_{\T}^{\alpha}f\right)(t)
=\mathcal{L}_{\T}^{-1}\left[z^{-\beta}\mathcal{L}_{\T}\left[I_{\T}^{\alpha}f\right]\right](t)
=\mathcal{L}_{\T}^{-1}\left[s^{-\alpha-\beta}F(z)\right](t)
=I_{\T}^{\alpha+\beta}f(t)\,.
\]
\end{proof}

\begin{prop}
\label{InvLap::prop:comp} Let $\alpha,\beta\in(0,1]$.
\begin{itemize}
\item[(a)] If $\alpha+\beta\leq 1$, then
\[
\left(f^{(\alpha)}\right)^{(\beta)}(t)
=f^{(\alpha+\beta)}(t)-\mathcal{L}^{-1}_{\T}\left[z^{\beta-1}f^{(\alpha)}(0)\right](t)\,.
\]
\item[(b)] If $1<\alpha+\beta\leq 2$, then
\[
\left(f^{(\alpha)}\right)^{(\beta)}(t)=f^{(\alpha+\beta)}(t)
-\mathcal{L}^{-1}_{\T}\left[z^{\beta-1}f^{(\alpha)}(0)\right](t)+
\mathcal{L}^{-1}_{\T}\left[z^{\alpha+\beta-2}f^{\Delta}(0)\right](t)\,.
\]
\item[(c)] If $\beta\in(0,1]$, then
\[
\left(f^{\Delta}\right)^{(\beta)}(t)=f^{(\beta+1)}(t)\,.
\]
\end{itemize}
\end{prop}

\begin{proof}
(a) Let $\alpha+\beta\leq 1$. Then
\begin{equation*}
\begin{split}
\left(f^{(\alpha)}\right)^{(\beta)}(t)
&=\mathcal{L}^{-1}_{\T}\left[z^{\beta}\mathcal{L}\left[f^{(\alpha)}\right](z)
-z^{\beta-1}f^{(\alpha)}(0)\right](t)\\
&=\mathcal{L}^{-1}_{\T}\left[z^{\alpha+\beta}F(z)-z^{\alpha+\beta-1}f(0)\right](t)
- \mathcal{L}^{-1}_{\T}\left[z^{\beta-1}f^{(\alpha)}(0)\right](t)\\
&=f^{(\alpha+\beta)}(t)
-\mathcal{L}^{-1}_{\T}\left[z^{\beta-1}f^{(\alpha)}(0)\right](t)\,.
\end{split}
\end{equation*}
(b) For $1<\alpha+\beta\leq 2$ we have
\[f^{(\alpha+\beta)}(t)=
\mathcal{L}^{-1}_{\T}\left[z^{\alpha+\beta}\mathcal{L}\left[f\right](z)
-z^{\alpha+\beta-1}f(0)-z^{\alpha+\beta-2}f^{\Delta}(0)\right](t)\,.\]
Hence,
\[\left(f^{(\alpha)}\right)^{(\beta)}(t)=f^{(\alpha+\beta)}(t)
+\mathcal{L}^{-1}_{\T}\left[z^{\alpha+\beta-2}f^{\Delta}(0)\right](t)
-\mathcal{L}^{-1}_{\T}\left[z^{\beta-1}f^{(\alpha)}(0)\right](t)\,.
\] (c) The intended relation follows from (b) by choosing
$\alpha=1$.
\end{proof}

In general $\left(f^{(\alpha)}\right)^{(\beta)}(t)
\neq\left(f^{(\beta)}\right)^{(\alpha)}(t)$. However, the equality
holds in particular cases:

\begin{prop}
If $\alpha+\beta\leq 1$ and $f(0)=0$, then
$\left(f^{(\alpha)}\right)^{(\beta)}(t)
=\left(f^{(\beta)}\right)^{(\alpha)}(t)$.
\end{prop}

\begin{proof}
It follows from item (a) of Proposition~\ref{InvLap::prop:comp}.
\end{proof}

\begin{prop}
Let $\alpha\in (n-1,n]$, $n \in \N$, and $\lim\limits_{t\rightarrow
0^+}f^{\Delta^k}(t)=f^{\Delta^k}(0^+)$ exist, $k=0,\ldots,n-1$. The
following equality holds:
\begin{equation*}
I^{\alpha}_{\T}\left(f^{(\alpha)}\right)(t)
=f(t)-\sum_{k=0}^{n-1}f^{\Delta^k}(0^+)h_k(t)\,.
\end{equation*}
\end{prop}

\begin{proof}
Let $F(z)=\mathcal{L}_{\T}\left[f\right](z)$. Then,
\begin{equation*}
\begin{split}
I^{\alpha}_{\T}\left(f^{(\alpha)}\right)(t)
&=\mathcal{L}_{\T}^{-1}\left[z^{-\alpha}\mathcal{L}_{\T}\left[f^{(\alpha)}\right](z)\right](t)
=\mathcal{L}_{\T}^{-1}\left[F(z)-\sum_{k=0}^{n-1}f^{\Delta^k}(0^+)\frac{z^{\alpha-k-1}}{z^{\alpha}}\right](t)\\
&=f(t)-\sum_{k=0}^{n-1}f^{\Delta^k}(0^+)\mathcal{L}_{\T}^{-1}\left[\frac{1}{z^{k+1}}\right](t)
=f(t)-\sum_{k=0}^{n-1}f^{\Delta^k}(0^+)h_k(t)\,.
\end{split}
\end{equation*}
\end{proof}

\begin{prop}
Let $\alpha\in (n-1,n]$, $n \in \N$, and
$F(z)=\mathcal{L}_{\T}\left[f\right](z)$. If
$\lim\limits_{z\rightarrow\infty}F(z)=0$ and
$\lim\limits_{z\rightarrow\infty}\frac{F(z)}{z^{\alpha-k}}=0$, $k
\in \{1,\ldots,n\}$, then
$\left(I^\alpha_\mathbb{T}f\right)^{(\alpha)}(t)=f(t)$.
\end{prop}

\begin{proof}
Firstly let us notice that $\lim\limits_{t\rightarrow
0^+}\left(I_{\T}^{\alpha}f\right)^{\Delta^k}(t)=0$ for
$k\in\{0,\ldots, n-1\}$. For that we check that
$\lim\limits_{z\rightarrow
\infty}z\mathcal{L}\left[I^{\alpha}_{\T}f\right](z)
=\lim\limits_{z\rightarrow \infty}z\frac{F(z)}{z^{\alpha}}=0$.
Nextly let us assume that it holds for $i=0,\ldots,k-1$ and let us
calculate
\begin{equation*}
\begin{split}
\lim\limits_{z\rightarrow
\infty}z\mathcal{L}\left[\left(I^{\alpha}_{\T}
f\right)^{\Delta^k}\right](z) &=\lim\limits_{z\rightarrow
\infty}z\left(z^k\frac{F(z)}{z^{\alpha}}
-\sum_{i=0}^{k-1}z^{k-1-i}\left(I_{\T}^{\alpha}\right)^{\Delta^i}(0)\right)\\
&= \lim\limits_{z\rightarrow
\infty}z\left(z^k\frac{F(z)}{z^{\alpha}}\right)
=\lim\limits_{z\rightarrow \infty}\frac{F(z)}{z^{\alpha-k-1}}=0 \,.
\end{split}
\end{equation*}
Then we easily conclude that
\[
\left(I^\alpha_\mathbb{T}f\right)^{(\alpha)}(t)
=\mathcal{L}_{\T}^{-1}\left[z^{\alpha}\frac{F(z)}{z^{\alpha}}\right](t)=f(t)\,.
\]
\end{proof}

The convolution of two functions $f:\T\rightarrow\R$ and $g:\T\times
\T\rightarrow\R$ on time scales, where $g$ is rd-continuous with
respect to the first variable, is defined in \cite{BL,laplace}:
\begin{equation*}\index{Convolution of functions on time scales}
\left(f\ast g \right)(t)=\int_0^t f(\tau)g(t,\sigma(\tau))\Delta
\tau\,.
\end{equation*}
As function $g$ we can consider, \textrm{e.g.},
$\mathrm{e}_c(t,t_0)$ or $h_k(t,t_0)$.

\begin{prop}
Let $t_0\in \T$. If $\alpha\in(0,1)$, then
\begin{equation*}
\left(f\ast g(\cdot,t_0)\right)^{(\alpha)}(t)
=\left(f^{(\frac{\alpha}{2})}\ast
g^{(\frac{\alpha}{2})}(\cdot,t_0)\right)(t) =\left(f^{(\alpha)}\ast
g(\cdot,t_0)\right)(t)\,,
\end{equation*}
where we assume the existence of the involved derivatives.
\end{prop}

\begin{proof}
From the convolution theorem for the generalized Laplace transform
\cite[Theorem 2.1]{laplace},
\[
\mathcal{L}_{\T}\left[\left(f\ast
g(\cdot,t_0)\right)^{(\alpha)}\right](z) =z^{\alpha}F(z)G(z)\,.
\]
Hence,
\begin{equation*}
\begin{split}
\left(f\ast g(\cdot,t_0)\right)^{(\alpha)}(t)
&=\mathcal{L}^{-1}_{\T}\left[z^{\alpha}F(z)G(z)\right](t)
=\mathcal{L}^{-1}_{\T}\left[z^{\frac{\alpha}{2}}F(z)\right](t)
\mathcal{L}^{-1}_{\T}\left[z^{\frac{\alpha}{2}}G(z)\right](t)\\
&=\left(f^{(\alpha/2)}\ast g^{(\alpha/2)}(\cdot,t_0)\right)(t)\,.
\end{split}
\end{equation*}
Equivalently, we can write
\[
\mathcal{L}^{-1}_{\T}\left[z^{\alpha}F(z)G(z)\right](t)
=\mathcal{L}^{-1}_{\T}\left[z^{\alpha}F(z)\right](t)
\mathcal{L}^{-1}_{\T}\left[G(z)\right](t)=\left(f^{(\alpha)}\ast
g(\cdot,t_0)\right)(t)\,.
\]
\end{proof}

\section{State of the Art}

The results of this chapter are already published in~\cite{BastosMozTorres} 
and were presented by the author in a seminar with the title 
{\it Fractional Derivatives and Integrals on Time Scales 
via the Inverse Generalized Laplace Transform} on June 9, 2011,
at the Faculty of Computer Science of Bia\l ystok University of Technology, Bia\l ystok, Poland.


\clearpage{\thispagestyle{empty}\cleardoublepage}

\chapter{Conclusions and future work}\label{conclusao}

In this thesis we define fractional sum and difference operators for
the time scales $\mathbb{T}=\mathbb{Z}$ and $\mathbb{T}=(h\mathbb{Z})_a$
and prove some properties for them. We also established and proved first and
second order necessary optimality conditions for problems
of the calculus of variations on both time scales.
They show that the solutions of the fractional problems coincide
with the solutions of the corresponding non-fractional variational
problems when the order of the discrete derivatives is an integer
value. We believe that the present work will potentiate further research in
the fractional calculus of variations where much remains to be done.\\
With respect to the fractional derivatives and integrals on an arbitrary
time scale $\mathbb{T}$, we introduce a fractional calculus on time scales
using the theory of delta dynamic equations
and the inverse Laplace transform on time scales.
The basic notions of fractional order integral
and fractional order derivative on an arbitrary time scale are given,
and some properties of the new introduced fractional operators are proved.

\medskip

Some possible directions of research are:
\begin{itemize}
\item to extend the results on fractional variational problems 
for higher-order problems of the calculus of variations
\cite{B:J:05,RD,NataliaHigherOrderNabla,Artur:Martins:Torres:2012} 
with fractional discrete derivatives of any order;
\item to explore new possibilities of research for fractional continuous 
variational problems, in particular, to get a fractional continuous Legendre 
necessary optimality condition (an interesting open question);
\item to explore several areas of application, for example, in signal processing, 
where fractional derivatives of a discrete-time signal are particularly useful
to describe noise processes \cite{Ortig:B};
\item to find a general formulation leading to the specification of ARMA parameters 
(as pointed out by one of the referees of~\cite{comNuno:Rui:Z});
\item to study optimality conditions for more general variable
endpoint variational problems \cite{Zeidan,AD:10b,MyID:169} and
isoperimetric problems \cite{Almeida1,AlmeidaNabla};
\item to obtain fractional sufficient optimality
conditions of Jacobi type and a version of Noether's theorem
\cite{Bartos,Cresson:Frederico:Torres,gastao:delfim,MyID:149}
for discrete-time fractional variational problems;
\item to explore direct methods of optimization for absolute extrema
\cite{Bohner:F:T,mal:tor,T:L:08};
\item to generalize our first and second order
optimality conditions to a fractional Lagrangian
possessing delay terms \cite{M:J:B:09,B:M:J:08};
\item to extend our results of discrete fractional variational problems 
to problems involving Caputo difference operators~\cite{Anastassiou2} 
or even Riemann-Liouvile and Caputo together on the same Lagrangian.
\end{itemize}

Another line of research, where we just have few things done~\cite{comNuno}, 
is related with the new discrete fractional diamond sum (\textrm{cf.} 
Definition~\ref{diamond::diamond} below) in the following topics:

\begin{itemize}
\item to investigate the discrete fractional diamond difference 
introduced in~\cite{comNuno} and generalize the new
discrete diamond fractional operator
to an arbitrary time scale $\mathbb{T}$;
\item to investigate
the usefulness of modeling with fractional equations~\cite{sofia}
and study corresponding fractional diamond variational principles;
\item to extend the results presented in Chapters~\ref{chap3} 
and~\ref{chap4} to an arbitrary time scale;
\item to study and develop results on other subjects of
mathematics using the theory of discrete fractional calculus.
\end{itemize}

All the results in Chapters~\ref{chap3},~\ref{chap4} and \ref{chap5}
were obtained for delta-calculus but could easily be rewritten using
the operator nabla. However, in our opinion, doesn't make sense to have
results for $\Delta$ and for $\nabla$ separately.
Motivated by the diamond-alpha dynamic derivative on time scales \cite{14,11,13}
and the fractional derivative of \cite{withBasia:Spain2010,Malinowska:Torres:2011:Caputo},
we introduce a new operator that unify and extend the operators $\Delta$ 
and $\nabla$ by using a convex linear combination of them.
To exemplify what we propose, we introduce a more general
fractional sum operator (for $\mathbb{T}=\mathbb{Z}$), making use of the symbol,
${_{\gamma}}\diamondsuit$ (\textrm{cf.}
Definition~\ref{diamond::diamond}).
With this new operator it's not our intention to change what was done in 
$\Delta$ and $\nabla$ cases -- as particular cases, results on delta and 
nabla discrete fractional calculus are obtained from the new operator -- 
but just to unify the two cases in one.

Looking to the literature of discrete fractional difference
operators, two approaches are found (see, \textrm{e.g.},
\cite{4,comNuno:Rui:Z}): one using the $\Delta$ point of view
(sometimes called the forward fractional difference approach),
another using the $\nabla$ perspective (sometimes called the
backward fractional difference approach).
When $\gamma = 1$ our ${_{\gamma}}\diamondsuit$ operator is reduced
to the $\Delta$ one; when $\gamma = 0$ the ${_{\gamma}}\diamondsuit$
operator coincides with the corresponding $\nabla$ fractional sum.\\\\
Analogously to Definition~\ref{Miller:fract}, one considers the
discrete nabla fractional sum operator:

\begin{definition}\cite{4}
\label{diamond::nabla} The discrete nabla fractional sum operator is
defined by
\begin{equation*}
(\nabla_a^{-\beta}f)(t)
=\frac{1}{\Gamma(\beta)}\sum_{s=a}^{t}(t-\rho(s))^{\overline{\beta-1}}f(s),
\end{equation*}
where $\beta>0$. Here $f$ is defined for $s=a \mod (1)$ and
$\nabla_a^{-\beta}f$ is defined for $t=a \mod (1)$.
\end{definition}

\begin{remark}
Let $\mathbb{N}_a=\{a,a+1,a+2,\ldots\}$. The operator
$\nabla_a^{-\beta}$ maps functions defined on $\mathbb{N}_a$ to
functions defined on $\mathbb{N}_{a}$. The fact that $f$ and
$\nabla_a^{-\beta} f$ have the same domain, while $f$ and
$\Delta_a^{-\alpha} f$ do not, explains why some authors prefer the
nabla approach.
\end{remark}

The next result gives a relation between the delta fractional sum
and the nabla fractional sum operators.

\begin{lemma}\cite{4}
\label{diamond::nabladelta} Let $0\leq m-1<\nu\leq m$, where $m$
denotes an integer. Let $a$ be a positive integer, and $y(t)$ be
defined on $t \in \mathbb{N}_a=\{a,a+1,a+2,\ldots\}$. The following
statement holds: $\left(\Delta_a^{-\nu}
y\right)(t+\nu)=\left(\nabla_a^{-\nu} y\right)(t)$, $t\in
\mathbb{N}_{a}$.
\end{lemma}

With the help of this lemma we are able to rewrite our results of Chapter~\ref{chap3} 
to the $\nabla$ operator. Our intention is not to have results for $\Delta$ and $\nabla$ 
separately but, instead, to extend and unify them through the following operator:


\begin{definition}
\label{diamond::diamond} The diamond-$\gamma$ fractional operator of
order $(\alpha,\beta)$ is given, when applied to a function $f$ at
point $t$, by
\begin{equation*}
\left(_{\gamma}\diamondsuit_a^{-\alpha,-\beta}f\right)(t) =\gamma
\left(\Delta_a^{-\alpha}f\right)(t+\alpha) +(1-\gamma)
\left(\nabla_a^{-\beta}f\right)(t),
\end{equation*}
where $\alpha>0$, $\beta>0$, and $\gamma\in[0,1]$. Here, both $f$
and $_{\gamma}\diamondsuit_a^{-\alpha,-\beta}f$ are defined for $t=a
\mod (1)$.
\end{definition}

\begin{remark}
Similarly to the nabla fractional operator, our operator
$_{\gamma}\diamondsuit_a^{-\alpha,-\beta}$ maps functions defined on
$\mathbb{N}_a$ to functions defined on $\mathbb{N}_a$,
$\mathbb{N}_a=\{a,a+1,a+2,\ldots\}$ for $a$ a given real number.
\end{remark}

\begin{remark}
The new diamond fractional operator of
Definition~\ref{diamond::diamond} gives, as particular cases, the
operator of Definition~\ref{Miller:fract} for $\gamma=1$,
$$
\left(_{1}\diamondsuit_a^{-\alpha,-\beta}f\right)(t)
=\left(\Delta_a^{-\alpha}f\right)(t+\alpha), \quad t\equiv a \mod
(1),
$$
and the operator of Definition~\ref{diamond::nabla} for $\gamma=0$,
$$
\left(_{0}\diamondsuit_a^{-\alpha,-\beta}f\right)(t)
=\left(\nabla_a^{-\beta}f\right)(t), \quad t\equiv a \mod (1).
$$
\end{remark}

The next theorems give important properties of the new, more
general, discrete fractional operator
$_{\gamma}\diamondsuit_a^{-\alpha,-\beta}$.

\begin{theorem}
\label{diamond::thm:SUM} Let $f$ and $g$ be real functions defined
on $\mathbb{N}_a$, $\mathbb{N}_a=\{a,a+1,a+2,\ldots\}$ for $a$ a
given real number. The following equality holds:
\begin{equation*}
\left(_{\gamma}\diamondsuit_a^{-\alpha,-\beta} (f+g)\right)(t)
=\left({_{\gamma}\diamondsuit}_a^{-\alpha,-\beta}f\right)(t)
+\left({_{\gamma}\diamondsuit}_a^{-\alpha,-\beta}g\right)(t).
\end{equation*}
\end{theorem}

\begin{proof}
The intended equality follows from the definition of
diamond-$\gamma$ fractional sum of order $(\alpha,\beta)$:
\begin{equation*}
\begin{split}
(_{\gamma}\diamondsuit_a^{-\alpha,-\beta}(f+g))(t)
&=\gamma(\Delta_a^{-\alpha}(f+g))(t+\alpha)+(1-\gamma)(\nabla_a^{-\beta}(f+g))(t)\\
&=\frac{\gamma}{\Gamma(\alpha)}\sum_{s=a}^{t}(t+\alpha-\sigma(s))^{(\alpha-1)}(f(s)
+g(s))\\
&\qquad+\frac{1-\gamma}{\Gamma(\beta)}\sum_{s=a}^{t}(t-\rho(s))^{\overline{\beta-1}}(f(s)+g(s))\\
&=\frac{\gamma}{\Gamma(\alpha)}\sum_{s=a}^{t}(t+\alpha-\sigma(s))^{(\alpha-1)}f(s)
+\frac{\gamma}{\Gamma(\alpha)}\sum_{s=a}^{t}(t+\alpha-\sigma(s))^{(\alpha-1)}g(s)\\
&\qquad
+\frac{1-\gamma}{\Gamma(\beta)}\sum_{s=a}^{t}(t-\rho(s))^{\overline{\beta-1}}f(s)
+\frac{1-\gamma}{\Gamma(\beta)}\sum_{s=a}^{t}(t-\rho(s))^{\overline{\beta-1}}g(s)\\
&=
\left[\frac{\gamma}{\Gamma(\alpha)}\sum_{s=a}^{t}(t+\alpha-\sigma(s))^{(\alpha-1)}f(s)
+\frac{1-\gamma}{\Gamma(\beta)}\sum_{s=a}^{t}(t-\rho(s))^{\overline{\beta-1}}f(s)\right]\\
&\quad +
\left[\frac{\gamma}{\Gamma(\alpha)}\sum_{s=a}^{t}(t+\alpha-\sigma(s))^{(\alpha-1)}g(s)
+\frac{1-\gamma}{\Gamma(\beta)}\sum_{s=a}^{t}(t-\rho(s))^{\overline{\beta-1}}g(s)\right]\\
&=(_{\gamma}\diamondsuit_a^{-\alpha,-\beta}f)(t) +
(_{\gamma}\diamondsuit_a^{-\alpha,-\beta}g)(t).
\end{split}
\end{equation*}
\end{proof}

\begin{theorem}
\label{diamond::thm:const} Let $f(t)=k$ on $\mathbb{N}_a$, $k$ a
constant. The following equality holds:
$$
(_{\gamma}\diamondsuit_a^{-\alpha,-\beta}f)(t)
=\gamma\frac{\Gamma(t-a+1+\alpha) k}{\Gamma(\alpha+1)\Gamma(t-a+1)}
+(1-\gamma)\frac{\Gamma(t-a+1+\beta)
k}{\Gamma(\beta+1)\Gamma(t-a+1)}\, .
$$
\end{theorem}

\begin{proof}
By definition of diamond-$\gamma$ fractional sum of order
$(\alpha,\beta)$, we have
\[\begin{split}
(_{\gamma}\diamondsuit_a^{-\alpha,-\beta}k)(t)
&=\gamma(\Delta_a^{-\alpha}k)(t+\alpha)+(1-\gamma)(\nabla_a^{-\beta}k)(t)\\
&=\frac{\gamma}{\Gamma(\alpha)}\sum_{s=0}^t
k(t+\alpha-\sigma(s))^{(\alpha-1)}
+\frac{1-\gamma}{\Gamma(\beta)}\sum_{s=0}^t
k(t-\rho(s))^{\overline{\beta-1}}\\
&=\gamma\frac{\Gamma(t-a+1+\alpha)}{\alpha\Gamma(\alpha)\Gamma(t-a+1)}k+
(1-\gamma)\frac{\Gamma(t-a+1+\beta)}{\beta\Gamma(\beta)\Gamma(t-a+1)}k\\
&=\gamma\frac{\Gamma(t-a+1+\alpha)}{\Gamma(\alpha+1)\Gamma(t-a+1)}k+
(1-\gamma)\frac{\Gamma(t-a+1+\beta)}{\Gamma(\beta+1)\Gamma(t-a+1)}k.
\end{split}
\]
\end{proof}

\begin{cor}
\label{diamond::m:f:cor} Let $f(t) \equiv k$ for a certain constant
$k$. Then,
\begin{equation}
\label{diamond::Miller_Constant} (\Delta_a^{-\alpha}f)(t+\alpha)
=\frac{\Gamma(t-a+1+\alpha)}{\Gamma(\alpha+1)\Gamma(t-a+1)}k.
\end{equation}
\end{cor}

\begin{proof}
The result follows from Theorem~\ref{diamond::thm:const} choosing
$\gamma=1$ and recalling that
$$(_{1}\diamondsuit_a^{-\alpha,-\beta}k)(t) =
(\Delta_a^{-\alpha}k)(t+\alpha)\,.$$
\end{proof}

\begin{remark}
In the particular case when $a=0$, equality
\eqref{diamond::Miller_Constant} coincides with the result of
\cite[Section~5]{Miller}.
\end{remark}

The fractional nabla result analogous to
Corollary~\ref{diamond::m:f:cor} is easily obtained:

\begin{cor}
If $k$ is a constant, then
\begin{equation*}
(\nabla_a^{-\beta}k)(t)
=\frac{\Gamma(t-a+1+\beta)}{\Gamma(\beta+1)\Gamma(t-a+1)}k.
\end{equation*}
\end{cor}

\begin{proof}
The result follows from Theorem~\ref{diamond::thm:const} choosing
$\gamma=0$ and recalling that
$$(_{0}\diamondsuit_a^{-\alpha,-\beta}k)(t)
=(\nabla_a^{-\beta}k)(t)\,.$$
\end{proof}

\begin{theorem}
\label{diamond::thm:10} Let $f$ be a real valued function and
$\alpha_1$, $\alpha_2$, $\beta_1$, $\beta_2>0$. Then,
\[\begin{split}
\left(_{\gamma}\diamondsuit_a^{-\alpha_1,
-\beta_1}\left(_{\gamma}\diamondsuit_a^{-\alpha_2,-\beta_2}f\right)\right)(t)
&=\gamma\left({_{\gamma}}\diamondsuit_a^{-(\alpha_1+\alpha_2),
-(\beta_1+\alpha_2)} f\right)(t)\\
&\qquad+(1-\gamma)\left(
{_{\gamma}}\diamondsuit_a^{-(\alpha_1+\beta_2),-(\beta_1+\beta_2)}
f\right)(t).
\end{split}
\]
\end{theorem}

\begin{proof}
Direct calculations show the intended relation:
\begin{equation*}
\begin{split}
(_{\gamma}&\diamondsuit_a^{-\alpha_1,-\beta_1}(_{\gamma}\diamondsuit_a^{-\alpha_2,-\beta_2}f))(t)
=\gamma(\Delta_a^{-\alpha_1}(_{\gamma}\diamondsuit_a^{-\alpha_2,-\beta_2}f))(t+\alpha_1)
\\&\quad+ (1-\gamma)(\nabla_a^{-\beta_1}(_{\gamma}\diamondsuit_a^{-\alpha_2,-\beta_2}f))(t)\\
&=\quad\gamma^2
(\Delta_a^{-\alpha_1}\left(\Delta_a^{-\alpha_2}f\right))(t+\alpha_1+\alpha_2)
+
\gamma(1-\gamma)\left(\Delta_a^{-\alpha_1}\left(\nabla_a^{-\beta_2}f\right)\right)(t+\alpha_1)
\\&\quad+ (1-\gamma)\gamma\left(\nabla_a^{-\beta_1}\left(\Delta_a^{-\alpha_2}f\right)\right)(t+\alpha_2)
+  (1-\gamma)^2\left(\nabla_a^{-\beta_1}\left(\nabla_a^{-\beta_2}f\right)\right)(t)\\
&=\gamma^2\left(\Delta_a^{-(\alpha_1+\alpha_2)}f\right)(t+\alpha_1+\alpha_2)
+ \gamma(1-\gamma)\left(\Delta_a^{-\alpha_1}\left(\Delta_a^{-\beta_2}f\right)\right)(t+\alpha_1+\beta_2)\\
&\quad +
(1-\gamma)\gamma\left(\nabla_a^{-\beta_1}\left(\nabla_a^{-\alpha_2}f\right)\right)(t)
+(1-\gamma)^2 \left(\nabla_a^{-(\beta_1+\beta_2)}f\right)(t)\\
&=\gamma^2\left(\Delta_a^{-(\alpha_1+\alpha_2)}f\right)(t+\alpha_1+\alpha_2)
+ \gamma(1-\gamma)\left(\Delta_a^{-(\alpha_1+\beta_2)}f\right)(t+\alpha_1+\beta_2)\\
&\quad +
(1-\gamma)\gamma\left(\nabla_a^{-(\beta_1+\alpha_2)}f\right)(t)
+(1-\gamma)^2\left(\nabla_a^{-(\beta_1+\beta_2)}f\right)(t)\\
&=\gamma\left[\gamma\left(\Delta_a^{-(\alpha_1+\alpha_2)}f\right)(t+\alpha_1+\alpha_2)
+ (1-\gamma)\left(\nabla_a^{-(\beta_1+\alpha_2)}f\right)(t)\right]\\
&\quad +
(1-\gamma)\left[\gamma\left(\Delta_a^{-(\alpha_1+\beta_2)}f\right)(t+\alpha_1+\beta_2)
+ (1-\gamma)\left(\nabla_a^{-(\beta_1+\beta_2)}f\right)(t)\right].
\end{split}
\end{equation*}
\end{proof}

\begin{remark}
If $\gamma=0$, then
$\left(_{\gamma}\diamondsuit_a^{-\alpha_1,
-\beta_1}(_{\gamma}\diamondsuit_a^{-\alpha_2,-\beta_2}f)\right)(t)
=\left(\nabla_a^{-(\beta_1+\beta_2)}f\right)(t)$.
\end{remark}

\begin{remark}
If $\gamma=1$, then
$\left(_{\gamma}\diamondsuit_a^{-\alpha_1,
-\beta_1}(_{\gamma}\diamondsuit_a^{-\alpha_2,-\beta_2}f)\right)(t)
=\left(\Delta_a^{-(\alpha_1+\alpha_2)}f\right)(t+\alpha_1+\alpha_2)$.
\end{remark}

\begin{remark}
If $\alpha_1=\alpha_2=\alpha$ and $\beta_1=\beta_2=\beta$, then
$$\left(_{\gamma}\diamondsuit_a^{-\alpha,
-\beta}(_{\gamma}\diamondsuit_a^{-\alpha,-\beta}f)\right)(t)
= \left({_{\gamma}}\diamondsuit_a^{-\alpha,-\beta}f\right)(t)\,.$$
\end{remark}

We now prove a general Leibniz formula.

\begin{theorem}[Leibniz formula]
\label{diamond::thm:ProductRule} Let $f$ and $g$ be real valued
functions, $0<\alpha,~\beta<1$. For all $t$ such that $t=a \mod
(1)$, the following equality holds:
\begin{multline}
\label{diamond::eq:GLF}
\left(_{\gamma}\diamondsuit_a^{-\alpha,-\beta}(fg)\right)(t)
=\gamma\sum_{k=0}^\infty\binom{-\alpha}{k}\left[\left(\nabla^k
g\right)(t)\right] \cdot\left[\left(\Delta_a^{-(\alpha+k)}f\right)(t+\alpha+k)\right]\\
+(1-\gamma)\sum_{k=0}^\infty\binom{-\beta}{k}\left[\left(\nabla^k
g\right)(t)\right]\left[\left(\Delta_a^{-(\beta+k)}f\right)(t+\beta
+ k)\right],
\end{multline}
where
$$
\binom{u}{v}=\frac{\Gamma(u+1)}{\Gamma(v+1)\Gamma(u-v+1)}.
$$
\end{theorem}

\begin{proof}
By definition of the diamond fractional sum,
\begin{equation*}
\begin{split}
\left(_{\gamma}\diamondsuit_a^{-\alpha,-\beta}(fg)\right)(t)
&=\gamma \left(\Delta_a^{-\alpha}(fg)\right)(t+\alpha)
+(1-\gamma)\left(\nabla_a^{-\beta}(fg)\right)(t)\\
&=\frac{\gamma}{\Gamma(\alpha)}\sum_{s=a}^{t}(t+\alpha-\sigma(s))^{(\alpha-1)}f(s)g(s)\\
&\quad+\frac{1-\gamma}{\Gamma(\beta)}\sum_{s=a}^{t}(t-\rho(s))^{\overline{\beta-1}}f(s)g(s).
\end{split}
\end{equation*}
By Taylor's expansion of $g(s)$ \cite{1},
$$
g(s)=\sum_{k=0}^\infty \frac{(s-t)^{\overline{k}}}{k!} (\nabla^k
g)(t)=\sum_{k=0}^\infty (-1)^k\frac{(t-s)^{(k)}}{k!}(\nabla^k g)(t).
$$
Substituting the Taylor series of $g(s)$ at $t$,
\begin{multline*}
\left(_{\gamma}\diamondsuit_a^{-\alpha,-\beta}(fg)\right)(t)
=\frac{\gamma}{\Gamma(\alpha)}\sum_{s=a}^{t}(t+\alpha-\sigma(s))^{(\alpha-1)}f(s)
\left[\sum_{k=0}^\infty (-1)^k(t-s)^{(k)}\frac{(\nabla^k g)(t)}{k!}\right]\\
+\frac{1-\gamma}{\Gamma(\beta)}\sum_{s=a}^{t}(t-\rho(s))^{\overline{\beta-1}}f(s)
\left[\sum_{k=0}^\infty (-1)^k(t-s)^{(k)}\frac{(\nabla^k
g)(t)}{k!}\right].
\end{multline*}
Since
\begin{equation*}
\begin{split}
(t+\alpha-\sigma(s))^{(\alpha-1)}(t-s)^{(k)}&=(t+\alpha-\sigma(s))^{(\alpha+k+1)},\\
(t-\rho(s))^{\overline{\beta-1}}(t-s)^{(k)}&=(t+\beta-\sigma(s))^{(\beta+k+1)},
\end{split}
\end{equation*}
and $\displaystyle\sum_{s=t-k+1}^{t}(t-s)^{(k)}=0$, we have
\begin{multline*}
\left(_{\gamma}\diamondsuit_a^{-\alpha,-\beta}(fg)\right)(t)
=\frac{\gamma}{\Gamma(\alpha)}\sum_{k=0}^\infty
(-1)^k\frac{(\nabla^k g)(t)}{k!} \sum_{s=a}^{t-k}(t+\alpha-\sigma(s))^{(\alpha+k-1)}f(s)\\
+\frac{1-\gamma}{\Gamma(\beta)}\sum_{k=0}^\infty
(-1)^k\frac{(\nabla^k g)(t)}{k!}
\sum_{s=a}^{t-k}(t+\beta-\sigma(s))^{(\beta+k-1)}f(s).
\end{multline*}
Because
$$
(-1)^k=\frac{\Gamma(-\alpha+1)\Gamma(\alpha)}{\Gamma(-\alpha+k+1)\Gamma(k+\alpha)}
=\frac{\Gamma(-\beta+1)\Gamma(\beta)}{\Gamma(-\beta+k+1)\Gamma(k+\beta)}
$$
and $k!=\Gamma(k+1)$, the above expression becomes
\begin{equation*}
\begin{split}
\left(_{\gamma}\diamondsuit_a^{-\alpha,-\beta}(fg)\right)(t)
&=\frac{\gamma}{\Gamma(\alpha)} \sum_{k=0}^\infty (\nabla^k g)(t)
\binom{-\alpha}{k}\cdot\left[\frac{1}{\Gamma(k+\alpha)}
\sum_{s=a}^{t-k}(t+\alpha-\sigma(s))^{(\alpha+k-1)}f(s)\right]\\
&\quad+\frac{1-\gamma}{\Gamma(\beta)} \sum_{k=0}^\infty (\nabla^k
g)(t) \binom{-\beta}{k}\left[\frac{1}{\Gamma(k+\beta)}
\sum_{s=a}^{t-k}(t+\beta-\sigma(s))^{(\beta+k-1)}f(s)\right]\\
&=\gamma\sum_{k=0}^\infty \binom{-\alpha}{k}(\nabla^k g)(t)
(\Delta_a^{-(\alpha+k)}f)(t+\alpha+k)\\
&\quad+(1-\gamma)\sum_{k=0}^\infty \binom{-\beta}{k}(\nabla^k
g)(t)(\Delta_a^{-(\beta+k)}f)(t+\beta+k).
\end{split}
\end{equation*}
\end{proof}

\begin{remark}
Choosing $\gamma=0$ in our Leibniz formula \eqref{diamond::eq:GLF},
we obtain that
$$(\nabla_a^{-\beta}(fg))(t)
=\sum_{k=0}^\infty\binom{-\beta}{k}\left[(\nabla^k
g)(t)\right]\left[(\Delta_a^{-(\beta+k)}f)(t+\beta+k)\right].
$$
\end{remark}

\begin{remark}
Choosing $\gamma=1$ in our Leibniz formula \eqref{diamond::eq:GLF},
we obtain that
\begin{equation}
\label{diamond::LeibnizDelta} (\Delta_a^{-\alpha}(fg))(t+\alpha)
=\sum_{k=0}^\infty\binom{-\alpha}{k}\left[(\nabla^k g)(t)\right]
\left[(\Delta_a^{-(\alpha+k)}f)(t+\alpha+k)\right].
\end{equation}
As a particular case of \eqref{diamond::LeibnizDelta}, let $a=0$.
Then, recalling Lemma~\ref{diamond::nabladelta}, we obtain the
Leibniz formulas of \cite{5}.
\\
\end{remark}

The preliminary results presented here, about this new fractional operator,
are published in~\cite{comNuno} and were presented by the author at the NSC'10,
3rd Conference on Nonlinear Science and Complexity, Ankara, Turkey, July 28--31, 2010.


\clearpage{\thispagestyle{empty}\cleardoublepage}

\appendix

\chapter{Maxima code used in Chapter 3}\label{Z::Maxima}

The following \textsf{Maxima} code implements Theorem~\ref{Z::thm0}.
Examples illustrating the use of our procedure \texttt{extremal} below are
found in Section~\ref{Z::sec2}.

\small
\begin{verbatim}

kill(all)$ ratprint:false$ simpsum:true$ tlimswitch:true$

sigma(t):=t+1$

rho(t):=t-1$

rho2(t):=rho(rho(t))$

Delta(exp,t):=block( define(f12(t),exp),
return((f12(sigma(t))-f12(t))) )$

p(x,y):=(gamma(x+1))/(gamma(x+1-y))$

SumL(a,t,nu,exp):=block( define(f1(x),exp),
f1(t)+nu/gamma(nu+1)*sum((p(t+nu-sigma(r),nu-1))*f1(r),r,(a),((t-1)))
)$

SumR(t1,b,nu1,exp1):=block( define(f2(x),exp1),
f2(t1)+nu1/gamma(nu1+1)*sum((p(s+nu1-sigma(t1),nu1-1))*f2(s),s,(t1+1),b)
)$

DeltaL(a2,t2,alpha2,exp2):=block(
[alpha1:ratsimp(alpha2),a:ratsimp(a2),t:ratsimp(t2)],
define(f3(x),exp2), define(q(x),SumL(a,x,1-alpha1,f3(x))), q0:q(o),
o4:float(ev(ratsimp(Delta(q0,o)),nouns)), o41:subst(o=t,o4),
remfunction(f), remfunction(q), return(o41) )$

DeltaR(t3,b,alpha3,exp3):=block(
[alpha1:ratsimp(alpha3),b:ratsimp(b),t:ratsimp(t3)],
define(f4(x),exp3), define(q1(o),(SumR(x,b,1-alpha1,f4(x)))),
q10:q1(z),
o5:float(ev(ratsimp(-Delta((SumR(x,b,1-alpha1,f4(x))),x)),nouns)),
o51:subst(x=t,o5), remfunction(f), remfunction(q1), return(o51) )$

EL(exp7,a,b,alpha7,beta7):=block(
[a:ratsimp(a),b:ratsimp(b),alpha:ratsimp(alpha7),beta:ratsimp(beta7)],
define(LL(t,u,v,w),exp7), b1:diff(LL(t,u,v,w),u),
sa:subst([u=y(sigma(t)),v=DeltaL(a,t,alpha,y(o)),
                        w=DeltaR(t,b,beta,y(o))],b1),
b2:diff(LL(t,u,v,w),v),
sb:subst([t=x,u=y(sigma(x)),v=DeltaL(a,x,alpha,y(x)),
                            w=DeltaR(x,b,beta,y(x))],b2),
sb1:DeltaR(o,rho(b),alpha,sb), sb11:subst(o=t,sb1),
b3:diff(LL(t,u,v,w),w),
sc:subst([t=x,u=y(sigma(x)),v=DeltaL(a,x,alpha,y(x)),
                            w=DeltaR(x,b,beta,y(x))],b3),
sc2:DeltaL(a,p2,beta,sc), sc22:subst(p2=t,sc2), return(sa+sb11+sc22)
)$

ELt(exp8,a,b,alpha8,beta8,t8):=
ratsimp(subst(t=t8,EL(exp8,a,b,alpha8,beta8)))$

extremal(L,a,b,A9,B9,alpha9,beta9):=block(
[a:ratsimp(a),b:ratsimp(b),alpha:ratsimp(alpha9),
beta:ratsimp(beta9),A1:ratsimp(A9), B1:ratsimp(B9)],
eqs:makelist(ratsimp(ELt(L,a,b,alpha,beta,a+i)),i,0,ratsimp((rho2(b)-a))),
vars:makelist(y(ratsimp(a+i)),i,1,ratsimp((rho(b)-a))),
Xi:[a],Xf:[b],Yi:[A1],Yf:[B1],
X:makelist(ratsimp(i),i,1,ratsimp((rho(b)-a))), X:append(Xi,X,Xf),
sols:algsys(subst([y(a)=A1,y(b)=B1],eqs),vars),
Y:makelist(rhs(sols[1][i]),i,1,ratsimp((rho(b)-a))),
Y:append(Yi,Y,Yf),
return(makegamma(ratsimp(minfactorial(makefact(sols[1])))) )$
\end{verbatim}

\normalsize

\clearpage{\thispagestyle{empty}\cleardoublepage}

\chapter{Maxima code used in Chapter 4}
\label{hZ::Maxima}

The following \textsf{Maxima} code implements 
Theorem~\ref{hZ::thm0} and Theorem~\ref{hZ::thm1}.

\small

\begin{verbatim}
kill(all)$remfunction(all)$remvalue(all)$
fpprec:100$load("draw")$
load("eval_string")$
ratprint:false$
simpsum:true$
simpproduct:true$
tlimswitch:true$

sigma(t,h):=t+h$

rho(t,h):=t-h$

rho2(t,h):=rho(rho(t,h),h)$

Delta(exp,t,h):=block(
define(f12(t),exp),
return((f12(sigma(t,h))-f12(t))/h)
)$

p(x,y,h):=h^y*((((gamma(x/h+1)))/((gamma(x/h+1-y)))))$

intHZ(f,var,a9,b9,h9):=block(
[a:ratsimp(a9),b:ratsimp(b9),h:ratsimp(h9)],
p1:subst(var=k*h,f),
res:sum(p1*h, k, a/h, rho(b,h)/h),
return(res)
)$

SumL(a,t,nu,exp9,h0):=block(
define(f1(x),exp9),
li:ratsimp(a/h0),
ls:ratsimp((t-h0)/h0),
l1:float((nu/gamma(nu+1))),
(h^nu)*f1(t)+l1*sum((p(t+nu*h0-sigma(r*h0,h0),nu-1,h0))*f1(r*h0),r,li,ls)*h0)$

SumR(t,b,nu1,exp1,h1):=block(
define(f2(x),exp1),
l1:((nu1/gamma(nu1+1))),
li:ratsimp(sigma(t,h1)/h1),
ls:ratsimp(b/h1),
(h1^nu1)*f2(t)+l1*sum((p(c*h1+nu1*h1-sigma(t,h1),nu1
-1,h1))*f2(c*h1),c,li,ls)*h1)$

DeltaL(a2,t2,alpha2,exp2,h):=block(
[alpha1:ratsimp(alpha2),a:ratsimp(a2),t:ratsimp(t2)],
define(f3(x),exp2),
define(q(x),SumL(a,x,1-alpha1,f3(x),h)),
q0:q(o),
o4:(ev((Delta(q0,o,h)),nouns)),
o41:subst(o=t,o4),
remfunction(f3),
remfunction(q),
return(o41)
)$

DeltaR(t3,b,alpha3,exp3,h3):=block(
[alpha1:ratsimp(alpha3),h:ratsimp(h3),b:ratsimp(b),t:ratsimp(t3)],
define(f4(x),exp3),
define(q1(x),(SumR(x,b,1-alpha1,f4(x),h))),
q01:q1(o),
o5:-(ev((Delta(q01,o,h)),nouns)),
o51:subst(o=t,o5),
remfunction(f4),
remfunction(q1),
return(o51))$

EL(exp7,a,b,alpha7,beta7,h7):=block(
[a:ratsimp(a),b:ratsimp(b),alpha:ratsimp(alpha7),beta:ratsimp(beta7),
                                                h:ratsimp(h7)],
define(LL(t,u,v,w),exp7),
b1:diff(LL(t,u,v,w),u),
sa:subst([u=y(sigma(t,h)),v=DeltaL(a,t,alpha,y(t),h),
                          w=DeltaR(t,b,beta,y(t),h)],b1),
b2:diff(LL(t,u,v,w),v),
sb:subst([t=x,u=y(sigma(t,h)),v=DeltaL(a,x,alpha,y(x),h),
                              w=DeltaR(x,b,beta,y(x),h)],b2),
sb1:DeltaR(o,rho(b,h),alpha,sb,h),
sb11:subst(o=t,sb1),
b3:diff(LL(t,u,v,w),w),
sc:subst([u=y(sigma(t,h)),v=DeltaL(a,x,alpha,y(x),h),
                          w=DeltaR(x,b,beta,y(x),h)],b3),
sc2:DeltaL(a,p2,beta,sc,h),
sc22:subst(p2=t,sc2),
return(sa+sb11+sc22)
)$

ELt(exp8,a,b,alpha8,beta8,t8,h):=
ratsimp(subst(t=t8,EL(exp8,a,b,alpha8,beta8,h)))$

extremal2(L,a,b,A9,B9,alpha9,beta9,h9):=block(
[a:ratsimp(a),b:ratsimp(b),h:ratsimp(h9),alpha:ratsimp(alpha9),
beta:ratsimp(beta9),A1:ratsimp(A9),
B1:ratsimp(B9)],
eqs:makelist(
(ELt(L,a,b,alpha,beta,a+i*h,h)),i,0,ratsimp((rho2(b,h)-a)/h)
),
vars:makelist(y(ratsimp(a+i*h)),i,1,ratsimp((rho(b,h)-a)/h)),
Xi:[a],Xf:[b],Yi:[A1],Yf:[B1],
X:makelist(ratsimp(i*h),i,1,ratsimp((rho(b,h)-a)/h)),
X:append(Xi,X,Xf),
eqs01:subst([y(a)=A1,y(b)=B1],eqs),
eqs1:((eqs01)),
return(eqs1)
)$

extremal(L,a,b,A9,B9,alpha9,beta9,h9):=block(
[a:ratsimp(a),b:ratsimp(b),h:ratsimp(h9),alpha:ratsimp(alpha9),
beta:ratsimp(beta9),A1:ratsimp(A9),
B1:ratsimp(B9)],
eqs:makelist((ELt(L,a,b,alpha,beta,a+i*h,h)),i,0,ratsimp((rho2(b,h)-a)/h)),
vars:makelist(y(ratsimp(a+i*h)),i,1,ratsimp((rho(b,h)-a)/h)),
Xi:[a],Xf:[b],Yi:[A1],Yf:[B1],
X:makelist(ratsimp(i*h),i,1,ratsimp((rho(b,h)-a)/h)),
X:append(Xi,X,Xf),
eqs01:subst([y(a)=A1,y(b)=B1],eqs),
eqs1:((eqs01)),
sols:algsys(eqs1,vars),
if length(sols)=1 then
(Y:makelist(rhs(sols[1][i]),i,1,ratsimp((rho(b,h)-a)/h)),
Y:append(Yi,Y,Yf)
),
disp(float(sols)),
return(makegamma(ratsimp(minfactorial(makefact(sols)))))
)$

LegEqH(L,a,b,alpha,beta,h):=block(
mu:1-alpha,
nu:1-beta,
duu:diff(L,u,2),
duv:diff(L,u,1,v,1),
duw:diff(L,u,1,w,1),
dvv:diff(L,v,2),
dvw:diff(L,v,1,w,1),
dww:diff(L,w,2),
part1:(h^2)*subst([u=y(sigma(t,h)),v=DeltaL(a,t,alpha,y(x),h),
                                   w=DeltaR(t,b,beta,y(x),h)],duu),
part2:2*h^(mu+1)*subst([u=y(sigma(t,h)),v=DeltaL(a,t,alpha,y(x),h),
                                        w=DeltaR(t,b,beta,y(x),h)],duv),
part3:(h^mu)*subst([u=y(sigma(t,h)),v=DeltaL(a,t,alpha,y(x),h),
                                    w=DeltaR(t,b,beta,y(x),h)],dvv),
part4:subst(t=sigma(t,h),subst([u=y(sigma(t,h)),v=DeltaL(a,t,alpha,y(x),h),
                           w=DeltaR(t,b,beta,y(x),h)],dvv))*(mu*h^mu-h^mu)^2,
part5:intHZ(h^3*subst(t=s,subst([u=y(sigma(t,h)),v=DeltaL(a,t,alpha,y(x),h),
      w=DeltaR(t,b,beta,y(x),h)],dvv))*((1/gamma(mu+1))*mu*(mu-1)
         *p(s+mu*h-sigma(sigma(t,h),h),mu-2,h))^2,s,sigma(sigma(t,h),h),b,h),
part6:2*h^(nu+1)*subst([u=y(sigma(t,h)),v=DeltaL(a,t,alpha,y(x),h),
                           w=DeltaR(t,b,beta,y(x),h)],duw)*(nu-1),
part7:2*h^(mu+nu)*subst([u=y(sigma(t,h)),v=DeltaL(a,t,alpha,y(x),h),
                           w=DeltaR(t,b,beta,y(x),h)],dvw)*(nu-1),
part8:2*h^(mu+nu)*subst(t=sigma(t,h),subst([u=y(sigma(t,h)),
   v=DeltaL(a,t,alpha,y(x),h),w=DeltaR(t,b,beta,y(x),h)],dvw))*(mu-1),
part9:h^(2*nu)*subst([u=y(sigma(t,h)),
   v=DeltaL(a,t,alpha,y(x),h),w=DeltaR(t,b,beta,y(x),h)],dww)*(1-nu)^2,
part10:h^(2*nu)*subst(t=sigma(t,h),subst([u=y(sigma(t,h)),
   v=DeltaL(a,t,alpha,y(x),h),w=DeltaR(t,b,beta,y(x),h)],dww)),
part11:intHZ(h^3*subst(t=s,subst([u=y(sigma(t,h)),v=DeltaL(a,t,alpha,y(x),h),
             w=DeltaR(t,b,beta,y(x),h)],dww))*((1/gamma(nu+1))*nu*(1-nu)
                                *p(t+nu*h-sigma(s,h),nu-2,h))^2,s,a,t,h),
return(part1+part2+part3+part4+part5+part6+part7+part8+part9+part10+part11)
)$

LegEqMinH(L,a,b,A,B,alpha,beta,h):=block(
vec:[],
vec:append(vec,[subst([t=a,y(a)=A],ratsimp(LegEqH(L,a,b,alpha,beta,h)))]),
vectTemp:[a],
for i:1 thru length(X)-3 do(
t1:ratsimp(a+i*h),
vectTemp:append(vectTemp,[t1]),
res:subst([t=ratsimp(t1)],ratsimp(LegEqH(L,a,b,alpha,beta,h))),
vec:append(vec,[res])
),
M:matrix(vectTemp,vec),
return(M)
)$

LeqEqMinAllH(L,a,b,A,B,alpha,beta,sols,h):=block(
a:ratsimp(a),b:ratsimp(b),
Validos:[],
for i:1 thru length(sols) do(
M:LegEqMinH(L,a,b,A,B,alpha,beta,h),
M:subst([y(a)=A,y(b)=B],M),
M:subst(makelist(sols[i][k], k, 1, length(sols[i])),M),
Sinal:[],
c:0,
for l:1 thru matrix_size(M)[2] do (
Sinal:append(Sinal,[is(M[2][l]>=0)]),
if is(M[2][l]>=0)=true then(c:c+1)
),
if c=matrix_size(M)[2] then (
Validos:append(Validos,[i])
)
),
disp(Validos)
)
$

TESTE(L,a,b,A,B,alpha,beta,h,sols):=block(
a:ratsimp(a),b:ratsimp(b),
indices:[],
valor:[],
for i:1 thru length(sols) do(
f00:subst(
[u=y(sigma(t,h)),v=DeltaL(a,t,alpha,y(t),h),w=DeltaR(t,b,beta,y(t),h)],L
),
r00:ratsimp(intHZ(f00,t,a,b,h)),
r01:subst(sols[i],r00),
r02:subst([y(a)=A,y(b)=B],r01),
indices:append(indices,[i]),
valor:append(valor,[r02])
),
Mat:matrix(indices,valor),
return(Mat)
)$
\end{verbatim}

\normalsize

\clearpage{\thispagestyle{empty}\cleardoublepage}



\clearpage{\thispagestyle{empty}\cleardoublepage}

\printindex

\clearpage{\thispagestyle{empty}\cleardoublepage}


\end{document}